\documentclass[reqno,10pt]{amsart}
\usepackage{amssymb}
\usepackage{amsmath}
\usepackage{amsthm}
\usepackage{graphicx}
\usepackage{cite}
\usepackage[colorlinks,linkcolor=blue,anchorcolor=red,citecolor=blue]{hyperref}
\usepackage{a4wide}
\usepackage{marginnote}

\usepackage{color}

\newtheorem{thm}{Theorem}[section]
\newtheorem{cor}[thm]{Corollary}
\newtheorem{lem}[thm]{Lemma}
\newtheorem{prop}[thm]{Proposition}
\newtheorem{rem}{Remark}[section]

\numberwithin{equation}{section}

\newcommand{\R}{\mathbb{R}}

\newcommand{\CE}{\mathcal{E}}

\newtheorem{defn}{Definition}[section]

\newcommand{\vertiii}[1]{{\left\vert\kern-0.25ex\left\vert\kern-0.25ex\left\vert #1
		\right\vert\kern-0.25ex\right\vert\kern-0.25ex\right\vert}}

\begin{document}
\title[Time-periodic problem on the Boltzmann equation]
{Three-dimensional time-periodic problem on the Boltzmann equation with external force}

\author[R.-J. ~Duan]{Renjun Duan}
\address[R.-J. ~Duan]{Department  of Mathematics, The Chinese University of Hong Kong,  Shatin, Hong Kong, P.R.~China}
\email{rjduan@math.cuhk.edu.hk}

\author[J.~Ni]{Jinkai Ni}   
\address[J.~Ni]{School of Mathematics, Nanjing University, Nanjing
 210093, P.R.~China}
\email{jinkaini123@gmail.com}

\begin{abstract}
The time-periodic problem for the Boltzmann equation with a prescribed time-periodic external force in the three-dimensional whole space has remained open since it was first treated in \cite{DUYZ-CMP-2008}, where the result was restricted to spatial dimensions not less than five. This paper gives an affirmative answer when the force is sufficiently small in $\mathcal{C}(\mathbb{R};\dot{B}^{-3/2}_{2,\infty}\cap\dot{H}^N)$ with $N\geq 4$. The proof follows Serrin's method, through the global stability of the Cauchy problem for a time-periodic force. As a direct consequence, the same argument yields the existence and stability of stationary solutions in three dimensions for a time-independent force, which is not required to be a gradient and may therefore be rotational.
\end{abstract}

\date{\today}
	

\maketitle
\thispagestyle{empty}

\setcounter{equation}{0}
\indent 
\allowdisplaybreaks

\section{Introduction}

The Boltzmann equation for the hard-sphere gas under the influence of an external force in the three-dimensional full space is expressed as
\begin{align}\label{G1.1} 
\partial_t F + v\cdot \nabla_{x} F+E\cdot\nabla_{v}F = Q(F,F), 
\end{align}
where the unknown non-negative function $F = F(t,x,v)$ represents the velocity distribution function of gas particles at time $t \in \mathbb{R}$, position $x \in \mathbb{R}^3$, and velocity $v \in \mathbb{R}^3$. The given external force field is denoted as $E = E(t, x)\in \R^3$ independent of velocity $v$.  The bilinear collision operator $Q(\cdot,\cdot)$ for the hard-sphere model takes the form:
\begin{align*}
Q(F,G)=\int_{\mathbb{R}^3}\!\int_{\mathbb{S}^2}|(v - v_{*})\cdot\omega| (F^\prime G^\prime_{*}-FG_{*}) \,{\rm d}\omega{\rm d}v_{*},
\end{align*}
where
\begin{align*}
F=F(t,x,v),\quad F^\prime=F (t,x,v^\prime),\quad G_{*}=G(t,x,v_{*}),\quad G_{*}^\prime=G(t,x,v_{*}^\prime),     
\end{align*}
with
\begin{align*}
v^\prime= v-[(v-v_{*})\cdot\omega]\omega,\quad v_{*}^\prime=v_{*}+[(v-v_*)\cdot\omega]\omega, \quad \omega\in \mathbb{S}^2.  
\end{align*}

The Boltzmann equation stands as a cornerstone of kinetic theory, describing the nonequilibrium statistical behavior of dilute gases at the mesoscopic level, cf.~\cite{CIP-1994, Glassey 1996, Villani 2002}; see also a recent breakthrough Deng-Hani-Ma \cite{DHM} for the rigorous derivation of Boltzmann equation from the hard sphere system over an arbitrarily long time. A fundamental question that has persisted in the mathematical analysis of this equation is whether time-periodic external forces---ubiquitous in physical applications ranging from alternating electromagnetic fields to oscillatory mechanical driving---can sustain time-periodic distribution functions in the physically relevant three-dimensional space. Namely, we assume that the external force is time-periodic of period $T > 0$, i.e., $E(t,x)=E(t + T,x)$ for any $t\in\mathbb R$ and $x\in\mathbb R^3$, which is also applicable to the stationary case if $E=E(x)$ is further independent of time. Under such assumption, our final goal of this paper is to consider the following time-periodic problem:
 \begin{equation}\label{G1.2}
\left\{\begin{aligned}
& \partial_t F_{T}+ v\cdot \nabla_{x} F_{T}+E\cdot\nabla_{v}F_{T}=Q(F_{T},F_{T}) \quad \text{in}\quad \mathbb R\times \mathbb R_{x}^3\times \mathbb R_{v}^3,\\
&F_{T}(t,x,v)=F_{T}(t+T,x,v)\quad \text{for}\quad (t,x,v)\in \mathbb R\times \mathbb R_x^3\times \mathbb R_v^3, 
 \end{aligned}
 \right.
\end{equation}   
where we include the subscript $T$ to emphasize the dependence of $F$ on the period $T$. This question addresses the very capability of the Boltzmann equation to model periodically forced nonequilibrium states. While the first author of this paper with Ukai, Yang and Zhao \cite{DUYZ-CMP-2008} solved the problem in higher spatial dimensions ($d\geq 5$) where stronger dissipative effects simplify the analysis, the three-dimensional problem---which is precisely the dimension of our physical world---has remained unresolved so far. The present work closes this gap by establishing, for the first time, the existence and stability of time-periodic solutions to the hard-sphere Boltzmann equation driven by a small time-periodic external force in three spatial dimensions. This result validates the physical relevance of the periodic forcing mechanism. Moreover, as a direct corollary, our analysis yields the existence and stability of stationary solutions when the external force becomes time-independent, offering a unified framework for understanding both steady and periodically forced nonequilibrium states. 

We search for solutions of equation \eqref{G1.1} around the normalized global Maxwellian $M$, 
which is defined as 
\begin{align*}M = M(v)=(2\pi)^{-\frac{3}{2}}e^{-\frac{|v|^2}{2}}.   \end{align*}
Subsequently, we define the standard perturbation $f(t,x,v)$ with respect to $M$ as 
\begin{align*}F = M+\sqrt{M}f,    \end{align*}
from which the Boltzmann equation \eqref{G1.1} for the perturbation $f$ can be reformulated as follows: \begin{align}\label{G1.3}
\partial_t f + v\cdot\nabla_x f + \mathcal{L}f = \Gamma (f,f)  -E\cdot \nabla_{v} f+\frac{1}{2}E\cdot v f+E\cdot v\sqrt{M} , \end{align}
for $(t,x,v)\in \mathbb R\times \mathbb R^3_{x}\times \mathbb R^3_{v}$.
In \eqref{G1.3}, the linearized collision operator $\mathcal{L}$ and the nonlinear collision operator $\Gamma$ are respectively given by
\begin{align*}
\mathcal{L}g=-\frac{1}{\sqrt{M}}\big[ Q(M,\sqrt{M}g)+Q(\sqrt{M}g,M)  \big]   ,  
\end{align*}
and
\begin{align*}
\Gamma(g_1,g_2)=\frac{1}{\sqrt{M}}Q(\sqrt{M }g_1,\sqrt{M}g_2). 
\end{align*}
Recall (cf.~\cite{CIP-1994}) that the linearized collision operator $\mathcal L$ is non-negative definite on $L^2(\mathbb{R}_v^3)$, and its null space is five-dimensional, defined by
\begin{align*}
\mathcal{N}=\ker\mathcal{L}=\mathrm{span}\big\{1, v_1,v_2,v_3,|v|^2\big\}\sqrt{M}.
\end{align*}
For any fixed $(t,x)$, we define $\mathbf{P}$ as the orthogonal projection from $L^2(\mathbb{R}_v^3)$ onto $\mathcal{N}$. 
For later use, for any function $f(t,x,v)$, we introduce the macro-micro decomposition as 
\begin{align}\label{G1.4}
f = \mathbf{P}f+\{\mathbf{I}-\mathbf{P}\}f,
\end{align}
where $\mathbf{P}f$ is called the macro (fluid) part of $f$, and $\{\mathbf{I}-\mathbf{P}\}f$ is the micro (non-fluid) part.

There is an extensive literature on the global existence and asymptotic behavior theory of Boltzmann's Cauchy problem. For the Boltzmann equation without external forces, DiPerna and Lions \cite{DL-AM-1989} first established the global existence of renormalized weak solutions to the Cauchy problem with general large initial data. However, the uniqueness of such solutions remains open. For the initial-boundary value problem under different boundary conditions, Hamdache \cite{Hamdache-ARMA-1992} and Arkeryd-Maslova \cite{AM-JSP-1994}  obtained the global existence of renormalized weak solutions. Desvillettes and Villani \cite{DV-IM-2005} obtained the almost exponential convergence result for large-data solutions satisfying some additional regularity assumptions. Subsequently, Gualdani, Mischler, and Mouhot \cite{GMM-MSM-2017} extended the result for attaining optimal exponential convergence.

For classical solutions near the global Maxwellian equilibrium, significant progress has been made on both global well-posedness and time-asymptotic behavior. Ellis and Pinsky \cite{EP-JMPA-1975} developed the spectral analysis of the linearized Boltzmann equation and its fluid-dynamical limits. Based on the spectral analysis theory of the Boltzmann equation and the bootstrap argument, Ukai \cite{Ukai-1974} first established the global existence and exponential decay of strong solutions on the torus. The corresponding algebraic decay of strong solutions on the whole space can be found in \cite{NI-1976-PRIMS,Ukai-1976}.  Moreover, regarding the long time behavior, by combining the compensation function and the Fourier energy method, the time decay rates of the Boltzmann equation were investigated by Kawashima \cite{Ks-JJAM-1990}. Furthermore, there are also extensive results concerning the global existence and uniqueness of classical solutions to the Boltzmann equation near vacuum. Interested readers can refer to \cite{BT-JMP-1985,Glassey 1996,IS-CMP-1984}.
On the other hand, to study the global existence and uniqueness of strong solutions to the Boltzmann equation near the Maxwellian equilibrium, Guo \cite{Gy-IUMJ-2004} proposed a nonlinear energy method based on macro-micro decomposition around global Maxwellians, and Liu-Yang-Yu \cite{LYY-PD-2004} developed the energy method around local Maxwellians for studying stability of wave patterns in the context of conservation laws.

The study of time-periodic solutions is a classical and significant issue in both kinetic theory and gas dynamics. Regarding the compressible Navier–Stokes equations, these solutions and their stability characteristics have been comprehensively analyzed in various settings; for instance, refer to Beirão da Veiga \cite{Bb-ARMA-2005}, Feireisl et al. \cite{EMPS-ARMA-1999}, and Valli–Zajaczkowski \cite{VZ-CMP-1986}. In particular, for the isentropic case, Kagei and Tsuda \cite{KaTs} proved the existence of time-periodic solutions for sufficiently small time-periodic external force with some symmetry for the space dimension $d\geq 3$. The proof is based on the spectral properties of the time-$T$-map associated with the linearized problem around the motionless state with constant density in some weighted Sobolev space.  Tsuda \cite{Tk-ARMA-2016} further obtained the same result without any symmetry condition on the external force field. In the non-isentropic case, the time-periodic problem on the full compressible Navier-Stokes-Fourier system with  a time-periodic external force was very recently studied by Deguchi \cite{Deguchi-2025}. 

On the kinetic aspect, the periodic solutions to the Boltzmann equation driven by time-periodic source terms were investigated by Ukai \cite{Ukai-2006} and Ukai-Yang \cite{UY-AA-2006}, provided that an extra spatially zero-average condition is imposed on the inhomogeneous source term $S_0$ in case of the physically relevant three dimensions. Their proof is based on the fixed point argument for a nonlinear mapping over an infinite time integral $(-\infty,t)$, where the zero-average condition helps gain the faster time decay of the semigroup such that the infinite time integral term is integrable for the time-periodic source. Recently, the authors of this paper with Lei \cite{DLN-2026} removed this zero-average restriction on $S_0$ and developed a global dynamic theory in three dimensions. Nevertheless, for the Boltzmann equation with a time-periodic external force in the whole space, the theory remains far from complete. To the best of our knowledge, the existing result in  \cite{DUYZ-CMP-2008} mentioned before is applicable in only dimensions $d\geq 5$, and the physically crucial case $d = 3$ remains unresolved. The purpose of this paper is to address this problem for $d = 3$. Motivated by earlier work \cite{DL-2015-AMSSB} and the aforementioned preprint \cite{Deguchi-2025}, we shall adopt Serrin's approach \cite{Serrin-1959-ARMA} to treat the issue.

Before studying the time-periodic solution of \eqref{G1.3}, we first consider the following Cauchy problem:
 \begin{equation}\label{G1.5}
\left\{\begin{aligned}
& \partial_t f + v\cdot\nabla_x f + \mathcal{L}f  \\
&\quad= \Gamma (f,f)-E\cdot\nabla_v f+\frac{1}{2}E\cdot vf+E\cdot v\sqrt{M}, \quad (t,x,v)\in \mathbb R^+\times \mathbb R^3_{x}\times \mathbb R^3_{v},\\
 &f(0,x,v)=f_0(x,v)=\frac{F_{0}(x,v)-M}{\sqrt{M}}, \quad    (x,v)\in    \mathbb R^3_{x}\times \mathbb R^3_{v}.
 \end{aligned}
 \right.
\end{equation} 
Similarly to \cite{DLN-2026}, for the low-frequency part, we do not select the general $L^2_{x,v}$-norm; instead, we choose the $L_v^2(\dot B^{1/2}_{2,\infty})$. Our selection of homogeneous Besov spaces $L_v^2(\dot B^{1/2}_{2,\infty})$ is motivated by \cite{Deguchi-2025} for a study of the full compressible Navier-Stokes-Fourier system in three dimensions, which essentially noted that $\frac{1}{|x|}\in \dot B_{2,\infty}^{1/2}$.

To begin with, we show the global existence of strong solutions to the Cauchy problem \eqref{G1.5} in order to apply Serrin's approach in \cite{Serrin-1959-ARMA} for further studying the time-periodic problem \eqref{G1.2}. For later use, we introduce the energy norm $\|\cdot\|_{\CE^{s,N}}$ as 
\begin{align}
\label{norm.E}
\|f\|_{\mathcal{E}^{s,N}}:=&\|f\|_{L_v^2(\dot B_{2,\infty}^{s}\cap \dot H^N)}+\|\langle v\rangle f \|_{L_v^2(\dot H^1\cap \dot H^{N-1})} +\|\{\mathbf{I}-\mathbf{P}\}f\|_{L_v^2(L^2)}\nonumber\\
&\quad+\sum_{\substack{ 1\leq |\beta|\leq N \\|\alpha|+|\beta| \leq N}}\|\partial^\alpha_{x}\partial^\beta_{v}\{\mathbf{I}-\mathbf{P}\}f \|_{L_{v}^2(L^2)}.
\end{align}
The first main result is stated as follows.

\begin{thm}\label{Th1}
    Let $F_0(x,v)=M+\sqrt{M}f_0(x,v)\geq 0$ with
    \begin{equation*}
    f_0\in  L_v^2(\dot B_{2,\infty}^{\frac{1}{2}}\cap\dot H^N),\quad \langle v\rangle f_0\in L_v^2(\dot H^1\cap \dot H^{N-1}),\quad \{\mathbf{I}-\mathbf{P}\}f_0\in H^N_{x,v},
    \end{equation*}
 for an integer $N\geq 3$, then there exists a constant $\delta_0>0$ such that if the external force $E$ and the initial data $f_0$ satisfy
\begin{equation}\label{G1.6}
\|f_0\|_{\CE^{\frac{1}{2},N}}+\|E (t)\|_{\mathcal{C} (\mathbb R; \dot B_{2,\infty}^{-\frac{3}{2}}\cap \dot H^N)  }\leq \delta_0,
\end{equation}
then the Cauchy problem \eqref{G1.5} admits a unique global solution $f(t,x,v)$ that satisfies $F(t,x,v)=M+\sqrt{M}f(t,x,v)\geq 0$ with
\begin{align*}
f\in \mathcal{C}\big([0,\infty);L_v^2(\dot B_{2,\infty}^{\frac{1}{2}}\cap \dot H^N)\big),\  \langle v\rangle f\in \mathcal{C}\big([0,\infty);L_v^2( \dot H^1\cap \dot H^{N-1})\big),\ 
\{\mathbf{I}-\mathbf{P}\}f\in \mathcal{C}\big([0,\infty);H^N_{x,v}\big),
\end{align*}
and
\begin{align}\label{G1.7}
\|f(t)\|_{\CE^{\frac{1}{2},N}}
\leq\, C_0 \left(\|f_0\|_{\CE^{\frac{1}{2},N}}
+\|E (t)\|_{\mathcal{C} (\mathbb R; \dot B_{2,\infty}^{-\frac{3}{2}}\cap \dot H^N)  }\right), 
\end{align}
for any $t\geq 0$. Here $C_0$ is a positive constant independent of $t$. 
\end{thm}

\begin{rem}
Unlike  \cite[Theorem 1.1]{DLN-2026}, the nonlinear terms involve the $v$-weighted term and $v$-derivative term corresponding to $E \cdot vf$ and $E  \cdot \nabla_v f$. Therefore, relying only on the initial data condition $f_0 \in L_v^2(\dot B^{1/2}_{2,\infty}\cap \dot H^N)$ remains insufficient. To address this trouble, we require additional initial-data conditions for velocity weight and velocity derivative, namely, $\langle v\rangle f_0 \in L_v^2(\dot H^1\cap\dot H^{N - 1})$, $\{\mathbf{I}-\mathbf{P}\}f_0\in L_v^2(L^2)$, and  $\partial^\alpha_{x}\partial^\beta_{v}\{\mathbf{I}-\mathbf{P}\}f_0\in L_{v}^2(L^2)$ with $|\alpha|+|\beta|\leq N$ and $1\leq |\beta|\leq N$.
\end{rem}

Next, we demonstrate the asymptotic stability of the Cauchy problem \eqref{G1.5} that also needs to be used in the proof of the time-periodic problem \eqref{G1.2}.

\begin{thm}\label{Th2}
Let $0<\varepsilon<\frac{1}{2}$ and let $N\geq 4$ be an integer. 
Let $F^{(1)}_0(x,v)=M+\sqrt{M}f^{(1)}_0(x,v)\geq 0$ and $F^{(2)}_0(x,v)=M+\sqrt{M}f^{{(2)}}_0(x,v)\geq 0$, and let $f^{(1)}$ and $f^{(2)}$ be the pair of solutions of the Cauchy problem \eqref{G1.5} with initial data $f^{(1)}_0(x,v)$ and $f^{(2)}_0(x,v)$, respectively. Suppose that the external force  $E$ and the initial data $f^{(1)}_0$ and $f^{(2)}_0$ satisfy \eqref{G1.6} in Theorem \ref{Th1} with $\delta_0>0$. Then, if the initial data further satisfy
\begin{align}\label{G1.8}
f^{(1)}_0 - f^{(2)}_0\in L_v^2(\dot B_{2,\infty}^{s_0}),
\end{align}
for some $s_0\in\big(-\frac{3}{2},\frac{1}{2}\big]$, then we obtain the following time-weighted estimate
\begin{align}\label{G1.9}
&\sup_{t\geq 0}(1 + t)^{\frac{s - s_0}{2}}\big\|(f^{(1)} - f^{(2)})(t)\big\|_{L_v^2(\dot B_{2,\infty}^{s}\cap\dot H^{N-1})}\nonumber\\
&+\sup_{t\geq 0} (1+t)^{\frac{1-\varepsilon-s_0}{2}}  \big \|\{\mathbf{I}-\mathbf{P}\}  (f^{(1)}-f^{(2)})(t)\big\|_{L_{x,v}^2}\nonumber\\
 &+\sup_{t\geq 0} (1+t)^{\frac{1-\varepsilon-s_0}{2}} 
 \big\|\langle v\rangle (f^{(1)}-f^{(2)}) (t)\big\|_{L_v^2(\dot H^1\cap\dot H^{N-2})}\nonumber\\
& +\sup_{t\geq 0} (1+t)^{\frac{1-\varepsilon-s_0}{2}}   \sum_{\substack{ 1\leq |\beta|\leq N-1 \\|\alpha|+|\beta| \leq N-1}}\big\|\partial^\alpha_{x}\partial^\beta_{v}\{\mathbf{I}-\mathbf{P}\}( f^{(1)}-f^{(2)})(t) \big\|_{L_{x,v}^2 }   \nonumber\\
\leq\,& C_1\|f^{(1)}_0 - f^{(2)}_0\|_{\CE^{s_0,N-1}},
\end{align}
for any $s\in\big[-\frac{3}{2}+\varepsilon,1-\varepsilon \big]$ with $s\geq s_0$. Here $C_1$ is a positive constant independent of $t$. 
\end{thm}

\begin{rem}
Our method here differs from that in \cite{Deguchi-2025,DLN-2026}. We need to overcome the difficulty caused by the additional velocity weight and velocity derivative. This is resolved by obtaining higher time-decay rates of $\langle v\rangle \widetilde f$ and $\nabla_v\widetilde f$ at higher frequencies and then iteratively returning to lower frequencies.
However, despite such difficulty we overcame, in contrast to the method in \cite{Deguchi-2025,DLN-2026}, we have not obtained the optimal rate of $\widetilde f$ in the $L_v^2(\dot B_{2,\infty}^{s})$-norm for $s \in \big(1-\varepsilon, \frac{3}{2}-\varepsilon\big]$. 
\end{rem}

\begin{rem}
Unlike the case of $N\geq 3$ in Theorem \ref{Th1}, we here require $N\geq4$ in Theorem \ref{Th2}. This is because when estimating $\{\mathbf{I}-\mathbf{P}\}(f^{(1)}-f^{(2)})$, due to technical reasons, we need to ensure the embedding $\dot H^1\cap \dot H^{N - 2}\hookrightarrow L^\infty$ in case of three spatial dimensions.
\end{rem}

With the aid of Theorems \ref{Th1} and \ref{Th2}, we  further consider the  time-periodic problem \eqref{G1.2}. For the sake of notational simplicity, we set $F_T = M+\sqrt{M}f_{T}$. Then, we rewrite the problem \eqref{G1.2} as 
\begin{equation}\label{G1.10}
\left\{\begin{aligned}
& \partial_t f_{T}+ v\cdot \nabla_{x} f_{T}+\mathcal{L}f_{T} \\
&\quad =\Gamma(f_{T},f_{T})-E\cdot\nabla_v f_{T}+ \frac{1}{2}E\cdot vf_{T}+E\cdot v\sqrt{M},\quad \text{in}\quad \mathbb R\times \mathbb R_{x}^3\times \mathbb R_{v}^3,\\
&f_{T}(t,x,v)=f_{T}(t+T,x,v)\quad \text{for}\quad (t,x,v)\in \mathbb R\times \mathbb R_x^3\times \mathbb R_v^3,   
\end{aligned} \right.
\end{equation} 
where $E(t,x )=E(t+T,x )$ for $(t,x)\in \mathbb R\times \mathbb R_x^3 $. Finally, we state the theorem about the existence and stability of the time-periodic problem \eqref{G1.10}.

\begin{thm} \label{Th1.3}
Let $T > 0$ and let $N\geq 4$ be an integer. There exists a small constant $\delta > 0$ such that if the time-periodic force term $E$ with the time period $T > 0$ satisfies
\begin{equation}\label{cond.E}
\|E (t)\|_{\mathcal{C} (\mathbb{R}; \dot B_{2,\infty}^{-\frac{3}{2}}\cap \dot H^N )}\leq \delta,
\end{equation}
then we have the following time-periodic results:
\begin{itemize}
    \item (Existence of time-periodic solutions) There exists a unique time-periodic solution $f_T$ of the same time period $T$ satisfying $F_T = M+\sqrt{M}f_{T}\geq 0$ and
    \begin{equation}\label{fT.bdd}
    \sup_{t\in \R}\|f_T(t)\|_{\CE^{\frac{1}{2},N}}\leq C_2\delta.
     \end{equation}
    \item (Stability of time-periodic solutions) If initial data $f_0(x,v)$ satisfy that $F_0 = M+\sqrt{M}f_0\geq 0$, 
    \begin{align*}
\|f_0\|_{\CE^{\frac{1}{2},N}}\leq \delta,
    \end{align*}
     with $\delta>0$ sufficiently small, and
    \begin{align*}
    f_0 - f_T(0) \in L^2_vL^p,
    \end{align*}
    with $1< p\leq 2$, then the Cauchy problem \eqref{G1.5} admits a unique global solution $f$ satisfying $F = M+\sqrt{M}f\geq 0$,
     \begin{align*}
    \|f(t)\|_{\mathcal{C} ([0,\infty);L_v^2(\dot B_{2,\infty}^{\frac{1}{2}}\cap \dot H^N))}+\|(\langle v\rangle f(t),\nabla_v f(t ))\|_{\mathcal{C} ([0,\infty);L_v^2( \dot H^1\cap \dot H^{N-1}))}   \leq C_2\delta,
    \end{align*} 
    and the time decay estimate:
    \begin{align}\label{G1.15}
&\,\|(f - f_T)(t)\|_{L_v^2(\dot H^{s})}\nonumber\\
\leq&\, C_2 (1 + t)^{-\frac{s}{2}-\frac{3}{2}(\frac{1}{p}-\frac{1}{2})}\Big(\|(f - f_T)(0)\|_{L_v^2( L^p\cap\dot H^N)}  +\|\langle v\rangle (f - f_T)(0)\|_{L_v^2(\dot H^1\cap \dot H^{N-2})}   \nonumber\\
&+ \|\{\mathbf{I}-\mathbf{P}\}(f - f_T)(0)\|_{L_{x,v}^2}+\sum_{\substack{ 1\leq |\beta|\leq N-1 \\|\alpha|+|\beta| \leq N-1}}\|\partial^\alpha_{x}\partial^\beta_{v}\{\mathbf{I}-\mathbf{P}\}(f - f_T)(0)\|_{L_{x,v}^2 }\Big),
    \end{align}
    for any $s\in\big(-\frac{3}{2}+\varepsilon,1-\varepsilon\big)$ with $\varepsilon>0$ and $\frac{s}{2}+\frac{3}{2}\big(\frac{1}{p}-\frac{1}{2}\big)>0$.
   Here $C_2$ is a positive constant independent of $t$. 
\end{itemize} 
\end{thm}

\begin{rem}
Indeed, building on the existence and asymptotic stability of time-periodic solutions established in Theorem \ref{Th1.3}, and following a strategy similar to that used in \cite[Theorem 1.4]{DLN-2026-VFP}, we are able to establish the corresponding existence and asymptotic stability results for the linearized Boltzmann equation. As highlighted in \cite{Valli-1983}, our analysis contributes to a broader, global understanding of the dynamical behavior of the linearized Boltzmann equation.
\end{rem}

\begin{rem}
The $v$-growth arising from the nonlinear term $E\cdot v f$ implies that the method in our paper can only attain $\gamma = 1$ corresponding to the hard-sphere model if we consider the general Boltzmann collision kernel of the form $B(|v-v_\ast|,\omega)=|v-v_\ast|^\gamma b(\cos\theta)$, cf.~\cite{Glassey 1996}, where for cutoff potentials, we assume $-3<\gamma\leq 1$ and $0\leq b(\cos\theta)\leq C |\cos\theta|$ with $\cos\theta=\omega\cdot (v-v_\ast)/|v-v_\ast|$ for a constant $C>0$. However, the same method fails to treat all the non-hard-sphere cases $-3<\gamma< 1$; such a problem will be left for the future.
\end{rem}

We give more remarks whenever $E=E(x)$ is time-independent, satisfying \eqref{cond.E} as well. Indeed, in such case, the result of Theorem  \ref{Th1.3} directly yields the existence and stability of stationary solutions to the steady Boltzmann equation in the physically realistic three spatial dimensions:  
\[
v\cdot \nabla_{x} F+E(x)\cdot\nabla_{v}F = Q(F,F),\ (x,v)\in \R^3\times \R^3.
\]
Note that if $E(x)=-\nabla\phi(x)$ of the potential form, then $F(x,v)=e^{-\phi(x)}M$ is the unique stationary solution to the steady problem above. The importance of our result in the stationary case is that the force field $E(x)$ can be allowed to be non potential and so rotational with $\text{curl}\,E\neq 0$. For example, let \(m > 2\) and define
\begin{align*}
E_\epsilon(x)=\frac{\epsilon(-x_2,x_1,0)}{\langle x\rangle^{2m}},\qquad x=(x_1,x_2,x_3)\in\mathbb{R}^3,
\end{align*}
where \(0 < |\epsilon|\ll1\). Then \(E_\epsilon\) is time-independent and represents a rotational force field around the \(x_3\)-axis. First, it is straightforward to see \(\nabla_x\cdot E_\epsilon = 0\). Also, by direct calculation, we have
\begin{align*}
\nabla_x\times E_\epsilon = 2\epsilon\left( \frac{mx_1x_3}{\langle x\rangle^{2m + 2}}, \frac{mx_2x_3}{\langle x\rangle^{2m + 2}}, \frac{1 + |x|^2 - m(x_1^2 + x_2^2)}{\langle x\rangle^{2m + 2}} \right).
\end{align*}
In particular,
\((\nabla_x\times E_\epsilon)(0)=(0,0,2\epsilon)\neq0\).
Hence \(E_\epsilon\) is not a potential force field. Next, we verify that \(E_\epsilon\) belongs to the function space used in Theorem  \ref{Th1.3}. Since \(m > 2\), one has
\begin{align*}
|E_\epsilon(x)|\lesssim|\epsilon|(1 + |x|)^{1 - 2m},    
\end{align*}
and this leads to $
E_\epsilon\in L^1(\mathbb R^3)\cap\dot H^N(\mathbb R^3)$.
Consequently, by virtue of the embedding \(L^1(\mathbb R^3)\cap\dot H^N(\mathbb R^3)\hookrightarrow\dot B^{-3/2}_{2,\infty}(\mathbb R^3)\cap\dot H^N(\mathbb R^3)\), we finally obtain
\begin{align*}
E_\epsilon\in\dot B^{-3/2}_{2,\infty}(\mathbb R^3)\cap\dot H^N(\mathbb R^3),
\end{align*}
and
\begin{align*}
\|E_\epsilon\|_{\dot B^{-3/2}_{2,\infty}\cap\dot H^N}\leq C_{m,N}|\epsilon|.
\end{align*}
We choose \(|\epsilon|\) to be sufficiently small so that \(C_{m,N}|\epsilon|\leq \delta_0\), which indicates that our framework is applicable to a genuinely rotational, non-conservative external force field.

In what follows, we give the proof strategies for the main results presented before. First of all, to prove the global existence of the  Boltzmann system \eqref{G1.5} in Theorem \ref{Th1}, we mainly adopt a method that combines the application of semi-group theory at low frequencies with the application of energy methods at high frequencies. 
First, we decompose $f$ as $f = f_{L} + f_{H}$. Then, according to the Duhamel principle, we have
\begin{align}\label{QQ1}
f_{L}(t)=&\, e_{L}^{t\mathcal B}f_{0}+\int_0^t e_L^{(t-\tau)\mathcal{B}} \Big[ \Gamma(f,f)(\tau)-[E \cdot \nabla_v f](\tau)+\frac{1}{2}[E\cdot v f](\tau)+E(\tau)\cdot v\sqrt{M}    \Big]{\rm d}\tau\nonumber\\
=&\,  e_{L}^{t\mathcal B}f_{0}+ \int_{0}^t e_{L}^{(t-\tau)\mathcal B} E(\tau)\cdot v\sqrt{M}  {\rm d}\tau+ \int_{0}^t e_{L}^{(t-\tau)\mathcal B} \Gamma(f,f)(\tau) {\rm d}\tau \nonumber\\
&+\int_{0}^t e_{L}^{(t-\tau)\mathcal B} \Big[  \frac{1}{2}[E\cdot v f](\tau) -[E \cdot \nabla_v f](\tau)  \Big] {\rm d}\tau \nonumber\\
\equiv:&\,  e_{L}^{t\mathcal B}f_{0}+\mathcal{I}_1+\mathcal{I}_2+\mathcal{I}_3.
\end{align}
Let's discuss the treatment of the nonlinear terms $\mathcal{I}_1$, $\mathcal{I}_2$, and $\mathcal{I}_3$ in \eqref{QQ1} one by one. On the one hand, similar to the approach in \cite{DLN-2026}, for the terms $\mathcal I_1$ and $\mathcal{I}_2$, by applying the duality argument, Fourier analysis, and semi-group theory, we can prove the following inequalities (see the proof in Lemma \ref{L3.1}):
\begin{align}\label{QQ2}
\|\mathcal{I}_1\|_{L_v^2(\dot B^{\frac{1}{2}}_{2,\infty})}\lesssim&\, \sup_{t\geq 0}\|E(t)\|_{\dot B_{2,\infty}^{-\frac{3}{2}}},  \\ \label{QQ3}
\|\mathcal{I}_2\|_{L_v^2(\dot B^{\frac{1}{2}}_{2,\infty})}\lesssim&\, \sup_{t\geq 0}\|\nu^{-1}\Gamma(f,f)(t)\|_{L_v^2(\dot B_{2,\infty}^{-\frac{1}{2}})}.
\end{align}
In fact, \eqref{QQ2} indicates that through the utilization of the second-order derivative loss, the $L^1$ integral over time is eliminated. This is a crucial situation. The reason lies in the fact that the second derivative is equivalent to the decay rate of $(1 + t)^{-1}$, which is non-integrable in the $L^1$ sense. This case is non-trivial within the framework of the usual semi-group theory. The estimate \eqref{QQ3} indicates that the microscopic component $\Gamma(f,f)(t)$ has a faster first-order derivative than the solution. Nevertheless, this characteristic can only be uncovered during spectral analysis because it is of zeroth order regarding spatial derivatives. 

On the other hand, the processing of $\mathcal I_3$ is more complicated, since it involves $v$-weighted growth and the first-order velocity derivative $\nabla_v$ loss. Additionally, our semi-group theory is only applicable to the physical space and offers no useful information for the velocity space. Thus, if, as in  \cite{DLN-2026}, the solution is restricted to the framework of $L_v^2(\dot B^{1/2}_{2,\infty}\cap \dot H^N)$ with $N\geq 3$, it is inadequate.
Fortunately, even though we are unable to obtain the regularity propagation of $\langle v\rangle f$ and $\nabla_{v} f$ at low frequencies, we can improve the regularity of $\langle v\rangle f$ and $\nabla_{v} f$ by sacrificing the partial regularity of $E(t,x)$. Namely, $\mathcal{I}_3$ can be estimated by
\begin{align*} 
\|\mathcal{I}_3\|_{L_v^2(\dot B_{2,\infty}^{\frac{1}{2}})}\lesssim \sup_{t\geq 0}\|E(t)\|_{\dot B^{-\frac{3}{2}}_{2,\infty}}  \|(\langle v\rangle f(t),\nabla_v f(t))\|_{L_v^2(\dot H^1\cap \dot H^{N-1})}.  
\end{align*}
This is because, in the high-frequency region, we are able to achieve the regularity propagation of $\langle v\rangle f$ and $\nabla_{v} f$ provided that we strengthen some additional initial data.
Our system here is more complex and challenging when compared to the general Navier-Stokes system \cite{Deguchi-2024-MathAnn} or the Navier-Stokes-Fourier system \cite{Deguchi-2025}. This is due to the fact that we need to conduct some refined energy estimates on the macro and micro parts and surmount the difficulties brought about by the $v$-weighted terms. 
With the help of the estimates for $\mathcal{I}_1$, $\mathcal{I}_2$, and $\mathcal{I}_3$, we can consequently obtain the estimate of $\|f_{L}(t)\|_{L_v^2(\dot B_{2,\infty}^{1/2})}$.

For the high-frequency part, we first establish the uniform {\it a priori estimates} of the solution $f(t,x,v)$ in the $L_v^2(\dot H^1\cap \dot H^N)$-norm, where $N\geq 3$, by utilizing the macro-micro decomposition method and the refined energy method. It is important to note that for $f$, we can only carry out interpolation between $L_v^2(\dot B_{2,\infty}^{1/2})$ and $L_v^2(\dot H^N)$ to obtain energy estimates. This differs from the interpolation estimates between $L_v^2(L^2)$ and $L_v^2(\dot H^N)$ for $f$ as described in \cite{GW-CPDE-2012}. We might require more precise estimates of $f(t,x,v)$. Moreover, we consequently obtain (See Corollary \ref{cor3.5}):
\begin{align}\label{QQ4}  
\|f(t)\|_{L_v^2(\dot H^1\cap \dot H^N)}\lesssim \sup_{0\leq t\leq T_1}\Big\{\|E(t)\|_{  H^N }+\|f_{L}(t)\|_{L_v^2(\dot B_{2,\infty}^{\frac{1}{2}})} \Big\}+\|f_0\|_{L_v^2(\dot H^1\cap \dot H^N)}.
\end{align}
Unlike \cite{DLN-2026}, since \eqref{QQ4} is not sufficient to cover assumption \eqref{G3.1}, considering that 
\begin{align*}
\|\langle v\rangle \mathbf{P}f\|_{L_{x,v}^2}+\|\nabla_v \mathbf{P}f\|_{L_{x,v}^2}\lesssim \|f\|_{L_{x,v}^2},    
\end{align*}
we need to further focus on estimating the micro components $v\{\mathbf{I}-\mathbf{P}\}f$ and $\nabla_{v}\{\mathbf{I}-\mathbf{P}\}f$ by regularity propagation.
This idea mainly originated from  \cite{Gy-IM-2003,Gy-CPAM-2006}. It is important to note that when handling the zero-order product term generated by $\Gamma(f,f)$ in the $L^2$ framework during the proof of Lemmas \ref{L3.7}--\ref{L3.8}, we need to estimate it through  using the square of the $L^4$ norm of $f$.
It is worth noting that there exists an embedding $\dot H^{3/4} \hookrightarrow  L^4$, since $L_v^2(\dot H^{3/4})$ lies between $L_v^2(\dot B^{1/2}_{2,\infty})$ and $L_v^2(\dot H^N)$. Therefore, we can consequently absorb it. By leveraging the refined energy method, we further obtain the following  Lyapunov-type inequality regarding the $\{\mathbf{I}-\mathbf{P}\}f$ (see Corollary \ref{cor3.9}):
\begin{align*}
&\frac{{\rm d}}{{\rm d}t}\bigg(\|\nu\{\mathbf{I}-\mathbf{P}\}f(t) \|_{L_v^2(\dot H^1\cap\dot H^{N-1})}^2 +\|\{\mathbf{I}-\mathbf{P}\}f(t)\|_{L_{x,v}^2}^2+ \sum_{\substack{ 1\leq |\beta|\leq N \\|\alpha|+|\beta| \leq N}}\|\partial^\alpha_{x}\partial^\beta_{v}\{\mathbf{I}-\mathbf{P}\}f(t) \|_{L_{x,v}^2 }^2\bigg)
 \nonumber\\
&+ \lambda_{13} \bigg(\|\nu^{\frac{3}{2}}\{\mathbf{I}-\mathbf{P}\}f(t)\|_{L_v^2(\dot H^1\cap\dot H^{N-1})}^2+ \|\{\mathbf{I}-\mathbf{P}\}f(t)\|_{\nu}^2+\sum_{\substack{ 1\leq |\beta|\leq N \\|\alpha|+|\beta| \leq N}} \| \partial^{\alpha}_{x}\partial^\beta_{v}\{\mathbf{I}-\mathbf{P}\}f(t)\|_{\nu}^2\bigg)\nonumber\\    
&\quad\lesssim \|f(t)\|_{L_v^2(\dot H^\frac{3}{4}\cap\dot H^N)}^2+\|\nu^\frac{1}{2}\{\mathbf{I}-\mathbf{P}\}f(t)\|_{L_v^2(\dot H^1\cap\dot H^{N-1})}^2+\sup_{0\leq t\leq T_1}\|E(t)\|_{H^N}^2,   
\end{align*}
which, together with the estimates of $f$ in the hybrid homogeneous space $L_v^2(\dot B^{1/2}_{2,\infty}\cap\dot H^N)$ established in Section 3.1, leads to the uniform estimate:
$$
\sup_{0\leq t\leq T_1}\|f(t)\|_{\CE^{\frac{1}{2},N}}\lesssim \|f_0\|_{\CE^{\frac{1}{2},N}}+ \sup_{0\leq t\leq T_1}\|E(t)\|_{\dot B_{2,\infty}^{-\frac{3}{2}}\cap \dot H^{N}}.
$$
Thanks to the standard continuity argument, we prove the global existence of the Cauchy problem \eqref{G1.5}. Therefore, we complete the proof of Theorem \ref{Th1}.

To prove the asymptotic stability of the system \eqref{G1.5}  in Theorem \ref{Th2}, we first present the error equation between two solutions $f^{(1)}(t,x,v)$ and $f^{(2)}(t,x,v)$ as follows:
\begin{align*} 
\partial_{t} \widetilde{f}+v\cdot\nabla_{x}\widetilde{f}+\mathcal{L}\widetilde{f}=\Gamma(f^{(1)}+f^{(2)},\widetilde{f})-E\cdot\nabla_{v}\widetilde{f}+\frac{1}{2} E\cdot v\widetilde{f},
\end{align*}
where $\widetilde f(t,x,v)=f^{(1)}(t,x,v)-f^{(2)}(t,x,v)$.
Our primary approach here is to employ the time-weighted estimate method at low and high frequencies, respectively.
Different from the Boltzmann equation with source term \cite{DLN-2026} and the Navier-Stokes-Fourier system \cite{Deguchi-2025}, the nonlinear terms, such as $E\cdot\nabla_{v} \widetilde f$ and $\frac{1}{2}E\cdot v\widetilde f$, contain the $v$-weighted or velocity derivative $\nabla_{v}$. This indicates that at low frequencies, they cannot be fully converted into energy without $v$ or $\nabla_{v}$, which prevents the obtainment of any time decay rates in the low-frequency range.
The key to overcoming this difficulty lies in initially establishing a time-weighted estimate of $\widetilde f$, $\langle v\rangle \widetilde f$, and $\nabla_{v}\widetilde f$ at high frequencies (See Theorem \ref{T4.6}):
\begin{align} \label{QQ5}
&\sup_{t\geq 0} (1+t)^{\frac{1-\varepsilon-s_0}{2}}  \Big( 
 \|\widetilde f(t)\|_{L_v^2(\dot H^1\cap \dot H^{N-1})}+\|\langle v\rangle \widetilde f(t) \|_{L_v^2(\dot H^1\cap\dot H^{N-2})} +\|\{\mathbf{I}-\mathbf{P}\}\widetilde f(t)\|_{L_{x,v}^2}  \Big)\nonumber\\
& +\sup_{t\geq 0} (1+t)^{\frac{1-\varepsilon-s_0}{2}}   \sum_{\substack{ 1\leq |\beta|\leq N-1 \\|\alpha|+|\beta| \leq N-1}}\|\partial^\alpha_{x}\partial^\beta_{v}\{\mathbf{I}-\mathbf{P}\}\widetilde f(t) \|_{L_{x,v}^2 }    \nonumber\\
\lesssim&\,   \widetilde{\mathcal{D}} +\|\widetilde f_0\|_{L_v^2(\dot H^1\cap \dot H^{N-1})}+\|\langle v\rangle \widetilde f_0 \|_{L_v^2(\dot H^1\cap\dot H^{N-2})} +\|\{\mathbf{I}-\mathbf{P}\}\widetilde f_0\|_{L_{x,v}^2} \nonumber\\
&+\sum_{\substack{ 1\leq |\beta|\leq N-1 \\|\alpha|+|\beta| \leq N-1}}\|\partial^\alpha_{x}\partial^\beta_{v}\{\mathbf{I}-\mathbf{P}\}\widetilde f_0\|_{L_{x,v}^2 },
\end{align}
where
\begin{align*}
\widetilde{\mathcal{D}}:=  \sup_{t\geq 0}\,(1+t)^{\frac{1}{2}-\frac{\varepsilon+s_0}{2}} \|\widetilde f (t)\|_{L_v^2(\dot B_{2,\infty}^{1-\varepsilon} )} .   
\end{align*}    
Based on the faster time decay rates of $\langle v\rangle \widetilde f$ and $\nabla_{v}\widetilde f$ established in \eqref{QQ5} when compared to those in the low-frequency regime, and by applying the semi-group theory and low-high frequency decomposition, for the cases where $s$ lies in the intervals $\big[-\frac{3}{2}+\varepsilon, \frac{1}{2}\big)$, $s = \frac{1}{2}$, and $\big(\frac{1}{2}, 1-\varepsilon\big]$, we eventually derived the following low-frequency estimate (see Theorem \ref{T4.7}):
\begin{align}\label{QQ6}
&\sup_{t\geq 0} (1+t)^{\frac{s-s_0}{2}}\|\widetilde f\|_{L_v^2(\dot B_{2,\infty}^s)} \nonumber\\
\lesssim&\, \delta_0 \big(\widetilde{\mathcal{D}}_{\varepsilon} +\widetilde{\mathcal{D}} \big)+\|\widetilde f_0\|_{L_v^2(\dot B_{2,\infty}^{s_0}\cap \dot B_{2,\infty}^{1-\varepsilon})}+\|\widetilde f_0\|_{L_v^2(\dot H^1\cap \dot H^{N-1})}+\|\langle v\rangle\widetilde f_0 \|_{L_v^2(\dot H^1\cap\dot H^{N-2})}  \nonumber\\
&+\|\{\mathbf{I}-\mathbf{P}\}\widetilde f_0\|_{L_{x,v}^2}+\sum_{\substack{ 1\leq |\beta|\leq N-1 \\|\alpha|+|\beta| \leq N-1}}\|\partial^\alpha_{x}\partial^\beta_{v}\{\mathbf{I}-\mathbf{P}\}\widetilde f_0\|_{L_{x,v}^2 },       
\end{align}
where
\begin{align*}
\widetilde{\mathcal{D}}_{\varepsilon}:=\sup_{s_1\leq \bar s\leq 1-\varepsilon} \sup_{t\geq 0}\,(1+t)^{\frac{\bar s-s_0}{2}} \|\widetilde f (t)\|_{L_v^2(\dot B_{2,\infty}^{\bar s}\cap \dot H^{N-1})},\quad s_1=\max\{0,{s_0}\}.  
\end{align*}
Then, by combining \eqref{QQ5} and \eqref{QQ6}, we prove the stability of the Cauchy problem \eqref{G1.5}.

With Theorem \ref{Th1} and Theorem \ref{Th2} at hand, we are able to prove the existence and stability of the time-periodic problem \eqref{G1.10} in Theorem \ref{Th1.3}. Our proof is based on Serrin's method (cf. \cite{Serrin-1959-ARMA,Mp-1991-Nonlinearity}). First, we construct a Cauchy sequence $\{f_{*}(nT)\}_{n\geq 1}$ in $L_v^2(\dot B_{2,\infty}^{1 - \varepsilon}\cap\dot H^{N - 1})$. Owing to the Banach Completeness Theorem and Fatou’s lemma, we find the limit $f^{\infty}_{*}$. By selecting the initial data $f_{T}(0)=f_{*}^\infty$ for our global solution $f_{T}(t)$ and using the {\it local in time} uniqueness, we can deduce that $f_{T}(t)$ is a unique time-periodic solution to the system \eqref{G1.10}. By using basic interpolation inequalities, we further obtain the time-asymptotic stability of the global solution $f(t)$ around $f_{T}(t)$. Hence, we complete the proof of Theorem \ref{Th1.3}.

The rest of this paper is organized as follows. In Section 2, we provide preliminaries that include some notations, the Littlewood-Paley decomposition, the macro-micro decomposition, decay estimates of the semi-group, and some useful lemmas. In Section 3, we first establish {\it a priori estimates} of $f(t,x,v)$ in the hybrid homogeneous space $L_v^2(\dot B^{1/2}_{2,\infty}\cap\dot H^N)$ using the semi-group theory and the refined energy method. By estimating the micro part  $\{\mathbf{I}-\mathbf{P}\}f$, we further derive the regularity propagation of $\langle v\rangle f$ and $\nabla_{v}f$. Consequently, we obtain the uniform estimate of $f$, and thus complete the proof of Theorem \ref{Th1}. In Section 4, by applying the time-weighted method at low and high frequencies respectively, we obtain the asymptotic stability of the Cauchy problem \eqref{G1.5}. Finally, using Serrin’s approach (cf. \cite{Serrin-1959-ARMA,Mp-1991-Nonlinearity}), we prove the existence and asymptotic stability of the time-periodic solution to the periodic problem \eqref{G1.10}.

\section{Preliminaries}
In this section, we present notations, basic estimates of the operator in the Boltzmann equation, Littlewood-Paley decomposition, some properties of homogeneous Besov spaces, time decay estimates of the semi-group, and some useful lemmas that are frequently used throughout the paper.

\subsection{Notations}
The letter $C$ represents a generic positive constant. The notation $A \lesssim B$ (resp., $A \gtrsim B$) is used to denote $A \leq CB$ (resp., $A \geq CB$). The expression $A \backsim B$ means that both $A \lesssim B$ and $A \gtrsim B$ hold. $\mathbf 1_E$ denotes the characteristic function of a set $E$. 
We denote the $L^2$ inner product in $\mathbb{R}^3_v$ by $\langle \cdot, \cdot \rangle$ 
and in $\mathbb{R}^3_x$ by $(\cdot, \cdot)$, with the corresponding $L^2$ norms 
$|\cdot|_2$ and $\|\cdot\|_{L^2}$, respectively. 
Furthermore, we use $\langle \cdot, \cdot \rangle_{x,v}$ to represent the $L^2$ inner product 
on $\mathbb{R}^3_x \times \mathbb{R}^3_v$, with the associated norm $\|\cdot\|_{L^2_{x,v}}$.
For functions that depend on both spatial and velocity variables, we define $L^2_{x,v} = L^2(\mathbb{R}^3_x \times \mathbb{R}^3_v)$, which is equipped with the norm $\|\cdot\|_{L^2_{x,v}}$. The $L^p$ norm on $\mathbb{R}^3_x$ is denoted as $\|\cdot\|_{L^p}$. For mixed spaces, we express $L^2_v(L^p) = L^2(\mathbb{R}^3_v; L^p(\mathbb{R}^3_x))$ along with the corresponding norm $\|\cdot\|_{L^2_v(L^p )}$. Similarly, we adopt the notations $L_v^2(H^s )$, $L_v^2(\dot H^s)$, and $L_v^2(\dot B^s_{p,r})$ for $s\in \mathbb R$ and $1\leq p,r\leq\infty$.
For the hard-sphere Boltzmann operator, we define the collision frequency as follows:
\begin{align}\label{defn:nu}
\nu(v)=\int_{\mathbb R^3}\!\int_{\mathbb S^2} |(v-v_\ast)\cdot\omega| M_{*}{\rm d}\omega{\rm d}v_{*},
\end{align}
which behaves as $\langle v\rangle :=\sqrt{1 + |v|^2} $. We define the velocity-weighted $L^2$-norms as
\begin{align*}
|g|_{\nu}=|\nu^\frac{1}{2}g|_{2},\quad \|g\|_{\nu}=\|\nu^\frac{1}{2}g\|_{L^2_{x,v}}.
\end{align*}

For a Banach space $X$, we define $\|(g,h)\|_{X}:=\|g\|_{X}+\|h\|_{X}$, where $g$ and $h$ are elements of $X$. For an interval $I$ of $\mathbb R$, we define $\mathcal C(I; X)$ as the set of continuous functions on $I$ that take values in $X$. For any time $T > 0$ and $1 \leq p \leq \infty$, $L^p(0,T ;X)$ denotes the space of measurable functions $f: [0,T ] \to X$ such that the mapping $t \mapsto \|f(t)\|_{X}$ belongs to $L^p(0, T )$, equipped with the norm $\|\cdot\|_{L^p(0, T ;X)} := \|\cdot\|_{L_{ T }^p(X)}$. We denote $\mathcal{S}$ as the set of all Schwartz functions on $\mathbb{R}^3_{x}$, and $\mathcal{S}^\prime$ as the set of all tempered distributions on $\mathbb{R}^3_{x}$. For a function $g(t,x,v)\in \mathcal{S}(\mathbb{R}^3_{x})$, its Fourier transform is defined as
\begin{align*} 
\widehat{g}(t,\xi,v)=\mathcal{F}g(t,\xi,v) =\int_{\mathbb{R}^3}e^{-i x\cdot\xi}  g(t,x,v){\rm d}x,\quad \text{where } x\cdot \xi=\sum_{j = 1}^3x_{j}\xi_{j},
\end{align*}
for all $\xi\in \mathbb{R}^3$, where $i = \sqrt{-1}$ represents the imaginary unit. Moreover, we use $\mathcal{F}^{-1}$ to denote the inverse Fourier transform.

Finally, for the multi-indices $\alpha=(\alpha_1,\alpha_2,\alpha_3)$ and $\beta=(\beta_1,\beta_2,\beta_3)$, we denote 
\begin{align*}
\partial_{x}^\alpha\partial_{v}^\beta = \partial_{x_1}^{\alpha_1}\partial_{x_2}^{\alpha_2} \partial_{x_3}^{\alpha_3} \partial_{v_1}^{\beta_1} \partial_{v_2}^{\beta_2}\partial_{v_3}^{\beta_3}.
\end{align*} 
The lengths of $\alpha$ and $\beta$ are defined as $|\alpha| = \alpha_1+\alpha_2+\alpha_3$ and $|\beta| = \beta_1+\beta_2+\beta_3$. 
For simplicity, for each $i = 1, 2, 3$, we use $\partial_{i}$ to represent $\partial_{x_{i}}$. 
To be concise, for each $i = 1, 2, 3$, we also write $\partial_i$ instead of $\partial_{x_i}$. Moreover, we use the symbol $\nabla$ to represent $\nabla_x$.

\subsection{Basic estimates of the operator in the Boltzmann equation}
As mentioned in \cite{UY-AA-2006,CIP-1994}, the operator $\mathcal{L}$ defined in \eqref{G1.3} can be decomposed as
\begin{align*}
\mathcal{L}=\nu - \mathcal{K},  
\end{align*}
where $\nu$ is defined by \eqref{defn:nu}.
The  operator $\mathcal{K}$ can also be decomposed as
\begin{align*}
\mathcal{K}=\mathcal{K}_2-\mathcal{K}_1,    
\end{align*}
with 
\begin{align*}
\mathcal{K}_2g(v):=\int_{\mathbb R^3}\!\int_{\mathbb S^2} |(v-v_\ast)\cdot\omega|\sqrt{M(v_*)}\Big\{\sqrt{M(v_*^\prime)}g(v^\prime)+\sqrt{M(v^\prime)}g(v^\prime_{*})\Big\}{\rm d}\omega {\rm d}v_{*},
\end{align*}
and 
\begin{align*}
\mathcal{K}_1g(v):=\int_{\mathbb R^3}\!\int_{\mathbb S^2} |(v-v_\ast)\cdot\omega|\sqrt{M(v_*)}\sqrt{M(v)}g(v_{*}){\rm d}\omega {\rm d}v_{*}.
\end{align*}
Then, by performing some basic computations on $\mathcal{K}_1$ and $\mathcal{K}_2$, we can obtain the following estimate.

\begin{prop}[{\!\!\cite[Lemma 2.2]{Gy-CPAM-2002}}]
Let $|\alpha| = k$. Then, for any   small $\eta > 0$, there exists a constant $C_{k,\eta} > 0$ such that, for any $g(x,v) \in H^k(\mathbb{R}^3_{x} \times \mathbb{R}^3_{v})$, the following inequality holds:
\begin{align*}
\left\|\partial^{\alpha}_{v}\left[ \mathcal{K}g\right]\right\|_{L_{x,v}^2}^2 \leq \eta \sum_{|\beta| = k}\left\|\partial^\beta_{v}g\right\|_{L_{x,v}^2}^2 + C_{k,\eta}\|g\|_{L_{x,v}^2}^2.
\end{align*}    
\end{prop}
Next, let's collect some useful estimates of the linear collision operator $\mathcal{L}$.

\begin{prop}[{\!\!\cite[Lemmas 3.2--3.3]{Gy-CPAM-2006}}]\label{prop2.2}
It holds that $\langle \mathcal{L} h_1, h_2\rangle=\langle h_1,\mathcal{L} h_2\rangle$, and $\langle\mathcal{L}h,h\rangle\geq 0 $, with $\mathcal{L}h=0$ if and only if $h=\mathbf{P}h$. Moreover, there exists a constant $\kappa_0>0$, such that
\begin{align}\label{G2.2}
\langle \mathcal{L}h,h\rangle\geq \kappa_0|\{\mathbf{I}-\mathbf{P}\}h|_{\nu}^2,    
\end{align}
and
\begin{align}\label{G2.3}
\langle \nu^{2l} \partial^\alpha_{x}\partial^\beta_{v} \mathcal{L}h,\partial^\alpha_{x}\partial^\beta_{v}h\rangle\geq \frac{1}{2}|\nu^{l} \partial^\alpha_{x}\partial^\beta_{v}h|_{\nu}^2-C|h|_{\nu}^2,
\end{align}
 for any $l\geq 0$.
\end{prop}

To handle the nonlinear collision operator $\mathbf{\Gamma}$, we need to obtain some velocity-weighted and velocity-derivative estimates as follows.

\begin{prop}[{\!\!{\cite[Lemma 2.3]{Gy-CPAM-2002}}} and  {\cite[Lemma 2.7]{UY-AA-2006}}] \label{prop2.3}
There exists $C>0$ such that
\begin{align*}
|\langle\Gamma(g_1,g_2),g_3\rangle|+|\langle\Gamma(g_2,g_1),g_3\rangle|\leq C\sup_{v}\{\nu^3g_3\} |g_1|_2|g_2|_2.   
\end{align*}
Moreover, for any $0\leq\eta\leq1$, we have
\begin{align}\label{G2.5}
|\nu^{-\eta}\Gamma(g,h)|_2\leq C \big(|\nu^{1-\eta}g|_{2}|h|_{2}+|\nu^{1-\eta }h|_{2}|g|_{2}\big).   
\end{align}
\end{prop}

\begin{prop}[{\!\!\cite[Lemma 3.3]{Gy-CPAM-2006}}]\label{prop2.4}
Let $g_{i}(x,v)$, $i=1,2,3,$ be smooth functions. Then we have
\begin{align}\label{NJKGnew}
&|\langle \partial^\alpha_{x}\partial^\beta_{v} \Gamma (g_1,g_2), \partial^\alpha_{x}\partial^\beta_{v}  g_3\rangle |\nonumber\\ 
&\quad \leq    C\sum_{\substack{\alpha_1\leq\alpha,\,\beta_1+\beta_{2}\leq\beta \\|\alpha|+|\beta| \leq N}}\big ( |\partial_{x}^{\alpha_1}\partial^{\beta_1}_v g_1|_2|\partial^{\alpha_2}_{x}\partial^{\beta_2}_{v}g_2|_{\nu} + |\partial_{x}^{\alpha_1}\partial^{\beta_1}_v g_2|_2|\partial^{\alpha_2}_{x}\partial^{\beta_2}_{v}g_1|_{\nu}\big)|\partial^{\alpha}_{x}\partial^{\beta}_{v}g_3|_{\nu} .
\end{align}
\end{prop}

\begin{prop}[{\!\!\cite[Lemma 3.3]{Gy-CPAM-2006}}]\label{NJKP2.5}
There exists $C>0$ such that
\begin{align}\label{NJKG2.6}
&|\langle\nu^2\partial^\alpha_{x} \Gamma(g_1,g_2),\partial^\alpha_x g_3\rangle|\nonumber\\
&\quad \leq  C \sum_{\alpha_1+\alpha_2=\alpha}\big (
|\nu   \partial^{\alpha_1}g_1|_{2} |\nu   \partial^{\alpha_2}g_2  |_{\nu}+ |\nu   \partial^{\alpha_1}g_1|_{\nu} |\nu   \partial^{\alpha_2}g_2  |_{2} \big) | \nu\partial^\alpha g_3|_{\nu}.
\end{align}
\end{prop}

\subsection{The Littlewood-Paley decomposition and certain properties of homogeneous Besov spaces.}
Now, we  recall the Littlewood-Paley decomposition theory and
some properties of
homogeneous Besov spaces. The reader can refer to \cite[Chapter 2]{BCD-Book-2011} for more details.
Let $\chi(\xi)$ be a smooth, radial, non-increasing function supported in the ball $B\big(0,\frac{4}{3}\big)$ and satisfying $\chi(\xi) = 1$ for all $\xi \in B\big(0,\frac{3}{4}\big)$. Define
\begin{align*}
 \varphi(\xi) := \chi\Big(\frac{\xi}{2}\Big) - \chi(\xi).   
\end{align*}
Then $\varphi$ satisfies
\begin{align*}
\sum_{k \in \mathbb{Z}} \varphi(2^{-k} \xi) = 1 \quad \text{for all } \quad\xi \neq 0,   
\end{align*}
and its support is contained in
\begin{align*}
\mathrm{supp}\,\varphi \subset \Big\{ \xi \in \mathbb{R}^3 \;\Big|\; \frac{3}{4} \le |\xi| \le \frac{8}{3} \Big\}.    
\end{align*}
For any $j\in \mathbb{Z}$, the homogeneous dyadic block $\dot{\Delta}_j$ is defined as 
\begin{align*}
\dot{\Delta}_j g := \mathcal{F}^{-1}\left( \varphi(2^{-j}\cdot)\mathcal{F}(g) \right) = 2^{3j} h(2^j\cdot) * g, 
\end{align*}
where $h := \mathcal{F}^{-1}\varphi$.

Let $\mathcal{P}$ denote the set of all polynomials on $\mathbb{R}^3$, and define
\begin{align*}
\mathcal{S}'_h (\mathbb R^3):= \mathcal{S}' / \mathcal{P}  (\mathbb R^3)  
\end{align*}
as the space of tempered distributions on $\mathbb{R}^3$ modulo polynomials. 
Then, for any $g \in \mathcal{S}'_h (\mathbb R^3)$, one has the homogeneous Littlewood--Paley decomposition
\begin{align*}
g= \sum_{j \in \mathbb{Z}} \dot{\Delta}_j g,    
\end{align*}
where the dyadic blocks satisfy
\begin{align*}
\dot{\Delta}_j \dot{\Delta}_k g = 0 \quad \text{whenever }\quad |j - k| \ge 2.    
\end{align*}
 
For any $j \in \mathbb{Z}$, we define the low-frequency cut-off operator $\dot{S}_j$ by
\begin{align*}
\dot{S}_j g := \sum_{j' < j} \dot{\Delta}_{j'} g, \quad \forall g \in \mathcal{S}'(\mathbb{R}^3),    
\end{align*}
which will be used in subsequent analysis. 
Moreover, for any $g \in \mathcal{S}'(\mathbb{R}^3)$, we denote its low- and high-frequency parts by
\begin{align*}
g_L := \dot{S}_{j_0} g, \qquad g_H := g - \dot{S}_{j_0} g,    
\end{align*}
for some fixed $j_0 \in \mathbb{Z}$.

With the assistance of those dyadic blocks, we   now present the definition of homogeneous Besov spaces. 
\begin{defn}
For any $ s \in \mathbb{R} $ and $ 1 \leq p, q\leq \infty $, the homogeneous Besov spaces $ \dot{B}_{p,q}^s$ are defined as
\begin{align*}
\dot B_{p,q}^s:=\big\{ g\in \mathcal{S}^\prime_h \,\big{|}\, \|g\|_{\dot{B}_{p,q}^s}:=\big\|\{2^{ks}\|\dot\Delta_k g\|_{L^p}\}_{k\in\mathbb Z}\big\|_{l^q}<\infty\big\}. 
\end{align*}
\end{defn}

Then, we present a class of mixed space-velocity homogeneous Besov spaces as described below.

\begin{defn}
For any $ s \in \mathbb{R} $ and $ 1 \leq p, q\leq \infty $, we use the notation $L_v^2(\dot B_{p,q}^s)$ to denote the space $L^2( \mathbb R^3_{v} ; \dot B_{p,q}^s(\mathbb R^3_{x}))$, and it is defined by
\begin{align*}
 L_v^2(\dot B_{p,q}^s):=\bigg\{ g\in L_v^2(\mathcal{S}^\prime_h) \,\bigg{|}\, \|g\|_{L_v^2(\dot{B}_{p,q}^s)}:=\bigg(\int_{\mathbb R^3} \|g\|_{\dot B_{p,q}^s}^2{\rm d}v\bigg)^\frac{1}{2} <\infty\bigg\}.    
\end{align*}
\end{defn}

Next, we state some fundamental properties of the Besov space.

\begin{prop}{\rm(\!\!\cite[Chapter 2]{BCD-Book-2011})}\label{prop2.5}
The following properties hold:
\begin{itemize}
\item {}For any non-negative integer $k\in\mathbb{Z}$, it follows that
\begin{align*}
\|\nabla^k g\|_{\dot B_{2,q}^s}\backsim \|g\|_{\dot B_{2,q}^{s+k}} .
\end{align*}
\item{} For $s\in\mathbb{R}$, $1\leq p_{1}\leq p_{2}\leq \infty$ and $1\leq q_{1}\leq q_{2}\leq \infty$, it holds that
\begin{equation}\notag
\begin{aligned}
\dot{B}^{s}_{p_{1},q_{1}}\hookrightarrow \dot{B}^{s-3 (\frac{1}{p_1} -\frac{1}{p_2})}_{p_{2},q_{2}}.
\end{aligned}
\end{equation}
\item{} For $1\leq p\leq q\leq\infty$, one has the   chain of continuous embedding  
\begin{equation}\nonumber
\begin{aligned}
  \dot{B}^{0}_{p,1}\hookrightarrow L^{p}\hookrightarrow \dot{B}^{0}_{p,\infty}\hookrightarrow \dot B_{q,\infty}^{\varsigma},\quad {\text{where}}\quad\varsigma=-3\Big(\frac{1}{p}-\frac{1}{q}\Big).
\end{aligned}
\end{equation}
\item{} If $p < \infty$, then the Besov space $\dot{B}^{\frac{3}{p}}_{p,1}$ is continuously embedded in the set of continuous functions that decay to 0 at infinity.

\item{}  The following real interpolation property is satisfied for $1\leq p,q\leq \infty$, $s_1<s_2$, and $\theta\in (0,1)$:
\begin{align}\label{interpolation}
\|g\|_{\dot B_{p,q}^{\theta s_1+(1-\theta)s_2}}\lesssim\|g\|_{\dot B_{p,1}^{\theta s_1+(1-\theta)s_2}}\lesssim \frac{1}{\theta(1-\theta)(s_2-s_1)}\|g\|_{\dot B_{p,\infty}^{s_1}}^{\theta} \|g\|_{\dot B_{p,\infty}^{s_2}}^{1-\theta}  .  
\end{align}

\item{} Let $s\in \mathbb{R}$ and $1\leq q\leq \infty$. Let $q^\prime$ denote the conjugate exponent of $q$. Then, the following duality estimates hold:
\begin{align} \label{duality}
\langle g,h\rangle\lesssim \|g\|_{\dot B_{2,q}^s} \|h\|_{\dot B_{2,q^\prime}^{-s}}, \quad \text{and} \quad \|g\|_{\dot B_{2,q}^s}\lesssim \sup_{\phi\in \mathcal{S}}\langle g,\phi\rangle.
\end{align}
where the supremum is taken over $\phi$ such that $\|\phi\|_{\dot B_{2,q^\prime}^{-s}}\leq 1$ and $0 \notin {\rm Supp}\mathcal{F}\phi$.

\item{}  For any $\epsilon>0$, it holds that
\begin{equation}\nonumber
\begin{aligned}
H^{s+\epsilon}\hookrightarrow \dot{B}^{s}_{2,1}\hookrightarrow \dot{H}^{s}.
\end{aligned}
\end{equation}
\item{}
Let $\Lambda^{\sigma}$ be defined by $\Lambda^{\sigma}:=(-\Delta )^{\frac{\sigma}{2}}g:=\mathcal{F}^{-1}\big{(} |\xi|^{\sigma}\mathcal{F}(g) \big{)}$ for $\sigma\in \mathbb{R}$ and $g\in{\mathcal S}^{'}_{h}(\mathbb{R}^3)$, then $\Lambda^{\sigma}$ is an isomorphism from $\dot{B}^{s}_{p,q}$ to $\dot{B}^{s-\sigma}_{p,q}$.
\item{} Let $1\leq p_{1},p_{2},q_{1},q_{2}\leq \infty$, $s_{1}\in\mathbb{R}$ and $s_{2}\in\mathbb{R}$ satisfy
\begin{align*}
 s_{2}<\frac{3}{2}\quad\text{\text{or}}\quad s_{2}=\frac{3}{2}~\text{and}~q_{2}=1.
\end{align*}
  Then, the space $\dot{B}^{s_{1}}_{2,q_{1}}\cap \dot{B}^{s_{2}}_{2,q_{2}}$ endowed with the norm $\|\cdot \|_{\dot{B}^{s_{1}}_{2,q_{1}}}+\|\cdot\|_{\dot{B}^{s_{2}}_{2,q_{2}}}$ is a Banach space and possesses the weak compactness and Fatou properties. If $\{g_{n}\}$ is a uniformly bounded sequence in $\dot{B}^{s_{1}}_{2,q_{1}}\cap \dot{B}^{s_{2}}_{2,q_{2}}$, then there exist an element $g$ of $\dot{B}^{s_{1}}_{2,q_{1}}\cap \dot{B}^{s_{2}}_{2,q_{2}}$ and a subsequence $\{g_{n_{k}}\}$ such that $g_{n_{k}}\rightarrow g$ in $\mathcal{S}'$, and 
    \begin{align*} 
    \begin{aligned}
    \|g\|_{\dot{B}^{s_{1}}_{2,q_{1}}\cap \dot{B}^{s_{2}}_{2,q_{2}}}\lesssim \liminf_{n_{k}\rightarrow \infty} \|g_{n_{k}}\|_{\dot{B}^{s_{1}}_{2,q_{1}}\cap \dot{B}^{s_{2}}_{2,q_{2}}}.
    \end{aligned}
    \end{align*}
\end{itemize}
\end{prop}

To control the nonlinear terms in the Boltzmann equation \eqref{G1.5}$_1$, we require the following product estimates.

 \begin{prop}\label{prop2.6}{\rm(\!\!\cite[Chapter 2]{BCD-Book-2011})}
The following statements hold:
\begin{itemize}
\item{} Let $s>0$, $1\leq  q\leq \infty$.  Then $\dot B^{s}_{2,q}\cap L^\infty$ is an algebra and
 \begin{align*} 
\|g_1g_2\|_{\dot B_{2,q}^s}\lesssim \|g_1\|_{L^\infty}\|g_2\|_{\dot B_{2,q}^s}+\|g_2\|_{L^\infty}\|g_1\|_{\dot B_{2,q}^s},    
 \end{align*}
which, in particular, gives rise to
\begin{align*}
\|g_1g_2\|_{\dot B_{2,1}^{\frac{3}{2}}}\lesssim \|g_1\|_{\dot B_{2,1}^{\frac{3}{2}}}\|g_2\|_{\dot B_{2,1}^{\frac{3}{2 }}} .    
\end{align*}

\item{} Let the real numbers $s_1$, $s_2$  and $r$ fulfill
\begin{align*}
 1\leq  q,q_1,q_2\leq\infty, \quad s_1<\frac{3}{2},\quad s_2<\frac{3}{2},\quad s_1+s_2>0,\quad \frac{1}{q}=\frac{1}{q_1}+\frac{1}{q_2}.
\end{align*}
  Then, it holds  that
 \begin{align} \label{newinterpolation}
\|g_1 g_2\|_{\dot B_{2,q}^{s_1+s_2-\frac{3}{2}}}\lesssim \|g_1\|_{ \dot B_{2,q_1}^{s_1 }}\|g_2\|_{\dot B_{2,q_2}^{s_2}} .    
 \end{align}

\item{} Let  the real numbers $s_1$, $s_2$, and $q$ fulfill
\begin{align*}
 s_1\leq \frac{3}{2},\quad s_2< \frac{3}{2},\quad s_1+s_2\geq 0.   
\end{align*}
Then, it holds that
 \begin{align} \label{inter}
\|g_1g_2\|_{\dot B_{2,\infty}^{s_1+s_2-\frac{3}{2}}}\lesssim \|g_1\|_{ \dot B_{2,1}^{s_1 }}\|g_2\|_{\dot B_{2,\infty}^{s_2}}.    
 \end{align}
    \end{itemize}
\end{prop}

\subsection{Time decay estimates of semi-group}
To study the behavior of spectral properties, we first rewrite equation \eqref{G1.5}$_1$ by
\begin{align}\label{G2.10}
\partial_t f  - \mathcal{B}f=\Gamma (f,f)-E\cdot\nabla_v f+\frac{1}{2}E\cdot vf+E\cdot v\sqrt{M}, \quad (t,x,v)\in \mathbb R\times \mathbb R^3_{x}\times \mathbb R^3_{v},    
\end{align}
where $\mathcal B$ is the linearized Boltzmann operator, which is given by
\begin{align*}
\mathcal{B}=-v\cdot\nabla_x-\mathcal{L}.
\end{align*}
Now, let's turn to the linearized Cauchy problem of \eqref{G2.10} as follows:
 \begin{equation}\label{G2.11}
\left\{\begin{aligned}
&  \partial_t f-\mathcal{B}f=0, \quad (t,x,v)\in \mathbb R\times \mathbb R^3_{x}\times \mathbb R^3_{v},\\ 
& f(0,x,v)=f_0(x,v)=\frac{F_{0}(x,v)-M}{\sqrt{M}},\quad \quad\,   (x,v)\in    \mathbb R^3_{x}\times \mathbb R^3_{v}.
 \end{aligned}
 \right.
\end{equation} 

Then, similar to the approach in \cite{SK-HMJ-1985}, we establish sharp point-wise decay estimates of the solution $e^{t\mathcal{B}}f_0(x,v)$ to the linearized problem \eqref{G2.11}. In fact, based on the classical spectral analysis methods as described in \cite{EP-JMPA-1975,UY-AA-2006,Ukai-1974,Ukai-1976,D-2011-Nonlinearity}, it is straightforward to observe that the solutions of system \eqref{G2.11} exhibit the diffusion effect at low frequencies and a spectrum gap at high frequencies. Moreover, the following proposition is presented.

\begin{prop}[{\!\!{\cite[Proposition 2.3]{DLN-2026}}\!\!}  ] \label{prop2.7}
For any $t \geq 0$ and $\xi\in\mathbb{R}^3$, the solution $e^{t\mathcal{B}} f_0(x,v)$ to \eqref{G2.11} satisfies
\begin{align}\label{G2.12}
\|\widehat{e^{t\mathcal{B}}f_0}(\xi,v)\|_{L_{v}^2}\lesssim   e^{-\kappa_{1}\min\{t,|\xi|^2t\}}  {}\|\widehat{f_0}(\xi,v)\|_{L_v^2},
\end{align}
where $\kappa_1 > 0$ is a uniform constant. If $\mathbf{P}f_0 = 0$ is additionally fulfilled, there exists a constant $\xi_0>0$ such that if $|\xi|\leq \xi_0$, then
\begin{align}\label{G2.13}
\|\widehat{e^{t\mathcal{B}}f_0}(\xi,v)\|_{L_{v}^2}\lesssim  |\xi| e^{-\frac{\kappa_{1}}{2} |\xi|^2t }   \|\nu^{-1}\widehat{f_0}(\xi,v)\|_{L_v^2}+e^{-\kappa_0t}\|\nu^{-1}\widehat{f_0}(\xi,v)\|_{L_v^2},
\end{align}
for any $t\geq 0$, where $\kappa_0$ is given by \eqref{G2.2}.
\end{prop}

Inspired by \cite[Lemma 3.3]{Deguchi-2025} and \cite[Lemma 4.1]{Deguchi-2024-MathAnn}, and based on Proposition \ref{prop2.7}, we further establish the following $L^2$-type 
time decay rates in both the low-frequency and high-frequency regions in the homogeneous Besov space.

\begin{prop}[{\!\!{\cite[Proposition 2.4]{DLN-2026}}\!\!}  ] \label{prop2.8}
Let $t\geq 0$ and $1\leq q\leq \infty$. Let $s,s_0\in \mathbb{R}$, satisfying $s_0\leq s$. Then, there exists $j_0\in \mathbb{Z}$ such that, for any $1\leq q\leq\infty$ and $f_0\in L_v^2(\dot B_{2,q}^{s_0})$, the following inequality holds:
\begin{align}\label{G2.14}
\|\dot S_{j_0}e^{t\mathcal{B}} f_0\|_{L_v^2(\dot B_{2,q}^s)}\lesssim (1+t)^{-\frac{s - s_0}{2}} \|f_0\|_{L_v^2(\dot B_{2,q}^{s_0})}.
\end{align}
For any $1\leq q\leq\infty$ and $f_0\in L_v^2(\dot B_{2,q}^{s})$, we have
\begin{align}\label{G2.15}
\|(1 - \dot S_{j_0})e^{t\mathcal{B}} f_0\|_{L_v^2(\dot B_{2,q}^s)}\lesssim e^{-\kappa_2t} \|f_0\|_{L_v^2(\dot B_{2,q}^{s})},
\end{align}
where $\kappa_2>0$ is a constant.
For any $1\leq q\leq\infty$ and $f_0\in L_v^2(\dot B_{2,q}^{s-\frac{2}{q}})$, the following inequality is valid:
\begin{align}\label{G2.16}
\|\dot S_{j_0}e^{t\mathcal{B}}f_0\|_{L^q(L_v^2(\dot B_{2,q}^s))}\lesssim \|f_0\|_{L_v^2(\dot B_{2,q}^{s-\frac{2}{q}})}.
\end{align}
Furthermore, if the condition $\mathbf{P}f_0 = 0$ is also satisfied, we can obtain $\frac{1}{2}$-order faster decay estimates compared to those in \eqref{G2.14} and \eqref{G2.16}, as shown below:
\begin{align}\label{G2.17}
\|\dot S_{j_0}e^{t\mathcal{B}} f_0\|_{L_v^2(\dot B_{2,q}^s)}\lesssim&\, (1+t)^{-\frac{s - s_0 + 1}{2}} \|\nu^{-1}f_0\|_{L_v^2(\dot B_{2,q}^{s_0})},\\ \label{G2.18}
\|(1 - \dot S_{j_0})e^{t\mathcal{B}} f_0\|_{L_v^2(\dot B_{2,q}^s)}\lesssim&\, e^{-\kappa_2t} \|\nu^{-1}f_0\|_{L_v^2(\dot B_{2,q}^{s})},\\ \label{G2.19}
\|\dot S_{j_0}e^{t\mathcal{B}}f_0\|_{L^q(L_v^2(\dot B_{2,q}^s))}\lesssim&\, \|\nu^{-1}f_0\|_{L_v^2\big(\dot B_{2,q}^{s+1-\frac{2}{q}}\big)}.
\end{align}
The estimates \eqref{G2.14}, \eqref{G2.15}, \eqref{G2.16}, \eqref{G2.17}, \eqref{G2.18} and \eqref{G2.19} still remain valid when $\mathcal{B}$ is replaced by its adjoint $\mathcal B^*$.
\end{prop}

\subsection{Some useful lemmas}
Below, we list some useful lemmas that will be used throughout this paper.

\begin{lem}[{\!\!\cite[Lemma A.8]{GW-CPDE-2012}}]\label{LA.1}
Let $1\leq p<\infty$. Suppose $g $ is a measurable function defined on $\mathbb{R}_{y}^3\times \mathbb{R}_{z}^3$. Then, we have
   
\begin{align*}
\bigg(   \int_{\mathbb{R}_{z}^3} \bigg(\int_{\mathbb{R}_{y}^3} |g(y,z)|{\rm d}y\bigg)^p  {\rm d}z  \bigg)^\frac{1}{p}\leq \int_{\mathbb{R}_{y}^3}     \bigg(\int_{\mathbb{R}_{z}^3} |g(y,z)|^p{\rm d}z\bigg)^{\frac{1}{p}}  {\rm d}y. 
\end{align*}
In particular, for $1\leq p\leq q\leq\infty$, we have
\begin{align*}
\|g\|_{L^q_{z}L^p_{y}}\leq \|g\|_{L^p_{y}L^q_{z
}}.    
\end{align*}
\end{lem}

\begin{lem} [{\!\!\cite[Theorem 1.38]{BCD-Book-2011}}]\label{LA.2}
If $s\in\big[0,\frac{3}{2}\big)$,  then the homogeneous Sobolev space $\dot H^s(\mathbb R^3)$ is continuously embedded in $L^{\frac{6}{3-2s}}(\mathbb R^3)$.
\end{lem}

\begin{lem} [{\!\!\cite[Lemma A.1]{GW-CPDE-2012}}]\label{LA.3}
Let \(2\leq p\leq \infty\) and \(0\leq m,\alpha\leq\ell\). When \(p = \infty\), we further require that \(m\leq \alpha + 1\) and \(\ell\geq \alpha + 2\). Then, for any \(g\in C_0^\infty(\mathbb{R}^3)\), we have
\begin{align*}
\|\nabla^\alpha g\|_{L^p}\lesssim \|\nabla^m g\|_{L^2}^{1 - \zeta}\|\nabla^\ell g\|_{L^2}^{\zeta},
\end{align*}
where \(0\leq \zeta\leq 1\) and \(\alpha\) satisfies
\begin{align*}
\frac{\alpha}{3}-\frac{1}{p}=\left(\frac{m}{3}-\frac{1}{2}\right)(1 - \zeta)+\left(\frac{\ell}{3}-\frac{1}{2}\right)\zeta.
\end{align*}
\end{lem}

\section{Global well-posedness of Boltzmann equation}
In this section, we conduct an analysis of the reformulated Cauchy problem \eqref{G1.5}. Our objective is to demonstrate the global existence and uniqueness of strong solutions to this problem. To attain this objective, we mainly establish uniform-in-time {\it a priori estimates} for $f(t,x,v)$.

\subsection{A priori estimates at low and high frequencies}
Recall \eqref{norm.E}. Let $T_1>0$ be an arbitrary constant. Assume that $f(t,x,v) $ is a strong solution to the Cauchy problem \eqref{G1.5} defined on the interval $0 \leq t \leq T_1$, satisfying  
\begin{equation}\label{G3.1}
\sup_{0\leq t\leq T_1} \|f(t)\|_{\CE^{\frac{1}{2},N}}\leq \sigma,
\end{equation}
where $N\geq 3$ is an integer and $\sigma>0$ is a sufficiently small   constant independent of $T_1$. By Sobolev’s inequality and \eqref{G3.1}, we directly have
\begin{align*}
&\|f(t)\|_{L_v^2(L^\infty)}+\|(\langle v\rangle f(t),\nabla_v f(t))\|_{L_v^2(L^\infty)}\nonumber\\
&\quad\lesssim \|f(t)\|_{L_v^2(\dot B_{2,\infty}^{\frac{1}{2}}\cap\dot H^N)}+  \|(\langle v\rangle f(t),\nabla_{v}f(t))\|_{L_v^2(\dot  H^1\cap\dot H^{N-1})}\lesssim \sigma,
\end{align*}
for $0\leq t\leq T_1$.

Fix $j_0\in \mathbb{Z}$ and decompose $f$ as $f = f_L + f_H$, where $f_L = \dot{S}_{j_0}f$ and $f_{H}=f-\dot{S}_{j_0}f$. First, we focus on low-frequency analysis and provide an estimate of $f_L$ in the $L_v^2(\dot{B}_{2,\infty}^{1/2})$-norm.

\begin{lem}\label{L3.1}
For strong solutions of the problem   \eqref{G1.5}, it holds that
\begin{align}\label{G3.3}
 \|f_{L}(t)\|_{L_v^2(\dot B_{2,\infty}^{\frac{1 }{2}})}  \lesssim&\, \sigma  \sup_{t\geq 0} \|f(t)\|_{L_v^2(\dot B_{2,\infty}^{\frac{1}{2}} )}  + \sup_{t\geq 0}\|E(t)\|_{ \dot B_{2,\infty}^{-\frac{3}{2}} }+\|f
_0\|_{L_v^2(\dot B_{2,\infty}^{\frac{1 }{2}})},
\end{align}
for $0\leq t \leq T_1$.
\end{lem}

\begin{proof}
Applying the operator $\dot S_{j_0}$ to equation \eqref{G2.10} and using the Duhamel principle, we have
\begin{align*}
f_{L}(t)= e_{L}^{t\mathcal B}f_{0}+\int_0^t e_L^{(t-\tau)\mathcal{B}} \Big[ \Gamma(f,f)(\tau)-[E \cdot \nabla_v f](\tau)+\frac{1}{2}[E\cdot v f](\tau)+E(\tau)\cdot v\sqrt{M}    \Big]{\rm d}\tau, 
\end{align*}
where $e_{L}^{t\mathcal{B}}=\dot S_{j_0}e^{t\mathcal{B}}$.
Taking the $L_v^2(\dot B_{2,\infty}^{\frac{1}{2}})$-norm of the above inequality and then applying Minkowski's inequality, we   derive that
\begin{align}\label{G3.4}
\|f_{L}(t)\|_{L_v^2(\dot B_{2,\infty}^{\frac{1}{2}})}\lesssim&\, \|e_{L}^{t\mathcal{B}} f_0\| _{L_v^2(\dot B_{2,\infty}^{\frac{1}{2}})} +\bigg\|\int_0^t  e_L^{(t-\tau)\mathcal{B}} \big[ \Gamma(f,f)(\tau)+E(\tau) \cdot v\sqrt{M}   \big]{\rm d}\tau     \bigg\|_{L_v^2(\dot B_{2,\infty}^{\frac{1}{2}})}  \nonumber\\
&+\bigg\|\int_0^t  e_L^{(t-\tau)\mathcal{B}} \Big[  \frac{1}{2}[E\cdot vf](\tau)-[E  \cdot\nabla_{v}f](\tau)   \Big]{\rm d}\tau     \bigg\|_{L_v^2(\dot B_{2,\infty}^{\frac{1}{2}})}\nonumber\\
\lesssim&\, \bigg\|\int_0^t  e_L^{(t-\tau)\mathcal{B}}  \Gamma(f,f)(\tau) {\rm d}\tau     \bigg\|_{L_v^2(\dot B_{2,\infty}^{\frac{1}{2}})}+\bigg\|\int_0^t  e_L^{(t-\tau)\mathcal{B}} E(\tau)\cdot v\sqrt{M}    {\rm d}\tau     \bigg\|_{L_v^2(\dot B_{2,\infty}^{\frac{1}{2}})}\nonumber\\
&+\bigg\|\int_0^t  e_L^{(t-\tau)\mathcal{B}} [E \cdot vf](\tau)   {\rm d}\tau     \bigg\|_{L_v^2(\dot B_{2,\infty}^{\frac{1}{2}})}+\bigg\|\int_0^t  e_L^{(t-\tau)\mathcal{B}} [E  \cdot \nabla_{v}f](\tau)   {\rm d}\tau     \bigg\|_{L_v^2(\dot B_{2,\infty}^{\frac{1}{2}})}\nonumber\\
&+\|f_0\|_{L_v^2(\dot B_{2,\infty}^{\frac{1}{2}})}\nonumber\\
\equiv:&\, \sum_{j=1}^4 I_{j}+\|f_0\|_{L_v^2(\dot B_{2,\infty}^{\frac{1}{2}})}.
\end{align}
We handle the terms $I_1$, $I_2$, $I_3$, and $I_4$ individually by employing the duality argument. For the term $I_1$, similar to the approach in \cite[Lemma 3.1]{DLN-2026}, given that $\mathbf{P}\Gamma(f,f)=0$, by applying the Fourier-Plancherel Theorem, \eqref{G2.13}, and \eqref{G2.16}, for any test function $\phi\in L_v^2(\mathcal{S})$, one obtains
\begin{align*}
\bigg\langle \int_0^t e^{(t-\tau)\mathcal{B}}_{L} \Gamma(f,f)(\tau){\rm d}\tau, \phi      \bigg\rangle_{x,v}   =&\, \int_0^t\big\langle  e^{(t-\tau)\mathcal{B}}_{L} \Gamma(f,f)(\tau), \phi      \big\rangle_{{x,v}} {\rm d}\tau \nonumber\\
=&\,\int_0^t \big \langle  \nu^{-1}\widehat{\Gamma}(f,f)(\tau), 
  \mathbf{1}_{|\xi|\leq\xi_0}e^{-|\xi|^2(t-\tau)}|\xi|\widehat{\phi}        \big\rangle_{\xi,v} {\rm d}\tau\nonumber\\
  =&\,  \int_0^t \big \langle\nu^{-1}{\Gamma}(f,f)(\tau), e_{L}^{(t-\tau )\mathcal{ B}^*} \Lambda \phi \big\rangle_{{x,v}} {\rm d}\tau \nonumber\\
  \lesssim&\, \sup_{0\leq t\leq T_1} \|\nu^{-1}\Gamma(f,f)(t)\|_{L_v^2(\dot B_{2,\infty}^{-\frac{1}{2}})} \int_0^\infty \|e_{L}^{\tau \mathcal{B}^*}\Lambda\phi\|_{L_v^2(\dot B_{2,1}^{\frac{1}{2}})} {\rm d}\tau \nonumber\\
   \lesssim&\,  \sup_{0\leq t\leq T_1} \|\nu^{-1}\Gamma(f,f)(t)\|_{L_v^2(\dot B_{2,\infty}^{-\frac{1}{2}})}    \| \phi\|_{L_v^2(\dot B_{2,1}^{-\frac{1}{2}})}, 
\end{align*}
which combined with \eqref{G2.5} in Proposition \ref{prop2.3} and \eqref{duality} in Proposition \ref{prop2.5} gives
\begin{align}\label{G3.6}
I_1\lesssim    \sup_{0\leq t\leq T_1} \|\nu^{-1}\Gamma(f,f)(t)\|_{L_v^2(\dot B_{2,\infty}^{-\frac{1}{2}})}   
\lesssim  \sup_{0\leq t\leq T_1} \||f(t)|_2^2\|_{\dot B_{2,\infty}^{-\frac{1}{2}}}
\lesssim&\, \sup_{0\leq t\leq T_1} \|f(t)\|_{L_v^2(\dot B_{2,\infty}^{\frac{1}{2}})}^2 \nonumber\\
\lesssim&\,\sigma \sup_{0\leq t\leq T_1} \|f(t)\|_{L_v^2(\dot B_{2,\infty}^{\frac{1}{2}})}.
\end{align}
Similar to the process of estimating the term $I_1$, when dealing with the term $I_2$, for any test function $\Psi\in L_v^2(\mathcal{S})$, we have
\begin{align}\label{G3.7}
\bigg\langle \int_0^t e^{(t-\tau)\mathcal{B}}_{L} E(\tau)\cdot v\sqrt{M}{\rm d}\tau, \Psi      \bigg\rangle_{{x,v}}   =&\, \int_0^t \big \langle  e^{(t-\tau)\mathcal{B}}_{L} E(\tau)\cdot v\sqrt{M}, \Psi      \big \rangle_{{x,v}}{\rm d}\tau \nonumber\\
=&\, \int_0^t \big \langle  E(\tau)\cdot v\sqrt{M}, e^{(t-\tau)\mathcal{B}^*}_{L} \Psi      \big \rangle_{{x,v}}{\rm d}\tau \nonumber\\
\lesssim&\, \int_0^t \big( \big|E(\tau)\cdot v\sqrt{M}\big|_{2}  ,\big|e_{L}^{(t-\tau)\mathcal{B}^*}\Psi\big|_{2}  \big)  {\rm d}\tau \nonumber\\
\lesssim&\, \sup_{0\leq t\leq T_1}\|E(t)\cdot v\sqrt{M}\|_{L_v^2(\dot B_{2,\infty}^{-\frac{3}{2}})} \int_0^{\infty}\| e_{L}^{\tau \mathcal{B}^*}\Psi\|_{L_v^2(\dot B_{2,1}^{\frac{3}{2}})} {\rm d}\tau \nonumber\\
\lesssim &\, \sup_{0\leq t\leq T_1}\|E(t)\|_{ \dot B_{2,\infty}^{-\frac{3}{2}}}  \|\Psi\|_{L_v^2(\dot B_{2,1}^{-\frac{1}{2}})},
\end{align}
where we used the fact that $v \sqrt{M}$ is $L^2$-integrable  and the fact from \eqref{G2.16}, which states that
\begin{align*}
  \int_0^{\infty}\| e_{L}^{\tau \mathcal{B}^*}\Psi\|_{L_v^2(\dot B_{2,1}^{\frac{3}{2}})} {\rm d}\tau\lesssim    \|\Psi\|_{L_v^2(\dot B_{2,1}^{-\frac{1}{2}})}. 
\end{align*}
Then, it can be deduced from  \eqref{duality} and \eqref{G3.7} that
\begin{align}\label{G3.8}
I_2\lesssim     \sup_{0\leq t\leq T_1}\|E(t)\|_{\dot B_{2,\infty}^{-\frac{3}{2}}}.     
\end{align}
Similarly as in the process of dealing with the term $I_2$, for the terms $I_3$ and $I_4$, by applying  \eqref{inter} and \eqref{interpolation}, it can be directly concluded that
\begin{align}\label{G3.9}
I_3+I_4\lesssim &\,     \sup_{0\leq t\leq T_1}\|[E\cdot vf](t)\|_{L_v^2(\dot B_{2,\infty}^{-\frac{3}{2}})} +\sup_{0\leq t\leq T_1}\|[E\cdot \nabla_vf](t)\|_{L_v^2(\dot B_{2,\infty}^{-\frac{3}{2}})}\nonumber\\
\lesssim&\,\sup_{0\leq t\leq T_1} \|E(t)\|_{\dot B_{2,\infty}^{-\frac{3}{2}}} \|(v f(t),\nabla_v f(t))\|_{L_v^2(\dot B^{\frac{3}{2}}_{2,1})}\nonumber\\
\lesssim&\,\sup_{0\leq t\leq T_1} \|E(t)\|_{\dot B_{2,\infty}^{-\frac{3}{2}}} \|(v f(t),\nabla_v f(t))\|_{L_v^2(\dot H^1\cap \dot H^{N-1})}\nonumber\\
\lesssim&\, \sigma \sup_{0\leq t\leq T_1} \|E(t)\|_{\dot B_{2,\infty}^{-\frac{3}{2}}}.
\end{align}
By inserting the estimates \eqref{G3.6}, \eqref{G3.8}, and \eqref{G3.9} into \eqref{G3.4}, we   obtain \eqref{G3.3}. 
\end{proof}

Next, we turn to the higher regularity estimates of $f(t,x,v)$ in the $L_v^2(\dot H^1\cap \dot H^N)$ framework.

\begin{lem}\label{L3.2}
For strong solutions of the problem \eqref{G1.5}, there exists a positive constant  $\lambda_1 >0$ such that  
\begin{align}\label{G3.10}
&\frac{{\rm d}}{{\rm d}t}\|  f(t)\|_{L_v^2(\dot H^1\cap \dot H^N)}^2+\lambda_1\sum_{1\leq k\leq N}\|\nabla^k\{\mathbf{I}-\mathbf{P}\}f(t)\|_{\nu}^2    \nonumber\\
&\quad\leq C \sum_{1\leq k\leq N}\sum_{|\beta|\leq k} \big\|  \big|\nabla^{|\beta|}f \big|_2 \big|  \nabla^{k-|\beta|}f \big|_{\nu} \big \|_{L^2}^2+{C\sup_{t\geq 0}\|E(t)\|_{     H^N}^2}+ \kappa \sum_{2\leq k\leq N}\|\nabla^k\mathbf{P}f(t)\|_{L_{x,v}^2}^2,
\end{align}
for $0\leq t \leq T_1$. Here and in the sequel, $0 < \kappa< 1$ is a constant that can be sufficiently small.
\end{lem}

\begin{proof}
Applying the derivative $\partial^\alpha$ ($|\alpha|=k$) with $1\leq k\leq N$ to \eqref{G1.5}$_1$, multiplying the result by $\partial^\alpha f$, taking integration and summation, and then using \eqref{G2.2}, one has
\begin{align}\label{G3.11}
 \frac{1}{2}\frac{{\rm d}}{{\rm d}t}\|\partial^\alpha f\|_{L_{x,v}^2}^2+\lambda_2 \|\partial^\alpha \{\mathbf{I}-\mathbf{P}\}f\|_{\nu}^2\lesssim&\, |\langle\partial^\alpha\Gamma(f,f) ,\partial^\alpha f\rangle_{{x,v}} |+ |\langle\partial^\alpha E\cdot v\sqrt{M},\partial^\alpha f\rangle_{{x,v}} |\nonumber\\
 &+|\langle\partial^ \alpha(E\cdot vf) ,\partial^\alpha f\rangle_{{x,v}} |+|\langle\partial^ \alpha(E\cdot \nabla_vf) ,\partial^\alpha f\rangle_{{x,v}} |\nonumber\\
 \equiv:&\,\sum_{j=5}^8I_j,
\end{align}
for some constant $\lambda_2>0$.

For the first term $I_5$,
by leveraging \eqref{G2.5} in Proposition \ref{prop2.3} and Lemma \ref{LA.1}, utilizing the collision invariant property, and applying Cauchy’s and Young’s inequalities, we get
\begin{align}\label{G3.12}
I_5\lesssim &\, \sum_{|\beta|\leq |\alpha|}\big| \big\langle \Gamma(\partial^{\alpha-\beta }f,\partial^\beta f)   , \partial^\alpha \{\mathbf{I}-\mathbf{P}\}f\big\rangle_{{x,v}}  \big|   \nonumber\\
\lesssim&\, \sum_{|\beta|\leq |\alpha|}\int_{\mathbb R^3_x
}\big| \nu^{-\frac{1}{2}}  \Gamma(\partial^{\alpha-\beta }f,\partial^\beta f)\big|_{2}   \big|\nu^{\frac{1}{2}} \partial^\alpha \{\mathbf{I}-\mathbf{P}\}f   \big|_2 {\rm d}x \nonumber\\
\lesssim&\, \sum_{|\beta|\leq |\alpha|} \int_{\mathbb R^3_x}  |\partial^{\alpha-\beta}f|_{\nu}|\partial^\beta f|_{2} |\partial^\alpha \{\mathbf{I}-\mathbf{P}\}f    |_{\nu} {\rm d}x \nonumber\\
\lesssim&\,  \kappa  \|\partial^\alpha \{\mathbf{I}-\mathbf{P}\}f\|_{\nu}^2+\sum_{|\beta|\leq k}\big\|  |\nabla^{|\beta|} f|_{2} |\nabla^{k-|\beta|}f|_{\nu} \big \|_{L^2}^2,
\end{align}
for any $k=1,2,3\dots,N$.

For the term $I_6$, by applying integration by parts and Young’s inequality, we have
\begin{align}\label{G3.13}
I_6\leq &\, \sum_{|\beta|=1}\|\partial^{\alpha+\beta}f\|_{L_{x,v}^2} \|\partial^{\alpha-\beta}E\cdot v\sqrt{M}\|_{L_{x,v}^2}\nonumber\\
\leq&\, \kappa \sum_{|\beta|=1}\|\nabla^{|\alpha|+|\beta|} f\|_{L_{x,v}^2}^2+C\sup_{0\leq t\leq T_1}\|E\|_{\dot H^{k-1}}^2,    
\end{align}
for $k=1,2,3,\dots,N-1$. Moreover, when $k = N$, we have
\begin{align}\label{G3.14}
I_6\leq  \kappa     \|\nabla^{N} f\|_{L_{x,v}^2}^2+C\sup_{0\leq t\leq T_1}\|E\|_{ \dot H^{N}}^2. 
\end{align}

Similarly, for the terms $I_7$ and $I_8$, we  have
\begin{align}\label{G3.15}
I_7+I_8\leq&\,      \kappa \sum_{|\beta|=1}\|\nabla^{|\alpha|+|\beta|} f\|_{L_{x,v}^2}^2+C\sup_{0\leq t\leq T_1}\|(E\cdot v f,E\cdot\nabla_vf)\|_{L_v^2(\dot H^{k-1})}^2  \nonumber\\
\leq&\,   \kappa \sum_{|\beta|=1}\|\nabla^{|\alpha|+|\beta|} f\|_{L_{x,v}^2}^2+C\sup_{0\leq t\leq T_1}\|(E\cdot v f,E\cdot\nabla_vf)\|_{L_{x,v}^2 }^2\nonumber\\
&+C\sup_{0\leq t\leq T_1}\|(E\cdot v f,E\cdot\nabla_vf)\|_{L_v^2(\dot H^1\cap \dot H^{N-2}) }^2 \nonumber\\
\leq&\, \kappa \sum_{|\beta|=1}\|\nabla^{|\alpha|+|\beta|} f\|_{L_{x,v}^2}^2+C\sup_{0\leq t\leq T_1} \|E\|_{\dot H^{\frac{1}{2}}}^2 \|(vf,\nabla_v f)\|_{L_v^2(\dot H^1)}^2\nonumber\\
&+C\sup_{0\leq t\leq
 T_1}\|\nabla E\|_{H^{N-3}}^2\|(vf,\nabla_v f)\|_{L_v^2(\dot H^1\cap \dot H^{N-2})}^2 \nonumber\\
 \leq&\,\kappa \sum_{|\beta|=1}\|\nabla^{|\alpha|+|\beta|} f\|_{L_{x,v}^2}^2+C\sigma \sup_{0\leq t\leq T_1}\|E\|_{ H^{N-2}}^2,
\end{align}
for $k=1,2,3,\dots, N-1$.
For $k = N$, by applying the decomposition \eqref{G1.4} and integration by parts, we derive
\begin{align}\label{G3.16}
I_7+I_8\lesssim&\,  \|Ef\|_{L_v^2(\dot H^N)}\Big( \|\nabla
^N\{\mathbf{I}-\mathbf{P}\}f\|_{\nu}+\|\mathbf{P}f\|_{L_v^2(\dot H^N)}  \Big)\nonumber\\
&+ \|\nabla^N(E\cdot\nabla_{v}f)-E\nabla^N\nabla_vf\|_{L_{x,v}^2} \|\nabla^N f\|_{L_{x,v}^2} \nonumber\\
\leq&\, \kappa    \Big( \|\nabla^{N} f\|_{L_{x,v}^2}^2+\|\nabla^N\{\mathbf{I}-\mathbf{P}\}f\|_{\nu}^2\Big)+C\sup_{0\leq t\leq T_1} \|E\|_{\dot H^{\frac{1}{2}}}^2 \|\nabla_v f\|_{L_v^2(\dot H^1)}^2\nonumber\\
&+ C{\sup_{0\leq t\leq
 T_1}\|\nabla E\|_{H^{N-1}}^2\| \nabla_v f\|_{L_v^2(\dot H^1\cap \dot H^{N-1})}^2} \nonumber\\
 \leq&\,  \kappa    \Big( \|\nabla^{N} f\|_{L_{x,v}^2}^2+\|\nabla^N\{\mathbf{I}-\mathbf{P}\}f\|_{\nu}^2\Big)+C\sigma \sup_{0\leq t\leq T_1}\|E\|_{  H^{N }}^2,
\end{align}
where we used the fact that
\begin{align*}
\int_{\mathbb{R}^3_{x}} \langle E\cdot\nabla^N f,\nabla_{v}\nabla^N f\rangle{\rm d}x=-\frac{1}{2}\int_{\mathbb R^3_{x}}\!\int_{\mathbb R^3_{v}}\nabla_{v}E |\nabla^N f|^2{\rm d}x{\rm d}v=0.
\end{align*}

Putting the estimates \eqref{G3.12}, \eqref{G3.13}, \eqref{G3.14}, \eqref{G3.15} and \eqref{G3.16} into \eqref{G3.11} gives rise to
\begin{align}\label{G3.17}
 &\frac{{\rm d}}{{\rm d}t}\|\nabla^k  f \|_{L_{x,v}^2 }^2+\lambda_3 \|\nabla^k\{\mathbf{I}-\mathbf{P}\}f \|_{\nu}^2    \nonumber\\
&\quad\lesssim \sum_{|\beta|\leq k} \big\|  \big|\nabla^{|\beta|}f \big|_2 \big|  \nabla^{k-|\beta|}f \big|_{\nu} \big \|_{L^2}^2+\sup_{0\leq t\leq T_1}\|E\|_{ H^{N-2}}^2+ \kappa  \|\nabla^{k+1}f \|_{L_{x,v}^2}^2,   
\end{align}
for some constant $\lambda_3>0$ and for any $k=1,2,3,\dots,N-1$,  and 
\begin{align}\label{G3.18}
 &\frac{{\rm d}}{{\rm d}t}\|\nabla^N  f \|_{L_{x,v}^2 }^2+\lambda_4 \|\nabla^N\{\mathbf{I}-\mathbf{P}\}f \|_{\nu}^2    \nonumber\\
&\quad\lesssim \sum_{|\beta|\leq N} \big\|  \big|\nabla^{|\beta|}f \big|_2 \big|  \nabla^{N-|\beta|}f \big|_{\nu} \big \|_{L^2}^2+\sup_{0\leq t\leq T_1}\|E\|_{ H^{N}}^2+ \kappa  \|\nabla^{N}f \|_{L_{x,v}^2}^2,   
\end{align}
for some constant $\lambda_4>0$. Combining \eqref{G3.17} and \eqref{G3.18}, we obtain \eqref{G3.10}.
\end{proof}

To estimate the first nonlinear term on the right-hand side of \eqref{G3.10}, we follow the  approach similar to that in \cite[Lemma 3.3]{DLN-2026}, and we can establish the following lemma. The proof is omitted here for brevity.

\begin{lem}   \label{L3.3}
For strong solutions of the problem \eqref{G1.5}, it holds that
\begin{align}\label{G3.19}
  \sum_{|\beta|\leq k} \big\|  \big|\nabla^{|\beta|}f \big|_2 \big|  \nabla^{k-|\beta|}f \big|_{\nu} \big \|_{L^2}^2 \lesssim \sigma \Big(\sum_{2\leq\ell\leq N} \|\nabla^{\ell} f\|_{L_{x,v}^2}^2+\sum_{1\leq \ell\leq N}\|\nabla^\ell
  \{\mathbf{I}-\mathbf{P}\}f\|_{\nu}^2            \Big),   
\end{align}
for any $0\leq t \leq  T_1$.
\end{lem}

To absorb the dissipation of the hydrodynamic part $\mathbf{P}f$ on the right-hand side of \eqref{G3.10}, we use the similar method in \cite[Lemma 6.1]{Gy-CPAM-2006} and \cite[Lemma 4.2]{GW-CPDE-2012} and define the following temporal energy functional $\mathfrak{E}_{k}(t)$:
 \begin{align}\label{G3.20}
 \mathfrak{E}_{k}(t):=&\,\sum_{|\alpha|=k}\int_{\mathbb R^3_x} \big(  \langle\{\mathbf{I}-\mathbf{P}\}\partial^\alpha f,\zeta_{a}\rangle\cdot\nabla \partial^\alpha a +\langle\{\mathbf{I}-\mathbf{P}\}\partial^\alpha f, \zeta_{ij}\rangle\cdot\partial_j\partial^\alpha b_{i}  \big) {\rm d}x\nonumber\\
 &+\sum_{|\alpha|=k}\int_{\mathbb R^3_x} \big(\langle\{\mathbf{I}-\mathbf{P}\}\partial^\alpha f,\zeta_{c}\rangle\cdot\nabla \partial^\alpha c+\partial^\alpha b\cdot\nabla \partial^\alpha a\big) {\rm d}x,
 \end{align}
where we have denoted $\mathbf{P}f=(a+b\cdot v+c|v|^2)\sqrt{M}$, and $\zeta,\zeta_{a}(v),\zeta_{ij}(v)$ and $\zeta_{c}(v)$ are some fixed linear combinations of the basis:
\begin{align*}
\big\{ \sqrt{M},\quad  v_{i}\sqrt{M},\quad v_iv_{j}\sqrt{M},\quad v_{i}|v|^2\sqrt{M},\quad 1\leq i,j\leq 3  \big\}.
\end{align*}
By using Young's inequality, we arrive at
\begin{align}\label{G3.21}
|\mathfrak{E}_{k}(t)|\lesssim \|\nabla^{k}f\|_{L_{x,v}^2}^2+\|\nabla^{k+1}f\|_{L_{x,v}^2}^2.
\end{align}
Then, we provide an estimate of $\mathbf{P}f$.

\begin{lem}\label{L3.4}
For strong solutions of the problem \eqref{G1.5}, there exists a positive constant $\lambda_5>0$    such that 
\begin{align}\label{G3.22}
&\frac{{\rm d}}{{\rm d}t}\mathfrak{E}_{ k}(t)+\lambda_5 \|\nabla^{k+1}\mathbf{P}f(t)\|_{L_{x,v}^2}^2\nonumber\\
&\quad\lesssim \|\nabla^k \{\mathbf{I}-\mathbf{P}\}f(t)\|_{L_{x,v}^2}^2+\|\nabla^{k+1} \{\mathbf{I}-\mathbf{P}\}f(t)\|_{L_{x,v}^2}^2+\sup_{t\geq 0}\|E(t)\|_{  H^N}^2,
\end{align}
for  any $k=1,2,3,\dots N-1$ and  any $0\leq t\leq T_1$,  where $\mathfrak{E}_{k}(t)$ is defined by \eqref{G3.20}. 
\end{lem}

\begin{proof}
Using a similar argument as presented in \cite[Lemma 6.1]{Gy-CPAM-2006}, we can deduce that
\begin{align}\label{G3.23}
&\frac{{\rm d}}{{\rm d}t} \mathfrak{E}_{k}(t)+\lambda_{6} \|\nabla^{k+1} \mathbf{P}f\|_{L_{x,v}^2}^2 \nonumber\\
&\quad\lesssim   \|\nabla^k \{\mathbf{I}-\mathbf{P}\}f \|_{L_{x,v}^2}^2+\|\nabla^{k+1} \{\mathbf{I}-\mathbf{P}\}f \|_{L_{x,v}^2}^2+\|\langle\nabla^k\Gamma(f,f),\zeta\rangle\|_{L^2}^2+\|\langle\nabla^k E\cdot v\sqrt{M},\zeta\rangle\|_{L^2}^2\nonumber\\
&\quad\quad +\|\langle\nabla^k (E\cdot vf),\zeta\rangle\|_{L^2}^2+\|\langle\nabla^k (E\cdot \nabla_vf),\zeta\rangle\|_{L^2}^2\nonumber\\
&\quad\equiv:\|\nabla^k \{\mathbf{I}-\mathbf{P}\}f \|_{L_{x,v}^2}^2+\|\nabla^{k+1} \{\mathbf{I}-\mathbf{P}\}f \|_{L_{x,v}^2}^2+\sum_{j=9}^{12}I_{j},
\end{align}
for some constant $\lambda_{6}>0$.

Since the estimate of $I_9$ is the same as in \cite[Lemma 3.4]{DLN-2026}, we can consequently obtain
\begin{align}\label{G3.24}
I_9\lesssim \sigma \|\nabla^{k+1} f\|_{L_{x,v}^2}^2, 
\end{align}
for any $k=1,2,3,\dots,N-1$.
Notice that $\zeta$ decays exponentially with respect to $v$ and can absorb any $v$-weighted term and the derivative $\nabla_v$. We directly compute that
\begin{align}\label{G3.25}
I_{10}+I_{11}+I_{12}\lesssim &\,    \sup_{0\leq t\leq T_1} \|E\|_{\dot H^k}^2 +  \sup_{0\leq t\leq T_1} \|E\cdot f\|_{L_{v}^2(\dot H^k)}^2 \nonumber\\
\lesssim&\,  \sup_{0\leq t\leq T_1} \|E\|_{\dot H^k}^2  +\sup_{0\leq t\leq T_1} \|  E\|_{  \dot H^{1}\cap \dot H^k}^2 \|f\|_{L_v^2(\dot H^1\cap\dot H^k)}^2\nonumber\\
\lesssim&\, \sup_{0\leq t\leq T_1} {\|E\|_{  H^N}^2}, 
\end{align}
  for any $k
=1,2,3,\dots,N-1$.
Substituting the estimates \eqref{G3.24}--\eqref{G3.25} into \eqref{G3.23}, we reach \eqref{G3.22}.
\end{proof}

With Lemmas \ref{L3.2}, \ref{L3.3} and \ref{L3.4} at our disposal, we establish the estimate of $f(t,x,v)$ in the $L_v^2(\dot H^1\cap \dot H^N)$-norm below.

\begin{cor}\label{cor3.5}
For strong solutions of the problem \eqref{G1.5}, it holds that
\begin{align} \label{G3.26}
\|f(t)\|_{L_v^2(\dot H^1\cap \dot H^N)}\lesssim \sup_{0\leq t\leq T_1}\Big\{\|E(t)\|_{  H^N }+\|f_{L}(t)\|_{L_v^2(\dot B_{2,\infty}^{\frac{1}{2}})} \Big\}+\|f_0\|_{L_v^2(\dot H^1\cap \dot H^N)},
\end{align}
for any $0\leq t \leq T_1$.
\end{cor}

\begin{proof}
The sum of \eqref{G3.10}, \eqref{G3.19}, and $\tau_1\times$ \eqref{G3.22} leads to 
\begin{align}\label{G3.27}
&\frac{{\rm d}}{{\rm d}t} \Big( \|f(t)\|_{L_v^2(\dot H^1\cap \dot H^N)}^2+\tau_1 \sum_{1\leq k\leq N-1}\mathfrak{E}_{k}(t)        \Big)\nonumber\\
&\quad+\lambda_{7}    \Big(\sum_{2\leq\ell\leq N}\|\nabla^\ell\mathbf{P}f\|_{L_{x,v}^2}^2+\sum_{1\leq
\ell\leq N}\|\nabla^\ell \{\mathbf{I}-\mathbf{P}\}f\|_{\nu}^2         \Big) \lesssim \sup_{0\leq t\leq T_1}\|E(t)\|_{  H^N}^2,
\end{align}
for some constant $\lambda_{7}>0$.  
Here, $\tau_1>0$ is a small enough constant chosen such that 
\begin{equation}\label{add.Esim}
 \|f(t)\|_{L_v^2(\dot H^1\cap \dot H^N)}^2+\tau_1 \sum_{1\leq k\leq N-1}\mathfrak{E}_{k}(t)\backsim \|f(t)\|_{L_v^2(\dot H^1\cap \dot H^N)}^2,
\end{equation}
owing to \eqref{G3.21}. Adding the term $\|f_{L}(t)\|_{L_v^2(\dot B_{2,\infty}^{\frac{1}{2}})}^2$ to both sides of \eqref{G3.27}, we further obtain the following Lyapunov inequality:
\begin{align*}
&\frac{{\rm d}}{{\rm d}t} \Big( \|f(t)\|_{L_v^2(\dot H^1\cap \dot H^N)}^2+\tau_1 \sum_{1\leq k\leq N-1}\mathfrak{E}_{k}(t)        \Big)\\
&\quad+\lambda_{8}\|f(t)\|_{L_v^2(\dot H^1\cap \dot H^N)}^2   \lesssim \sup_{0\leq t\leq T_1}\Big\{\|E(t)\|_{H^N}^2+\|f_{L}(t)\|_{L_v^2(\dot B_{2,\infty}^{\frac{1}{2}})}^2 \Big\}, 
\end{align*}
for some constant $\lambda_{8}>0$. In terms of \eqref{add.Esim}, according to Gr\"{o}nwall’s inequality, it then follows that there is a constant $\lambda_{9}>0$ such that
\begin{align*}
 \|f(t)\|_{L_v^2(\dot H^1\cap \dot H^N)}\lesssim &\, \sup_{0\leq t\leq T_1}\Big\{\|E(t)\|_{ H^N } +\|f_{L}(t)\|_{L_v^2(\dot B_{2,\infty}^{\frac{1}{2}})}  \Big\}\Big(\int_0^t e^{-\frac{\lambda_{9} \tau}{2} }{\rm d}\tau \Big)^\frac{1}{2}\nonumber\\
 &+e^{-\frac{\lambda_{9} t}{2}  } \|f_0\|_{L_v^2(\dot H^1\cap \dot H^N)}\nonumber\\
 \lesssim&\, \sup_{0\leq t\leq T_1}\Big\{\|E(t)\|_{H^N} +\|f_{L}(t)\|_{L_v^2(\dot B_{2,\infty}^{\frac{1}{2}})}  \Big\}+\|f_0\|_{L_v^2(\dot H^1\cap \dot H^N)}.
\end{align*}
This proves \eqref{G3.26}. Thus, we complete the proof of Corollary \ref{cor3.5}.
\end{proof}

\subsection{A priori estimates for the microscopic part}
Note the inequality 
\begin{align*}
\|\langle v\rangle \mathbf{P}f\|_{L_{x,v}^2}+\|\nabla_v \mathbf{P}f\|_{L_{x,v}^2}\lesssim \|f\|_{L_{x,v}^2}.
\end{align*}
In order to obtain the estimates of $\langle v\rangle f $ and $\nabla_v f $, we only concentrate on the estimates for
\begin{center}
$\|\nu\{\mathbf{I}-\mathbf{P}\} f\|_{L_v^2(\dot H^1\cap \dot H^{N - 1})}$, $\sum\limits_{\substack{ 1\leq |\beta|\leq N \\|\alpha|+|\beta| \leq N}}\|\partial^\alpha_{x}\partial^\beta_{v}\{\mathbf{I}-\mathbf{P}\}f\|_{L_v^2 (L^2)}$.
\end{center}

\begin{lem}\label{L3.6}
For strong solutions of the problem   \eqref{G1.5}, there exists a positive constant $\lambda_{10}>0$ such that
\begin{align} \label{G3.28}
&\frac{{\rm d}}{{\rm d}t}\|\nu\{\mathbf{I}-\mathbf{P}\}f(t) \|_{L_v^2(\dot H^1\cap\dot H^{N-1})}^2+ \lambda_{10} \|\nu^{\frac{3}{2}}\{\mathbf{I}-\mathbf{P}\}f(t)\|_{L_v^2(\dot H^1\cap\dot H^{N-1})}^2 \nonumber\\
\lesssim&\,  \|f(t)\|_{L_v^2(\dot H^2\cap\dot H^N)}^2+\sigma \|\nu^{\frac{1}{2}}\nabla_v\{\mathbf{I}-\mathbf{P}\}f(t)\|_{L_v^2(\dot H^1\cap \dot H^{N-1})}^2+\sup_{0\leq t\leq T_1}\|E(t)\|_{H^N}^2,
\end{align}
for any $0\leq t \leq T_1$.
\end{lem} 

\begin{proof}
Applying the operator $\{\mathbf{I}-\mathbf{P}\}$ to   \eqref{G1.5}$_1$ yields
\begin{align}\label{G3.29}
&\partial_t \{\mathbf{I}-\mathbf{P}\}f+v\cdot\nabla \{\mathbf{I}-\mathbf{P}\}f+E\cdot\nabla_v \{\mathbf{I}-\mathbf{P}\}f-\frac{1}{2}E\cdot v\{\mathbf{I}-\mathbf{P}\}f
+\mathcal{L}\{\mathbf{I}-\mathbf{P}\}f\nonumber\\
=&\,\Gamma(f,f)-v\cdot \nabla \mathbf{P}f+\mathbf{P}(v\cdot\nabla f)-E\cdot\nabla_v\mathbf{P}f+\frac{1}{2}E\cdot v\mathbf{P}f- \mathbf{P}\Big[ \frac{1}{2}E\cdot v f-E\cdot\nabla_v f\Big],
\end{align}
where we have used the following facts:
\begin{align*}
\{\mathbf{I}-\mathbf{P}\} E\cdot v\sqrt{M}=0,\quad \{\mathbf{I}-\mathbf{P}\}\mathcal{L}f=\mathcal{L} \{\mathbf{I}-\mathbf{P}\}f.   
\end{align*}

Applying the operator $\partial^{\alpha}$ with $1\leq|\alpha|\leq N-1$ to \eqref{G3.29} and then taking the $L^2$ inner product of the resulting identity with $\nu^2\partial^\alpha \{\mathbf{I}-\mathbf{P}\}f$ leads to
\begin{align}\label{G3.30}
&\frac{1}{2}\frac{{\rm d}}{{\rm d}t}\|\nu\{\mathbf{I}-\mathbf{P}\}\partial^\alpha f\|_{L_{x,v}^2}^2+ \langle \nu^2 \mathcal{L}\{\mathbf{I}-\mathbf{P}\}\partial^\alpha f,  \{\mathbf{I}-\mathbf{P}\} \partial^\alpha f \rangle_{x,v} \nonumber\\
=&\,\langle \nu\partial^\alpha\Gamma(f,f), \nu \{\mathbf{I}-\mathbf{P}\}\partial^\alpha f \rangle_{x,v}+\big\langle \partial^\alpha\big[\mathbf{P}(v\cdot\nabla f)-v\cdot\nabla \mathbf{P}f\big],\nu^2 \{\mathbf{I}-\mathbf{P}\} \partial^\alpha f\rangle_{x,v}\nonumber\\
&+  \Big\langle  \partial^\alpha\mathbf{P}\Big[ \frac{1}{2}E\cdot vf-E\cdot\nabla_v f  \Big], \nu^2\{\mathbf{I}-\mathbf{P}\}\partial^\alpha f  \Big\rangle_{x,v}\nonumber\\
&+\Big\langle  \partial^\alpha \Big[ \frac{1}{2}E\cdot v\mathbf{P}f-E\cdot\nabla_v \mathbf{P}f  \Big], \nu^2\{\mathbf{I}-\mathbf{P}\}\partial^\alpha f  \Big\rangle_{x,v}\nonumber\\
&-\Big\langle  \partial^\alpha \Big[ \frac{1}{2}E\cdot v 
\{\mathbf{I}- \mathbf{P}\}f-E\cdot\nabla_v\{ \mathbf{I}- \mathbf{P} \}f  \Big], \nu^2\{\mathbf{I}-\mathbf{P}\}\partial^\alpha f  \Big\rangle_{x,v}\nonumber\\
\equiv:&\, \sum_{i=1}^5 J_{i}.
\end{align}
Thanks to \eqref{G2.3} in Proposition \ref{prop2.2}, we have
\begin{align}\label{G3.31}
 \langle \nu^2 \mathcal{L}\partial^{\alpha}\{\mathbf{I}-\mathbf{P}\}f
,\partial^{\alpha}\{\mathbf{I}-\mathbf{P}\}f \rangle_{x,v}\geq \frac{1}{2} \|\nu\partial^\alpha\{\mathbf{I}-\mathbf{P}\}f\|_{\nu}^2-C\|\partial^\alpha \{\mathbf{I}-\mathbf{P}\}f\|_{\nu}^2.
\end{align}
For the first term $J_1$, by using \eqref{NJKG2.6} in Proposition \ref{NJKP2.5}, H\"{o}lder's and Sobolev's inequalities and Lemma \ref{LA.1}, we arrive at
\begin{align*}
J_1\lesssim&\, \sum_{|\beta|\leq |\alpha|} \int_{\mathbb R^3_{x}}|\nu \partial^\beta f|_{2} |\nu \partial^{\alpha-\beta}f|_{\nu}  |\nu \partial^{\alpha}\{\mathbf{I}-\mathbf{P}\}f|_{\nu} {\rm d}x \nonumber\\
\lesssim&\, \sum_{|\beta|\leq |\alpha|}  \big\| \nu  \partial^\beta f|_{2} |\nu  \partial^{\alpha-\beta} f|_{\nu} \big\|_{L^2}^2+\frac{1}{9}\|\nu \partial^{\alpha}\{\mathbf{I}-\mathbf{P}\}f\|_{\nu}^2.   
\end{align*}
Owing to the macro-micro decomposition \eqref{G1.4}, one gets
\begin{align*}
 &\sum_{|\beta|\leq |\alpha|}  \big\| \nu   \partial^\beta f|_{2} |\nu  \partial^{\alpha-\beta} f|_{\nu} \big\|_{L^2}^2\nonumber\\
\lesssim&\, \sum_{|\beta|\leq |\alpha|}  \big\|  \big|\nu \partial^{\beta}f \big|_2 \big|  \partial^{\alpha-\beta}f \big|_{2} \big \|_{L^2}^2+\sum_{|\beta|\leq |\alpha|}\big\|  \big|\nu \partial^{\beta}f \big|_2 \big|  \nu \partial^{\alpha-\beta}\{\mathbf{I}-\mathbf{P}\}f \big|_{2} \big \|_{L^2}^2  \nonumber\\
\lesssim&\, \sum_{|\beta|\leq |\alpha|}  \big\|  \big| \partial^{\beta}f \big|_2 \big|  \partial^{\alpha-\beta}f \big|_{2} \big \|_{L^2}^2 + \sum_{|\beta|\leq |\alpha|}  \big\|  \big|\nu  \partial^{\beta}\{\mathbf{I}-\mathbf{P}\} f \big|_2 \big|  \partial^{\alpha-\beta}f \big|_{2} \big \|_{L^2}^2\nonumber\\
&+\sum_{|\beta|\leq |\alpha|}\big\|  \big|\nu \partial^{\beta}f \big|_2 \big|  \nu \partial^{\alpha-\beta}\{\mathbf{I}-\mathbf{P}\}f \big|_{2} \big \|_{L^2}^2  \nonumber\\
\equiv:&\, \sum_{i=6}^8J_{i}.
\end{align*}

Using a similar argument as in \cite[Lemma 3.3]{DLN-2026}, we can infer that
\begin{align*}
J_6\lesssim  \sigma \Big(\sum_{2\leq\ell\leq N-1} \|\nabla^{\ell} f\|_{L_{x,v}^2}^2+\sum_{1\leq \ell\leq N-1}\|\nabla^\ell
  \{\mathbf{I}-\mathbf{P}\}f\|_{\nu}^2            \Big).   
\end{align*}
For the term $J_8$, applying Lemma \ref{LA.1}, one has
\begin{align*}
J_8\lesssim &\, \|\nu f\|_{L_v^2(L^\infty)}^2 \|\nu \{\mathbf{I}-\mathbf{P}\}f\|_{L_v^2(\dot H^{N-1})}^2+  \| f\|_{L_v^2(\dot H^{N-1})}^2 \|\nu\{\mathbf{I}-\mathbf{P}\}f\|_{L_v^2(L^\infty)}^2 \nonumber\\
&+\|\nu  f\|_{L_v^2(\dot H^1\cap\dot H^{N-1})}^2  \|\nu\{\mathbf{I}-\mathbf{P}\}f\|_{L_v^2(\dot H^1\cap\dot H^{N-1})}^2 \nonumber\\
\lesssim&\,  \|\nu f\|_{L_v^2(\dot H^1\cap \dot H^2)}^2 \|\nu \{\mathbf{I}-\mathbf{P}\}f\|_{L_v^2(\dot H^{N-1})}^2+  \|\nu f\|_{L_v^2(\dot H^{N-1})}^2 \|\nu\{\mathbf{I}-\mathbf{P}\}f\|_{L_v^2(\dot H^1\cap\dot H^{2})}^2 \nonumber\\
&+\|\nu  f\|_{L_v^2(\dot H^1\cap\dot H^{N-1})}^2  \|\nu\{\mathbf{I}-\mathbf{P}\}f\|_{L_v^2(\dot H^1\cap\dot H^{N-1})}^2 \nonumber\\
\lesssim&\, \sigma \|\nu\{\mathbf{I}-\mathbf{P}\}f\|_{L_v^2(\dot H^1\cap\dot H^{N-1})}^2,
\end{align*}
which implies
\begin{align*}
J_7\lesssim J_8\lesssim  \sigma \|\nu\{\mathbf{I}-\mathbf{P}\}f\|_{L_v^2(\dot H^1\cap\dot H^{N-1})}^2.   
\end{align*}
Consequently, collecting the estimates $J_6$, $J_7$ and $J_8$ above, we have
\begin{align}\label{G3.32}
J_1\lesssim     \sigma  \sum_{2\leq\ell\leq N-1} \|\nabla^{\ell} f\|_{L_{x,v}^2}^2 +\Big( \sigma+\frac{1}{9}\Big) \|\nu^{\frac{3}{2}}\{\mathbf{I}-\mathbf{P}\}f\|_{L_v^2(\dot H^1\cap\dot H^{N-1})}^2.
\end{align}
Since $\langle v\rangle^k\sqrt{M}\lesssim 1$ for any $k\geq0$, by direct calculation, we have
\begin{align}\label{G3.33}
J_2\leq       C\|f\|_{L_v^2(\dot H^2\cap\dot H^{N})}^2 +\frac{1}{9}\|\nu \partial^{\alpha}\{\mathbf{I}-\mathbf{P}\}f\|_{\nu}^2.
\end{align}
Similarly, we also derive
\begin{align}\label{G3.34}
J_3+J_4\leq&\,   C\sup_{0\leq t\leq T_1}\|Ef\|_{L_v^2(\dot H^1\cap\dot H^{N-1})}^2 +\frac{1}{9}\|\nu \partial^{\alpha}\{\mathbf{I}-\mathbf{P}\}f\|_{\nu}^2  \nonumber\\
\leq&\,C\sup_{0\leq t\leq T_1}\|E\|_{\dot H^1\cap \dot H^{N-1}}^2 \|f\|_{L_v^2(\dot H^1\cap \dot H^{N-1})}^2+\frac{1}{9}\|\nu \partial^{\alpha}\{\mathbf{I}-\mathbf{P}\}f\|_{\nu}^2 \nonumber\\
\leq &\, C  \sigma\sup_{0\leq t\leq T_1} \|E\|_{H^N}^2+\frac{1}{9}\|\nu \partial^{\alpha}\{\mathbf{I}-\mathbf{P}\}f\|_{\nu}^2 ,
\end{align}
and 
\begin{align}\label{G3.35}
 J_5\leq&\,   C\|E\cdot {\nu}^{\frac{3}{2}}\{\mathbf{I}-\mathbf{P}\}f\|_{L_v^2(\dot H^1\cap \dot H^{N-1})}^2+ C\|E\cdot {\nu}^{\frac{1}{2}} \nabla_v\{\mathbf{I}-\mathbf{P}\}f\|_{L_v^2(\dot H^1\cap \dot H^{N-1})}^2\nonumber\\
& +\frac{1}{9}\|\nu \partial^{\alpha}\{\mathbf{I}-\mathbf{P}\}f\|_{\nu}^2 \nonumber\\
 \leq&\, C\|E\|_{\dot H^1\cap \dot H^{N-1}}^2 \Big( \|\nu^{\frac{3}{2}}\{\mathbf{I}-\mathbf{P}\}f\|_{L_v^2(\dot H^1\cap \dot H^{N-1})}^2 +\|\nu^{\frac{1}{2}}\nabla_v\{\mathbf{I}-\mathbf{P}\}f\|_{L_v^2(\dot H^1\cap \dot H^{N-1})}^2\Big) \nonumber\\
 &+\frac{1}{9}\|\nu \partial^{\alpha}\{\mathbf{I}-\mathbf{P}\}f\|_{\nu}^2\nonumber\\
 \leq&\,  C  \Big( \sigma+\frac{1}{9}\Big) \|\nu^\frac{3}{2}\{\mathbf{I}-\mathbf{P}\}f\|_{L_v^2(\dot H^1\cap\dot H^{N-1})}^2+\sigma \|\nu^{\frac{1}{2}}\nabla_v\{\mathbf{I}-\mathbf{P}\}f\|_{L_v^2(\dot H^1\cap \dot H^{N-1})}^2.
\end{align}
Putting the estimates \eqref{G3.31}, \eqref{G3.32}, \eqref{G3.33}, \eqref{G3.34} and \eqref{G3.35} into \eqref{G3.30}, we end up with \eqref{G3.28}.
\end{proof}

\begin{lem}\label{L3.7}
For strong solutions of the problem   \eqref{G1.5}, there exists a positive constant $\lambda_{11}>0$ such that
\begin{align} \label{G3.36}
&\frac{{\rm d}}{{\rm d}t} \sum_{\substack{ 1\leq |\beta|\leq N \\|\alpha|+|\beta| \leq N}}\|\partial^\alpha_{x}\partial^\beta_{v}\{\mathbf{I}-\mathbf{P}\}f(t) \|_{L_{x,v}^2 }^2+ \lambda_{11}\sum_{\substack{ 1\leq |\beta|\leq N \\|\alpha|+|\beta| \leq N}} \| \partial^{\alpha}_{x}\partial^\beta_{v}\{\mathbf{I}-\mathbf{P}\}f(t)\|_{\nu}^2 \nonumber\\
\lesssim&\, \|\nu^{\frac{1}{2}}\{\mathbf{I}-\mathbf{P}\}f(t)\|_{L_v^2(   H^{N-1})}^2+ \|f(t)\|_{L_v^2(\dot H^{\frac{3}{4}}\cap\dot H^N)}^2,
\end{align}
for any $0\leq t \leq T_1$.
\end{lem} 

\begin{proof}
We apply a similar method used in \cite[Theorem 3.1]{Gy-CPAM-2002}. We first fix $k$ such that $1\leq k \leq N$ and choose $|\alpha|+|\beta|\leq N$. By applying the derivative $\partial_{x}^{\alpha}\partial^{\beta}_v$  to \eqref{G3.29}, multiplying the result by $\partial^{\alpha}_{x}\partial^{\beta}_v\{\mathbf{I}-\mathbf{P}\}f$, and taking the integration of the resulting identity, we finally arrive at
\begin{align}\label{G3.37}
&\frac{1}{2}\frac{{\rm d}}{{\rm d}t} \|\partial^\alpha_{x}\partial^\beta_{v}\{\mathbf{I}-\mathbf{P}\}f\|_{L_{x,v}^2}^2+  \langle \partial^\alpha_{x} \partial^\beta_{v} \mathcal{L}\{\mathbf{I}-\mathbf{P}\}f, \partial^\alpha_{x} \partial^\beta_{v}  \{\mathbf{I}-\mathbf{P}\}f   \rangle _{x,v} \nonumber\\
=&\, \langle  \partial^\alpha_{x} \partial^\beta_{v} \Gamma(f,f),  \partial^\alpha_{x} \partial^\beta_{v} \{\mathbf{I-\mathbf{P}}\}f \rangle_{x,v}+\big\langle \partial^\alpha_{x}\partial^\beta_{v}\big[\mathbf{P}(v\cdot\nabla f)-v\cdot\nabla \mathbf{P}f\big],\partial^\alpha_{x}\partial^\beta_{v} \{\mathbf{I}-\mathbf{P}\} f\rangle_{x,v} \nonumber\\
 &+  \Big\langle  \partial^\alpha_{x}\partial^\beta_{v}\Big[\mathbf{P}\Big( \frac{1}{2}E\cdot vf-E\cdot\nabla_v f  \Big)\Big], \partial^\alpha_{x}\partial^\beta_{v}\{\mathbf{I}-\mathbf{P}\} f  \Big\rangle_{x,v}\nonumber\\
&+\Big\langle \partial^\alpha_{x}\partial^\beta_{v}\Big[ \frac{1}{2}E\cdot v\mathbf{P}f-E\cdot\nabla_v \mathbf{P}f  \Big], \partial^\alpha_{x}\partial^\beta_{v}\{\mathbf{I}-\mathbf{P}\}  f  \Big\rangle_{x,v}\nonumber\\
&-\Big\langle  \partial^\alpha_{x}\partial^\beta_{v} \Big[ \frac{1}{2}E\cdot v 
\{\mathbf{I}- \mathbf{P}\}f-E\cdot\nabla_v\{ \mathbf{I}- \mathbf{P} \}f  \Big],  \partial^\alpha_{x}\partial^\beta_{v}\{\mathbf{I}-\mathbf{P}\} f  \Big\rangle_{x,v} \nonumber\\
\equiv:&\, \sum_{i=9}^{13}J_{i}.
\end{align}
Due to \eqref{G2.3} in Proposition \ref{prop2.2}, we reach
\begin{align}\label{G3.38}
 \langle  \partial^\alpha_{x}\partial^\beta_{v}\mathcal{L} \{\mathbf{I}-\mathbf{P}\}f
,\partial^\alpha_{x}\partial^\beta_{v}\{\mathbf{I}-\mathbf{P}\}f \rangle_{x,v}\geq \frac{1}{2} \| \partial^\alpha_{x}\partial^\beta_{v}\{\mathbf{I}-\mathbf{P}\}f\|_{\nu}^2-C\|\partial^\alpha_{x} \{\mathbf{I}-\mathbf{P}\}f\|_{\nu}^2.
\end{align}
For the term $J_9$, by virtue of \eqref{NJKGnew} in Proposition \ref{prop2.4}, one  derives
\begin{align}\label{G3.39}
J_9\leq&\,\sum_{\substack{\alpha_1\leq\alpha,\,\beta_1+\beta_{2}\leq\beta \\|\alpha|+|\beta| \leq N}} \int_{\mathbb R^3_{x}} | \partial^{\alpha_1}_{x}\partial^{\beta_1}_{v} f |_{2} |\partial^{\alpha_2}_{x}\partial^{\beta_2}_{v} f   |_{\nu} |\partial^\alpha_{x}\partial^\beta_{v}\{\mathbf{I}-\mathbf{P}\}f|_{\nu} {\rm d}x   \nonumber\\
\leq&\,\sum_{\substack{\alpha_1\leq\alpha,\,\beta_1+\beta_{2}\leq\beta \\|\alpha|+|\beta| \leq N}}  \big\| |\partial^{\alpha_1}_{x}\partial^{\beta_1}_{v} f|_{2} |\partial^{\alpha_2}_{x}\partial^{\beta_2}_{v} f|_{\nu} \big\|_{L^2}^2+\frac{1}{9}\| \partial^{\alpha}_{x}\partial^\beta_{v}\{\mathbf{I}-\mathbf{P}\}f\|_{\nu}^2\nonumber\\
\leq&\, \sum_{\substack{\alpha_1\leq\alpha,\,\beta_1+\beta_{2}\leq\beta \\|\alpha|+|\beta| \leq N}} \Big( \big\||\partial^{\alpha_1}_{x}\partial^{\beta_1}_{v} f|_{2} |\partial^{\alpha_2}_{x}\partial^{\beta_2}_{v} f|_{2} \big\|_{L^2}^2+  \big\||\partial^{\alpha_1}_{x}\partial^{\beta_1}_{v} f|_{2} |\partial^{\alpha_2}_{x}\partial^{\beta_2}_{v} \{\mathbf{I}-\mathbf{P}\}f|_{\nu} \big\|_{L^2}^2\Big)\nonumber\\
&+\frac{1}{9}\| \partial^{\alpha}_{x}\partial^\beta_{v}\{\mathbf{I}-\mathbf{P}\}f\|_{\nu}^2 \nonumber\\
\leq&\,C\sigma\|\mathbf{P}f\|_{L_v^2(L^4)}^2+C\sigma\|\mathbf{P}f\|_{L_v^2(\dot H^1\cap\dot H^N)} ^2 +\Big( C\sigma+\frac{1}{9}\Big) \sum_{\substack{1\leq|\beta |\leq N\\|\alpha|+|\beta |\leq N}}\| \partial^{\alpha}_{x}\partial^\beta_{v}\{\mathbf{I}-\mathbf{P}\}f\|_{\nu}^2\nonumber\\
 \leq&\,C\| f\|_{L_v^2(\dot H^\frac{3}{4}\cap\dot H^N)}^2  +\Big( C\sigma+\frac{1}{9}\Big) \sum_{\substack{1\leq|\beta |\leq N\\|\alpha|+|\beta |\leq N}}\| \partial^{\alpha}_{x}\partial^\beta_{v}\{\mathbf{I}-\mathbf{P}\}f\|_{\nu}^2,
\end{align}
where we utilized the embedding $\dot H^{\frac{3}{4}}\hookrightarrow  L^4$ in Lemma \ref{LA.2}.

For the terms $J_{10}$, $J_{11}$, and $J_{12}$, noting the fact that $\|\partial^{\alpha}_{x}\partial^\beta_{v} \mathbf{P}f\|_{L_{x,v}^2}\lesssim \|\partial^\alpha_{x}\mathbf{P}f\|_{L_{x,v}^2}$, we have
\begin{align}\label{G3.40}
J_{10}+J_{11}+J_{12}\leq &\, C\sum_{ |\alpha|\leq N-1}\|\nabla\partial^\alpha f\|_{L_{x,v}^2}^2+  C\sup_{0\leq t\leq T_1}\|Ef\|_{L_v^2(H^{N-1})}^2+  \frac{1}{9}  \| \partial^{\alpha}_{x}\partial^\beta_{v}\{\mathbf{I}-\mathbf{P}\}f\|_{\nu}^2\nonumber\\
\leq&\, C\|f\|_{L_v^2(\dot H^1\cap \dot H^N)}^2+C\sup_{0\leq t\leq T_1} \|E\|_{\dot H^\frac{1}{2}}^2 \|f\|_{L_v^2(\dot H^1)}^2\nonumber\\
&+C\sup_{0\leq t\leq T_1}\|E\|_{\dot H^1\cap\dot H^{N-1}}^2 \|f\|_{L_v^2(\dot H^1\cap\dot H^{N-1})}^2+\frac{1}{9}  \| \partial^{\alpha}_{x}\partial^\beta_{v}\{\mathbf{I}-\mathbf{P}\}f\|_{\nu}^2\nonumber\\
\leq&\, C\|f\|_{L_v^2(\dot H^1\cap \dot H^N)}^2+\Big(\frac{1}{9} + C\sigma\Big)\| \partial^{\alpha}_{x}\partial^\beta_{v}\{\mathbf{I}-\mathbf{P}\}f\|_{\nu}^2.
\end{align}
For the remaining term $J_{13}$, by using Lemmas \ref{LA.1}--\ref{LA.3}, we obtain 
\begin{align}\label{G3.41}
J_{13} \lesssim&\, \sum_{|\alpha^\prime|<|\alpha|}\int_{\mathbb{R}^3\times{\mathbb{R}^3}}|\partial^{\alpha-\alpha^{\prime}}E||\nabla_v\partial_x^{\alpha^{\prime}}\partial^\beta_v\{\mathbf{I}-\mathbf{P}\}f|_2|\partial_x^{\alpha}\partial^\beta_v\{\mathbf{I}-\mathbf{P}\}f|_2\mathrm{d}x\mathrm{d}v\nonumber\\
&+\sum_{|\alpha^\prime|<|\alpha|}\int_{\mathbb{R}^3\times{\mathbb{R}^3}} |\partial^{\alpha-\alpha^{\prime}}E||\partial_x^{\alpha^{\prime}}\partial^\beta_v (v\{\mathbf{I}-\mathbf{P}\}f)|_2|\partial_x^{\alpha}\partial^\beta_v\{\mathbf{I}-\mathbf{P}\}f|_2\mathrm{d}x  {\mathrm{d}v} \nonumber\\
\lesssim&\, \sup_{0\leq t\leq T_1}\|E\|_{H^N}\bigg(\|\nu^{\frac{1}{2}}\{\mathbf{I}-\mathbf{P}\}f \|_{L_v^2(   H^{N-1})}^2+\sum_{\substack{1\leq|\beta^{\prime}|\leq N\\|\alpha^{\prime}+|\beta^{\prime}|\leq N}}\|\partial^{\alpha^{\prime}}_x\partial^{\beta^{\prime}}_v\{\mathbf{I}-\mathbf{P}\}f\|_{\nu}^2\bigg),\nonumber\\
\lesssim&\,\sigma\|\nu^{\frac{1}{2}}\{\mathbf{I}-\mathbf{P}\}f \|_{L_v^2(   H^{N-1})}^2+\sigma\sum_{\substack{1\leq|\beta^{\prime}|\leq N\\|\alpha^{\prime}|+|\beta^{\prime}|\leq N}}\|\partial^{\alpha^{\prime}}_x\partial^{\beta^{\prime}}_v\{\mathbf{I}-\mathbf{P}\}f\|_{\nu}^2.   
\end{align}
Putting together all the estimates \eqref{G3.38}, \eqref{G3.39}, \eqref{G3.40} and \eqref{G3.41} back to \eqref{G3.37}, and taking the sum over $|\beta| = k$ and $|\alpha| + |\beta| \leq N$ yields \eqref{G3.36}. This completes the proof of Lemma \ref{L3.7}.
\end{proof}

Finally, to absorb the zero-th order term $\|  \{\mathbf{I}-\mathbf{P}\}f\|_{\nu} $ on the right-hand side of \eqref{G3.36}, we need the following estimate.

 \begin{lem}\label{L3.8}
For strong solutions of the problem   \eqref{G1.5}, there exists a positive constant $\lambda_{12}>0$ such that
\begin{align} \label{G3.42}
&\frac{{\rm d}}{{\rm d}t}\| \{\mathbf{I}-\mathbf{P}\}f(t) \|_{L_{x,v}^2 }^2+ \lambda_{12} \| \{\mathbf{I}-\mathbf{P}\}f(t)\|_{\nu}^2  \lesssim \|f(t)\|_{L_v^2(\dot H^{\frac{3}{4}}\cap\dot H^1)}^2,
\end{align}
for any $0\leq t \leq T_1$.
\end{lem} 

\begin{proof}
Multiplying \eqref{G3.29} by $\{\mathbf{I}-\mathbf{P}\}f$, and then integrating the resulting identity over $\mathbb R^3_{x}\times \mathbb R^3_{v}$, we have
\begin{align}\label{G3.43}
&\frac{1}{2}\frac{{\rm d}}{{\rm d}t} \|\{\mathbf{I}-\mathbf{P}\}f\|_{L_{x,v}^2}^2+ \langle\mathcal{L}\{\mathbf{I}-\mathbf{P}\}f, \{\mathbf{I}-\mathbf{P}\}f\rangle _{x,v} \nonumber\\
=&\,\langle \Gamma(f,f),\{\mathbf{I}-\mathbf{P}\}f \rangle_{x,v}+ \langle  \mathbf{P}(v\cdot\nabla f)-v\cdot\nabla \mathbf{P}f ,  \{\mathbf{I}-\mathbf{P}\}   f\rangle_{x,v} \nonumber\\
&+  \Big\langle   \mathbf{P}\Big[ \frac{1}{2}E\cdot vf-E\cdot\nabla_v f  \Big],  \{\mathbf{I}-\mathbf{P}\}  f  \Big\rangle_{x,v} +\Big\langle     \frac{1}{2}E\cdot v\mathbf{P}f-E\cdot\nabla_v \mathbf{P}f  ,  \{\mathbf{I}-\mathbf{P}\}  f  \Big\rangle_{x,v}\nonumber\\
&-\Big\langle    \frac{1}{2}E\cdot v 
\{\mathbf{I}- \mathbf{P}\}f    ,  \{\mathbf{I}-\mathbf{P}\}  f  \Big\rangle_{x,v}\nonumber\\
\equiv:&\, \sum_{i=14}^{18}J_{i}.
\end{align}
Making use of \eqref{G2.2} in Proposition \ref{prop2.2}, one has
\begin{align} \label{G3.44}
 \langle\mathcal{L}\{\mathbf{I}-\mathbf{P}\}f, \{\mathbf{I}-\mathbf{P}\}f\rangle _{x,v}    \geq \kappa_0 \|\{\mathbf{I}-\mathbf{P}\}f\|_{\nu}^2.
\end{align}
It follows from the decomposition \eqref{G1.4}, \eqref{G2.5} in Proposition \ref{prop2.3}, Lemma \ref{LA.1} and Young's inequality that
\begin{align}\label{G3.45}
J_{14}\leq&\, C\int_{\mathbb R^3_{x}} |\nu^\frac{1}{2} f|_{2}|f|_{2} |\nu^\frac{1}{2}\{\mathbf{I}-\mathbf{P}\}f|_{2} {\rm d}x   \nonumber\\
\leq&\,C \int_{\mathbb R^3_{x}} \big( |\nu^\frac{1}{2}\{\mathbf{I}-\mathbf{P}\}f|_{2}+|\mathbf{P}f|_{2}\big) |f|_{2} |\nu^\frac{1}{2}\{\mathbf{I}-\mathbf{P}\}f|_{2} {\rm d}x  \nonumber\\
\leq&\, C \|f\|_{L^\infty_{x}L_v^2} \|\{\mathbf{I}-\mathbf{P}\}f\|_{\nu}^2+C\|f\|_{L^4_{x}(L^2_{v})}^2\|\{\mathbf{I}-\mathbf{P}\}f\|_{\nu} \nonumber\\
\leq&\, C\|f\|_{L_v^2(\dot H^1\cap \dot H^2)}\|\{\mathbf{I}-\mathbf{P}\}f\|_{\nu}^2+\frac{\kappa_0}{6}\|\{\mathbf{I}-\mathbf{P}\}f\|_{\nu}^2+C\|f\|_{L_v^2(\dot H^{\frac{3}{4}})}^4\nonumber\\
\leq&\, \Big(C\sigma+\frac{\kappa_0}{6}\Big) \|\{\mathbf{I}-\mathbf{P}\}f\|_{\nu}^2+C\|f\|_{L_v^2(\dot H^\frac{3}{4}\cap\dot H^1)}^2.
\end{align}
By direct calculation, we have
\begin{align}\label{G3.46}
J_{15}\leq C\|\nabla f\|_{L_{x,v}^2}^2+\frac{\kappa_0}{6}\|\{\mathbf{I}-\mathbf{P}\}f\|_{\nu}^2.    
\end{align}
Similarly, 
\begin{align}\label{G3.47}
J_{16}+J_{17} \leq&\, C\sup_{0\leq t\leq{T_1}}\|Ef\|_{L_{x,v}^2}\|\{\mathbf{I}-\mathbf{P}\}f\|_{\nu}\nonumber\\
\leq&\,C \sup_{0\leq t\leq{T_1}} {\|E\|_{L^3}}\|f\|_{L_{v}^2(L^6)}\|\{\mathbf{I}-\mathbf{P}\}f\|_{\nu}\nonumber\\
\leq&\, C \|f\|_{L_v^2( \dot H^1)}^2+\frac{\kappa_0}{6}\|\{\mathbf{I}-\mathbf{P}\}f\|_{\nu}^2,
\end{align}
and
\begin{align}\label{G3.48}
J_{18}\lesssim    \sup_{0\leq t\leq T_1}\|E\|_{L^\infty} \|\{\mathbf{I}-\mathbf{P}\}f\|_{\nu}^2 \lesssim \sigma \|\{\mathbf{I}-\mathbf{P}\}f\|_{\nu}^2.  
\end{align}
Inserting all the estimates \eqref{G3.44}, \eqref{G3.45}, \eqref{G3.46}, \eqref{G3.47} and \eqref{G3.48} into \eqref{G3.43}, we obtain the desired \eqref{G3.42}. This completes the proof of Lemma \ref{L3.8}.
\end{proof}

With Lemmas \ref{L3.6}--\ref{L3.8} at hand, by combining \eqref{G3.28}, \eqref{G3.36} and \eqref{G3.42}, we present the estimate of the microscopic part $\{\mathbf{I}-\mathbf{P}\}f$ as follows.

\begin{cor}\label{cor3.9}
For strong solutions of the problem   \eqref{G1.5}, there exists a positive constant $\lambda_{13}>0$ such that
\begin{align}\label{G3.49}
&\frac{{\rm d}}{{\rm d}t}\bigg(\|\nu\{\mathbf{I}-\mathbf{P}\}f(t) \|_{L_v^2(\dot H^1\cap\dot H^{N-1})}^2 +\|\{\mathbf{I}-\mathbf{P}\}f(t)\|_{L_{x,v}^2}^2+ \sum_{\substack{ 1\leq |\beta|\leq N \\|\alpha|+|\beta| \leq N}}\|\partial^\alpha_{x}\partial^\beta_{v}\{\mathbf{I}-\mathbf{P}\}f(t) \|_{L_{x,v}^2 }^2\bigg)
 \nonumber\\
&+ \lambda_{13} \bigg(\|\nu^{\frac{3}{2}}\{\mathbf{I}-\mathbf{P}\}f(t)\|_{L_v^2(\dot H^1\cap\dot H^{N-1})}^2+ \|\{\mathbf{I}-\mathbf{P}\}f(t)\|_{\nu}^2+\sum_{\substack{ 1\leq |\beta|\leq N \\|\alpha|+|\beta| \leq N}} \| \partial^{\alpha}_{x}\partial^\beta_{v}\{\mathbf{I}-\mathbf{P}\}f(t)\|_{\nu}^2\bigg)\nonumber\\    
&\quad\lesssim \|f(t)\|_{L_v^2(\dot H^\frac{3}{4}\cap\dot H^N)}^2+\|\nu^\frac{1}{2}\{\mathbf{I}-\mathbf{P}\}f(t)\|_{L_v^2(\dot H^1\cap\dot H^{N-1})}^2+\sup_{0\leq t\leq T_1}\|E(t)\|_{H^N}^2,
\end{align}
for any $0\leq t \leq T_1$.
\end{cor}

\subsection{Proof of Theorem \ref{Th1} }
To establish the uniform {\it a priori estimates} of $f(t,x,v)$, we first define the temporal energy functional $\mathcal{E}^H(t)$ and its corresponding dissipation rate $\mathcal{D}^H(t)$ as follows:
\begin{align}\label{G3.50}
\mathcal{E}^H(t)\backsim&\, \|\nu\{\mathbf{I}-\mathbf{P}\}f(t) \|_{L_v^2(\dot H^1\cap\dot H^{N-1})}^2 +\|\{\mathbf{I}-\mathbf{P}\}f(t)\|_{L_{x,v}^2}^2+\|f(t)\|_{L_{v}^2(\dot H^1\cap \dot H^N)}^2\nonumber\\
&+ \sum_{\substack{ 1\leq |\beta|\leq N \\|\alpha|+|\beta| \leq N}}\|\partial^\alpha_{x}\partial^\beta_{v}\{\mathbf{I}-\mathbf{P}\}f(t) \|_{L_{x,v}^2 }^2,\\\label{G3.51}
\mathcal{D}^H(t)=&\,\|\nu^{\frac{3}{2}}\{\mathbf{I}-\mathbf{P}\}f(t)\|_{L_v^2(\dot H^1\cap\dot H^{N-1})}^2 +\sum_{
\ell\leq N} \|\nabla^\ell\{\mathbf{I}-\mathbf{P}\}f(t)\|_{\nu}^2\nonumber\\
&+\sum_{\substack{ 1\leq |\beta|\leq N \\|\alpha|+|\beta| \leq N}} \| \partial^{\alpha}_{x}\partial^\beta_{v}\{\mathbf{I}-\mathbf{P}\}f(t)\|_{\nu}^2,
\end{align}
where the explicit expression of $\mathcal{E}^H(t)$ can be determined from the later proof. 

We shall combine the macroscopic estimates with the microscopic estimates. By adding \eqref{G3.27} and $\tau_2\times$ \eqref{G3.49} together, we obtain
\begin{align}\label{G3.52}
 \frac{{\rm d}}{{\rm d}t} \mathcal{E}^H(t)
 + \lambda_{14} \Big(\mathcal{D}^H(t)+\sum_{2\leq\ell\leq N} \|\nabla^\ell\mathbf{P}f\|_{L_{x,v}^2}^2-\tau_2 \|f\|_{L_v^2(\dot H^{\frac{3}{4}}\cap\dot H^2)}^2\Big)     
 \lesssim  \sup_{0\leq t\leq T_1}\|E\|_{H^N}^2.
\end{align}
for some constant $\lambda_{14}>0$. Here, $0<\tau_2\ll 1$ is a sufficiently small constant.
Adding the term $\lambda_{14}\|f_{L}(t)\|_{L_v^2(\dot B_{2,\infty}^{\frac{1}{2}})}^2$
to both sides of \eqref{G3.52} yields
\begin{align} \label{G3.53}
 \frac{{\rm d}}{{\rm d}t} \mathcal{E}^H(t)
 + \lambda_{15} \Big(\mathcal{D}^H(t)+\|f(t)\|_{L_v^2(\dot B_{2,\infty}^\frac{1}{2} \cap\dot H^N)}^2\Big)     
 \lesssim  \sup_{0\leq t\leq T_1}\|E\|_{H^N}^2+\|f_{L}(t)\|_{L_v^2(\dot B^\frac{1}{2}_{2,\infty})}^2,
\end{align}
for some constant $\lambda_{15}>0$.
By the definitions of $\mathcal{E}^H(t)$ in \eqref{G3.50} and $\mathcal{D}^H(t)$ in \eqref{G3.51}, we have
\begin{align*}
\mathcal{E}^H(t)\lesssim \mathcal{D}^H(t)+\|f(t)\|_{L_v^2(\dot B^\frac{1}{2}_{2,\infty}\cap\dot H^N)}^2,  
\end{align*}
which together with \eqref{G3.53} yields
\begin{align}\label{G3.54}
  \frac{{\rm d}}{{\rm d}t} \mathcal{E}^H(t)
 + \lambda_{16}  \mathcal{E}^H(t)     
 \lesssim  \sup_{0\leq t\leq T_1}\|E\|_{H^N}^2+\|f_{L}\|_{L_v^2(\dot B^\frac{1}{2}_{2,\infty})}^2,   
\end{align}
for some constant $\lambda_{16}>0$.
Multiplying  \eqref{G3.54} by $e^{\lambda_{16}t}$, we further obtain
\begin{align*}
\mathcal{E}^H(t)\lesssim&\, e^{-\lambda_{16}t}\mathcal{E}^H(0)+\sup_{0\leq t\leq T_1}\bigg\{ \|E(t)\|_{H^N}^2+\|f_{L}(t)\|_{L_v^2(\dot B^\frac{1}{2}_{2,\infty})}^2 \bigg\}\int _0^{\infty} e^{-\lambda_{16}\tau} {\rm d}\tau \nonumber\\
\lesssim&\,  \mathcal{E}^H(0)+\sup_{0\leq t\leq T_1}\bigg\{ \|E(t)\|_{H^N}^2+\|f_{L}(t)\|_{L_v^2(\dot B^\frac{1}{2}_{2,\infty})}^2 \bigg\},
\end{align*}
which combined with \eqref{G3.3} gives rise to
\begin{align}\label{G3.56}
 &\|f(t)\|_{L_v^2(\dot B^{\frac{1}{2}}_{2,\infty}\cap \dot H^N)}+ \|\langle v\rangle f(t)\|_{L_v^2(\dot H^1\cap \dot H^{N-1})} +\|\{\mathbf{I}-\mathbf{P}\}f(t)\|_{L_{v}^2(L^2)}
\nonumber\\
&+ \sum_{\substack{ 1\leq |\beta|\leq N \\|\alpha|+|\beta| \leq N}}\|\partial^\alpha_{x}\partial^\beta_{v}\{\mathbf{I}-\mathbf{P}\}f(t) \|_{L_{v}^2(L^2) }\nonumber\\
 \lesssim&\,  \sup_{0\leq t\leq T_1}\|E(t)\|_{\dot B_{2,\infty}^{-\frac{3}{2}}\cap \dot H^{N}}+\|f_0\|_{L_v^2(\dot B^{\frac{1}{2}}_{2,\infty}\cap \dot H^N)}+ \|\langle v\rangle f_0\|_{L_v^2(\dot H^1\cap \dot H^{N-1})} \nonumber\\
& +\|\{\mathbf{I}-\mathbf{P}\}f_0\|_{L_{v}^2(L^2)}+\sum_{\substack{ 1\leq |\beta|\leq N \\|\alpha|+|\beta| \leq N}}\|\partial^\alpha_{x}\partial^\beta_{v}\{\mathbf{I}-\mathbf{P}\}f_0 \|_{L_{v}^2(L^2) },
\end{align}
for all $0\leq t\leq T_1$.

Similar to \cite[Section 6]{Gy-IUMJ-2004}, by a continuity argument, \eqref{G3.1} can be ensured if we choose $\delta_0$ in \eqref{G1.6} sufficiently small. By applying the Banach contraction mapping principle \cite{Gy-IM-2003}, we can establish the local existence and uniqueness of the Cauchy problem \eqref{G1.5}; the detailed proof is omitted here. Thanks to the standard continuity argument and {\it a priori} estimates \eqref{G3.56}, we obtain the global well-posedness of the strong solution $f(t,x,v)$. The estimate \eqref{G1.7} follows from \eqref{G3.56}. For the hard-sphere model,  the proof of local existence also ensures the non-negativity of solutions as 
\begin{align*}
F(t,x,v)=M+\sqrt{M}f(t,x,v)\geq 0.  
\end{align*}
Thus, the proof of Theorem \ref{Th1} is completed. \qed

\section{Asymptotic  stability of Boltzmann equation }
This section is dedicated to proving the asymptotic stability of solutions to the Cauchy problem \eqref{G1.5}. First, we present the error equation between $f^{(1)}(t,x,v)$ and $f^{(2)}(t,x,v)$ as follows:
\begin{align}\label{G4.1}
\partial_{t} \widetilde{f}+v\cdot\nabla_{x}\widetilde{f}+\mathcal{L}\widetilde{f}=\Gamma(f^{(1)}+f^{(2)},\widetilde{f})-E\cdot\nabla_{v}\widetilde{f}+\frac{1}{2} E\cdot v\widetilde{f},
\end{align}
where $\widetilde f(t,x,v)=f^{(1)}(t,x,v)-f^{(2)}(t,x,v)$ with the initial data
\begin{align*}
\widetilde{f}(t,x,v)|_{t=0}={\widetilde{f}_0(x,v)=f_0^{(1)}(x,v)-f_0^{(2)}(x,v)},
\end{align*}
satisfying \eqref{G1.8}. Here we utilized the bilinearity and symmetry of $\Gamma$: \begin{align*}
\Gamma(f^{(1)},f^{(1)}) - \Gamma(f^{(2)},f^{(2)}) = \Gamma(f^{(1)} + f^{(2)},\widetilde{f}).
\end{align*}

Next, we establish the time decay estimates of $\widetilde{f}(t,x,v)$. Unlike the approaches in \cite{DLN-2026,Deguchi-2025}, the semi-group method loses its efficacy here because the nonlinear terms $E\cdot\nabla_v \widetilde f$ and $\frac{1}{2}E\cdot v\widetilde f$ in \eqref{G4.1} contain $\nabla_v$ and weighted $v$, which cannot provide any decay rates in the low-frequency range. To overcome this difficulty, we first need to conduct the energy method at high frequencies to obtain the estimate of $\nabla_v\widetilde f$ and $\langle v\rangle \widetilde f$.

\subsection{Refined energy method at high frequencies}
 
We initially provide the estimate of $\widetilde f(t,x,v)$ in the $L_v^2(\dot H^1\cap \dot H^{N-1})$-norm with $N\geq 4$.
\begin{lem}\label{L4.1}
For strong solutions of equation \eqref{G4.1}, there exists a positive constant  $\widetilde \lambda_{1} >0$ such that 
\begin{align} \label{G4.2}
&\frac{{\rm d}}{{\rm d}t}\| \widetilde  f(t)\|_{L_v^2(\dot H^1\cap \dot H^{N-1})}^2+\widetilde \lambda_1\sum_{1\leq k\leq N-1}\|\nabla^k\{\mathbf{I}-\mathbf{P}\}\widetilde f(t)\|_{\nu}^2  \nonumber\\
\leq&\,  (\kappa+C\delta_0)  \| \mathbf{P}\widetilde f(t)\|_{L_{v}^2(\dot H^1\cap\dot H^{N-1})}^2 +C\delta_0\sum_{\substack{ 1\leq |\beta|\leq N -1\\|\alpha|+|\beta| \leq N-1}} \| \partial^{\alpha}_{x}\partial^\beta_{v}\{\mathbf{I}-\mathbf{P}\}\widetilde f(t)\|_{\nu}^2 ,
\end{align}
for any $t\geq 0$.   Here, $\delta_0$ is defined in \eqref{G1.6}, and $0 < \kappa < 1$ is a sufficiently small constant defined in \eqref{G3.10}. 
\end{lem}

\begin{proof}
 Applying the derivative $\partial^\alpha$ with $|\alpha| = k$, where $1\leq k\leq N - 1$, to \eqref{G4.1}, then multiplying the resulting identity by $\partial^\alpha{\widetilde f}$ and integrating over $\mathbb R^3_{x}\times \mathbb R^3_{v}$ yields
\begin{align}\label{G4.3}
&\frac{{\rm d}}{{\rm d}t}\|\partial^\alpha \widetilde f\|_{L_{x,v}^2}^2+\widetilde\lambda_2 \|\partial^\alpha \{\mathbf{I}-\mathbf{P}\}\widetilde f\|_{\nu}^2\nonumber\\
\leq&\, |\langle\partial^\alpha\Gamma(f^{(1)}+f^{(2)},\widetilde f) ,\partial^\alpha \widetilde f\rangle_{{x,v}} | +|\langle\partial^ \alpha(E\cdot v\widetilde f) ,\partial^\alpha \widetilde f\rangle_{{x,v}} |+|\langle\partial^ \alpha(E\cdot \nabla_v\widetilde f) ,\partial^\alpha \widetilde f\rangle_{{x,v}} |\nonumber\\
 \equiv:&\,\sum_{j=1}^3\widetilde I_j,
\end{align}
for some constant $\widetilde\lambda_2>0$.
Taking $\eta=\frac{1}{2}$ in \eqref{G2.5}, and by using \eqref{newinterpolation}, Young's inequality and the decomposition \eqref{G1.4},    we have
\begin{align} \label{G4.4}
\widetilde I_1 
\leq&\, \sum_{|\beta|\leq |\alpha|}\int_{\mathbb R^3_x
}\big|( \nu^{-\frac{1}{2}}  \Gamma(\partial^{\alpha-\beta }f^{(1)}+\partial^{\alpha-\beta }f^{(2)},\partial^\beta \widetilde f)\big|_{2}   \big|\nu^{\frac{1}{2}} \partial^\alpha \{\mathbf{I}-\mathbf{P}\}\widetilde f   \big|_2 {\rm d}x \nonumber\\
\leq&\, \sum_{|\beta|\leq |\alpha|} \int_{\mathbb R^3_x}  |(\partial^{\alpha-\beta} f^{(1)},\partial^{\alpha-\beta} f^{(2)})|_{\nu}|\partial^\beta \widetilde f|_{2} |\partial^\alpha \{\mathbf{I}-\mathbf{P}\}\widetilde f    |_{\nu} {\rm d}x \nonumber\\
&+ \sum_{|\beta|\leq |\alpha|} \int_{\mathbb R^3_x}  |(\partial^{\alpha-\beta} f^{(1)},\partial^{\alpha-\beta} f^{(2)})|_{2}|\partial^\beta \widetilde f|_{\nu} |\partial^\alpha \{\mathbf{I}-\mathbf{P}\}\widetilde f    |_{\nu} {\rm d}x \nonumber\\
\leq&\,  \kappa  \|\partial^\alpha \{\mathbf{I}-\mathbf{P}\}\widetilde f\|_{\nu}^2+C\sum_{|\beta|\leq k}\big\|  |\nabla^{|\beta|} (f^{(1)},f^{(2)})|_{2} |\nabla^{k-|\beta|}\widetilde f|_{\nu} \big \|_{ L^2}^2\nonumber\\
&+C\sum_{|\beta|\leq k}\big\|  |\nabla^{|\beta|} (f^{(1)},f^{(2)})|_{\nu} |\nabla^{k-|\beta|}\widetilde f|_{2} \big \|_{L^2}^2\nonumber\\
\leq&\, \kappa  \|\partial^\alpha \{\mathbf{I}-\mathbf{P}\}\widetilde f\|_{\nu}^2+C\|(\nu f^{(1)},\nu f^{(2)})\|_{L_v^2(\dot H^1\cap \dot H^{N-1})}^2 \|\widetilde f\|_{L_v^2(\dot H^1\cap \dot H^{N-1})}^2\nonumber\\
&+ C\|(  f^{(1)},  f^{(2)})\|_{L_v^2(\dot B^{\frac{1}{2}}_{2,\infty}\cap \dot H^{{N}})}^2 \bigg(\sum_{1\leq\ell\leq {N-1}}\|\nabla^\ell\{\mathbf{I}-\mathbf{P}\}\widetilde f\|_{\nu}^2+\|\widetilde f\|_{L_v^2(\dot H^1\cap\dot H^{N-1})}^2\bigg)\nonumber\\
\leq &\, (\kappa+C\delta_0)\|\partial^\alpha \{\mathbf{I}-\mathbf{P}\}\widetilde f\|_{\nu}^2+C\delta_0 \bigg(\sum_{1\leq\ell\leq N-1}\|\nabla^\ell\{\mathbf{I}-\mathbf{P}\}\widetilde f\|_{\nu}^2+\|\mathbf{P}\widetilde f\|_{L_v^2(\dot H^1\cap \dot H^{N-1})}^2\bigg).
\end{align}
Similar to the inequalities \eqref{G3.15}--\eqref{G3.16},   we have
\begin{align}\label{G4.5}
\widetilde I_2+\widetilde I_3\leq &\, \kappa \sum_{2\leq\ell\leq N-1}\|{\nabla^\ell}\mathbf{P}\widetilde f\|_{L_{x,v}^2}^2+ C\sup_{t\geq 0}\|(E\cdot v\widetilde f,E\cdot\nabla_{v}\widetilde f)\|_{L_v^2(H^{N-3})} ^2\nonumber\\
&+C\sup_{t \geq 0}\|E\widetilde f\|_{L_v^2(\dot H^{N-1})} \Big( \|\nabla
^{N-1}\{\mathbf{I}-\mathbf{P}\}\widetilde f\|_{\nu}+\|\mathbf{P}\widetilde f\|_{L_v^2(\dot H^{N-1})}  \Big)\nonumber\\
&+C\|\nabla^{N-1}(E\cdot\nabla_{v}\widetilde f)-E\nabla^{N-1}\nabla_v\widetilde f\|_{L_{x,v}^2} \|\nabla^{N-1}\widetilde f\|_{L_{x,v}^2} \nonumber\\
\leq&\, \kappa\|\mathbf{P}\widetilde f\|_{L_v^2(\dot H^1\cap \dot H^{N-1})}^2+(C\delta_0+\kappa)\bigg(\sum_{1\leq\ell\leq {N-1}}\|\nabla^\ell\{\mathbf{I}-\mathbf{P}\}\widetilde f\|_{\nu}^2+\|\mathbf{P}\widetilde f\|_{L_v^2(\dot H^1\cap \dot H^{N-1})}^2\bigg)\nonumber\\
&+C\delta_0\sum_{\substack{ 1\leq |\beta|\leq {N-1} \\|\alpha|+|\beta| \leq {N-1}}} \| \partial^{\alpha}_{x}\partial^\beta_{v}\{\mathbf{I}-\mathbf{P}\}\widetilde f\|_{\nu}^2. 
\end{align}
Substituting the  estimates \eqref{G4.4}--\eqref{G4.5} into \eqref{G4.3} and then taking the summation over the range $1\leq  k\leq {N-1}$ results in \eqref{G4.2}.
\end{proof}

Similar to \eqref{G3.20}, we define the following temporal energy functional $\widetilde {\mathfrak{E}}_{k}(t)$ related to $\widetilde f(t,x,v)$:
 \begin{align} \label{G4.6}
 \widetilde{\mathfrak{E}}_{k}(t):=&\,\sum_{|\alpha|=k}\int_{\mathbb R^3_x} \big(  \langle\{\mathbf{I}-\mathbf{P}\}\partial^\alpha \widetilde f,\zeta_{a}\rangle\cdot\nabla \partial^\alpha\widetilde a +\langle\{\mathbf{I}-\mathbf{P}\}\partial^\alpha\widetilde f, \zeta_{ij}\rangle\cdot\partial_j\partial^\alpha\widetilde b_{i}  \big) {\rm d}x\nonumber\\
 &+\sum_{|\alpha|=k}\int_{\mathbb R^3_x} \big(\langle\{\mathbf{I}-\mathbf{P}\}\partial^\alpha \widetilde f,\zeta_{c}\rangle\cdot\nabla \partial^\alpha \widetilde c+\partial^\alpha \widetilde b\cdot\nabla \partial^\alpha \widetilde a\big) {\rm d}x,
 \end{align}
for any $k=1,\dots,N-2$.
It is direct to find that
\begin{align*}
\sum_{1\leq k\leq N-2} | \widetilde{\mathfrak{E}}_{k}(t)|\lesssim \sum_{1\leq k\leq N-2}\|\widetilde f\|_{L_v^2(\dot H^k)} \|\widetilde f\|_{L_v^2(\dot H^{k+1})} \lesssim \|\widetilde f\|_{L_v^2(\dot H^1\cap \dot H^{N-1})}^2.  
\end{align*}
Then, we provide the estimate of $\|\mathbf{P}\widetilde f\|_{L_v^2(\dot H^2\cap \dot H^{N - 1})}$.

\begin{lem}\label{L4.2}
For strong solutions of equation \eqref{G4.1}, there exists a positive constant $\widetilde \lambda_3>0$    such that 
\begin{align}\label{G4.7} 
&\frac{{\rm d}}{{\rm d}t}\sum_{1\leq k\leq N-2}\widetilde{\mathfrak{E}}_{ k}(t)+\widetilde\lambda_3 \|\mathbf{P}\widetilde f(t)\|_{L_v^2(\dot H^2\cap\dot H^{N-1})}^2 \nonumber\\
&\quad\lesssim \|  \{\mathbf{I}-\mathbf{P}\}\widetilde f(t)\|_{L_{v}^2(\dot H^1\cap \dot H^{N-1})}^2+\delta_0\|\mathbf{P}\widetilde
 f(t)\|_{L_v^2(\dot H^1)}^2,
\end{align}
for    any $t\geq 0$,  where $\widetilde {\mathfrak{E}}_{k}(t)$ is defined by \eqref{G4.6}.  Here, $\delta_0$ is defined in \eqref{G1.6}.
\end{lem}
\begin{proof}
Similar to \eqref{G3.23}, we can derive that
\begin{align}\label{G4.8}
&\frac{{\rm d}}{{\rm d}t} \sum_{1\leq k\leq N-2}\widetilde{\mathfrak{E}}_{k}(t)+\widetilde\lambda_{4} \|  \mathbf{P}\widetilde f\|_{L_{v}^2(\dot H^2\cap\dot H^{N-1})}^2 \nonumber\\
&\quad\lesssim   \|  \{\mathbf{I}-\mathbf{P}\}\widetilde f \|_{L_{v}^2(\dot H^1\cap \dot H^{N-1})}^2 +\sum_{1\leq k\leq N-2}\|\langle\nabla^k\Gamma(f^{(1)}+f^{(2)},\widetilde f),\zeta\rangle\|_{L^2}^2 \nonumber\\
&\quad\quad +\sum_{1\leq k\leq N-2}\|\langle\nabla^k (E\cdot v\widetilde f),\zeta\rangle\|_{L^2}^2+\sum_{1\leq k\leq N-2}\|\langle\nabla^k (E\cdot \nabla_v\widetilde f),\zeta\rangle\|_{L^2}^2\nonumber\\
&\quad\equiv:\|  \{\mathbf{I}-\mathbf{P}\}\widetilde f \|_{L_{v}^2(\dot H^1\cap \dot H^{N-1})}^2 +\sum_{j=4}^{6}\widetilde I_{j},
\end{align}
for some constant $\widetilde\lambda_{4}>0$.
By direct calculation, 
we arrive at
\begin{align}\label{G4.9}
 \widetilde I_{4}\lesssim&\,   \sum_{1\leq k\leq N-2} \sum_{|\beta|\leq k} \big\|  |\nabla^{|\beta|}(f^{(1)},f^{(2)})|_{2}  |\nabla^{k-|\beta|}\widetilde f|_{2}  \big\|_{L^2}^2\nonumber\\
 \lesssim&\, \|(f^{(1)},f^{(2)})\|_{L_v^2(\dot B^{\frac{1}{2}}_{2,\infty}\cap\dot H^{N} )}^2 \|\widetilde f\|_{L_v^2(\dot H^1\cap\dot H^{N-1})}^2\lesssim \delta_0^2 \|\widetilde f\|_{L_v^2(\dot H^1\cap\dot H^{N-1})}^2,
\end{align}
and
\begin{align}\label{G4.10}
\widetilde I_5+\widetilde I_{6}\lesssim \sup_{t\geq 0} \|E\widetilde f\|_{L_v^2(\dot H^1\cap \dot H^{N-2})} ^2\lesssim \sup_{t\geq 0}\|E\|_{H^N}^2\|\widetilde f\|_{L_v^2(\dot H^1\cap \dot H^{N-1})}^2\lesssim \delta_0 \|\widetilde f\|_{L_v^2(\dot H^1\cap \dot H^{N-1})}^2.
\end{align}
Inserting the estimates \eqref{G4.9} and \eqref{G4.10} into \eqref{G4.8} yields the desired estimate \eqref{G4.7}.
\end{proof}

To estimate the microscopic part of the difference,
we use the identity
\begin{align*}
 \Gamma(f^{(1)},f^{(1)})
-\Gamma(f^{(2)},f^{(2)}) =
\frac12\sum_{i=1}^2
\big[
\Gamma(f^{(i)},\widetilde f)
+\Gamma(\widetilde f,f^{(i)})
\big],
\end{align*}
together with the fact that
\begin{align*}
\mathbf P\big[
\Gamma(f^{(1)},f^{(1)})
-\Gamma(f^{(2)},f^{(2)})
\big]=0.    
\end{align*}
For $i=1,2$, the notation
$\mathcal D^H[f^{(i)}](t)$ denotes the dissipation
in \eqref{G3.51} evaluated at $f^{(i)}$.
The functional $\widetilde{\mathcal E}^H(t)$ appearing
in the following Lemmas \ref{L4.3}--\ref{L4.4} is the positive
energy to be defined in \eqref{G4.17} in the proof of Theorem \ref{T4.6} later.

\begin{rem}
The most important point in the proof of Lemmas \ref{L4.3}--\ref{L4.5} below is that we need to notice the more precise zeroth-order estimate in \eqref{newinterpolation} in Proposition \ref{prop2.6}:
\begin{align*}
\big\|f^{(1)} |\widetilde f|_{2}\big\|_{L^2}^2+\big\|f^{(2)} |\widetilde f|_{2}\big\|_{L^2}^2\lesssim \bigg(\|f^{(1)}\|_{\dot B^{\frac{1}{2}}_{2,\infty}}^2+\|f^{(2)}\|_{\dot B^{\frac{1}{2}}_{2,\infty}}^2 \bigg)   \|\widetilde f\|_{L_{v}^2(\dot H^1)}^2\lesssim \delta_0 \|\widetilde f\|_{L_{v}^2(\dot H^1)}^2.  
\end{align*}
\end{rem}

\begin{lem}\label{L4.3}
For strong solutions of equation   \eqref{G4.1}, there exists a positive constant $\widetilde \lambda_{5}>0$ such that
\begin{align}\label{G4.11}
&\frac{{\rm d}}{{\rm d}t}
\|\nu\{\mathbf I-\mathbf P\}\widetilde f(t)\|_
{L_v^2(\dot H^1\cap\dot H^{N-2})}^2
 +
\widetilde\lambda_5
\|\nu^{\frac32}\{\mathbf I-\mathbf P\}\widetilde f(t)\|_
{L_v^2(\dot H^1\cap\dot H^{N-2})}^2
\nonumber\\
\lesssim&\,
\|\nu^{\frac12}\{\mathbf I-\mathbf P\}\widetilde f(t)\|_
{L_v^2(\dot H^1\cap\dot H^{N-2})}^2
+\|\widetilde f(t)\|_
{L_v^2(\dot H^1\cap\dot H^{N-1})}^2
\nonumber\\
& +
\delta_0
\|\nu^{\frac12}\nabla_v\{\mathbf I-\mathbf P\}
\widetilde f(t)\|_
{L_v^2(\dot H^1\cap\dot H^{N-2})}^2
 +\bigg(
\sum_{i=1}^2\mathcal D^H[f^{(i)}](t)
\bigg)\widetilde{\mathcal E}^H(t),
\end{align}
for any $t\geq 0$.
\end{lem}

\begin{proof}
Apply $\{\mathbf I-\mathbf P\}\partial_x^\alpha$
to the difference of the two equations, where
$1\leq|\alpha|\leq N-2$, and take the inner product
with $\nu^2\partial_x^\alpha
\{\mathbf I-\mathbf P\}\widetilde f$.
The linear and external-force terms are treated
as in Lemma \ref{L3.6}.
By using \eqref{NJKG2.6}, the macro-micro decomposition,
Theorem \ref{Th1} and Young's inequality,
for any $\eta>0$,
\begin{align*}
&\Bigl|\Bigl\langle
\nu^2\partial_x^\alpha
\bigl[\Gamma(f^{(1)},f^{(1)})
-\Gamma(f^{(2)},f^{(2)})\bigr],
\partial_x^\alpha\{\mathbf I-\mathbf P\}\widetilde f
\Bigr\rangle_{x,v}\Bigr|\\
&\quad\leq
\eta\|\nu^{\frac32}\partial_x^\alpha
\{\mathbf I-\mathbf P\}\widetilde f\|_{L_{x,v}^2}^2 +
C\delta_0
\|\nu^{\frac32}\{\mathbf I-\mathbf P\}\widetilde f\|_
{L_v^2(\dot H^1\cap\dot H^{N-2})}^2
+C\delta_0
\|\widetilde f\|_
{L_v^2(\dot H^1\cap\dot H^{N-1})}^2\\
&\quad\quad+
C_\eta\sum_{i=1}^2
\sum_{\alpha_1+\alpha_2=\alpha}
\Bigl\|
|\nu\partial_x^{\alpha_1}\widetilde f|_2
|\nu^{\frac32}\partial_x^{\alpha_2}
\{\mathbf I-\mathbf P\}f^{(i)}|_2
\Bigr\|_{L^2_x}^2.
\end{align*}
Here the terms controlled by the background energy
are estimated using
$\dot H^1\cap\dot H^2\hookrightarrow L^\infty$;
only the displayed microscopic product requires
the background dissipation.

For this product, use $L^2_x\times L^\infty_x$
when $\alpha_2=0$ or $|\alpha_1|\geq N-3$,
and $L^\infty_x\times L^2_x$ otherwise.
In the first case, $|\alpha_1|\geq1$ and
$|\alpha_2|\leq1$; in the second,
$|\alpha_1|\leq N-4$ and $|\alpha_2|\geq1$.
Since $N\geq4$, Sobolev's inequality yields
\begin{align*}
 \sum_{\alpha_1+\alpha_2=\alpha}
\Bigl\|
|\nu\partial_x^{\alpha_1}\widetilde f|_2
|\nu^{\frac32}\partial_x^{\alpha_2}
\{\mathbf I-\mathbf P\}f^{(i)}|_2
\Bigr\|_{L^2_x}^2 \lesssim&\,
\|\nu\widetilde f\|_
{L_v^2(\dot H^1\cap\dot H^{N-2})}^2
\|\nu^{\frac32}\{\mathbf I-\mathbf P\}f^{(i)}\|_
{L_v^2(\dot H^1\cap\dot H^{N-1})}^2\\
 \lesssim&\,
\mathcal D^H[f^{(i)}](t)
\widetilde{\mathcal E}^H(t).
\end{align*}
Substituting this bound into the collision estimate,
combining with the linear and external-force estimates,
and summing over $|\alpha|=1$ and $|\alpha|=N-2$,
we obtain \eqref{G4.11} by choosing
$\eta$ and then $\delta_0$ sufficiently small.
\end{proof}


\begin{lem}\label{L4.4}
For strong solutions of equation   \eqref{G4.1}, there exists a positive constant $\widetilde \lambda_{6}>0$ such that
\begin{align}\label{G4.12}
&\frac{{\rm d}}{{\rm d}t}
\sum_{\substack{
1\leq|\beta|\leq N-1\\
|\alpha|+|\beta|\leq N-1}}
c_{|\beta|}
\|\partial_x^\alpha\partial_v^\beta
\{\mathbf I-\mathbf P\}\widetilde f(t)\|_{L_{x,v}^2}^2
 +
\widetilde\lambda_6
\sum_{\substack{
1\leq|\beta|\leq N-1\\
|\alpha|+|\beta|\leq N-1}}
c_{|\beta|}
\|\partial_x^\alpha\partial_v^\beta
\{\mathbf I-\mathbf P\}\widetilde f(t)\|_\nu^2
\nonumber\\
&\quad\lesssim
\|\nu^{\frac12}\{\mathbf I-\mathbf P\}\widetilde f(t)\|_
{L_v^2(H^{N-2})}^2
+\|\widetilde f(t)\|_
{L_v^2(\dot H^1\cap\dot H^{N-1})}^2+
\bigg(
\sum_{i=1}^2\mathcal D^H[f^{(i)}](t)
\bigg)\widetilde{\mathcal E}^H(t),
\end{align}
for any $t\geq 0$.
The constants $c_k$ are independent of the two solutions.
\end{lem}

\begin{proof}
Apply $\partial_x^\alpha\partial_v^\beta$
to the microscopic difference equation, where
$|\beta|\geq1$ and $|\alpha|+|\beta|\leq N-1$,
and take the inner product with
$\partial_x^\alpha\partial_v^\beta
\{\mathbf I-\mathbf P\}\widetilde f$.
By employing \eqref{NJKGnew}, the macro-micro decomposition
and Young's inequality, for any $\eta>0$,
\begin{align*}
&\Bigl|\Bigl\langle
\partial_x^\alpha\partial_v^\beta
\bigl[\Gamma(f^{(1)},f^{(1)})
-\Gamma(f^{(2)},f^{(2)})\bigr], 
\partial_x^\alpha\partial_v^\beta
\{\mathbf I-\mathbf P\}\widetilde f
\Bigr\rangle_{x,v}\Bigr|\\
\leq&\,
\eta\|\partial_x^\alpha\partial_v^\beta
\{\mathbf I-\mathbf P\}\widetilde f\|_\nu^2 +
C_\eta\delta_0^2
\Big(
\|\nu^{\frac12}\{\mathbf I-\mathbf P\}\widetilde f\|_
{L_v^2(H^{N-2})}^2
+\|\widetilde f\|_
{L_v^2(\dot H^1\cap\dot H^{N-1})}^2
\Big)\\
& +
C_\eta\delta_0^2
\sum_{\substack{
1\leq|\beta'|\leq N-1\\
|\alpha'|+|\beta'|\leq N-1}}
\|\partial_x^{\alpha'}\partial_v^{\beta'}
\{\mathbf I-\mathbf P\}\widetilde f\|_\nu^2\\
&+
C_\eta\sum_{i=1}^2
\sum_{\substack{
\alpha_1+\alpha_2=\alpha\\
\beta_1+\beta_2\leq\beta,\ |\beta_2|\geq1}}
\Bigl\|
|\partial_x^{\alpha_1}\partial_v^{\beta_1}
\widetilde f|_2
|\partial_x^{\alpha_2}\partial_v^{\beta_2}
\{\mathbf I-\mathbf P\}f^{(i)}|_\nu
\Bigr\|_{L^2_x}^2.
\end{align*}
Here the energy-controlled products follow from
Theorem \ref{Th1}, \eqref{newinterpolation}
and Sobolev's inequality.
For the weighted pure spatial derivatives, we can use
$$
\Bigl\|
|\nu^{\frac12}\partial_x^{\alpha_2}
\{\mathbf I-\mathbf P\}f^{(i)}|_2
\Bigr\|_{L^3_x}
\lesssim
\|\partial_x^{\alpha_2}
\{\mathbf I-\mathbf P\}f^{(i)}\|_{L_{x,v}^2}^{\frac12}
\|\nu^{\frac12}\nabla_x\partial_x^{\alpha_2}
\{\mathbf I-\mathbf P\}f^{(i)}\|_{L_{x,v}^2}^{\frac12}
\lesssim
\delta_0^{\frac12}\,
\mathcal D^H[f^{(i)}](t)^{\frac12},
$$
for $|\alpha_2|\leq N-2$.
In particular, at $|\alpha_2|=N-2$ the other factor
has total order at most one and is placed in $L^6_x$.
Products of two macroscopic factors without
$x$-derivatives are estimated by \eqref{newinterpolation}.

For the last sum, use $L^\infty_x\times L^2_x$
when $|\alpha_1|+|\beta_1|\leq N-3$.
Otherwise,
$|\alpha_1|+|\beta_1|=N-2$, $\alpha_2=0$
and $|\beta_2|=1$; use $L^2_x\times L^\infty_x$,
or $L^6_x\times L^3_x$ for the macroscopic
difference factor with $\alpha_1=0$.
Since $N\geq4$, the available derivative orders yield
\begin{align*}
 \Bigl\|
|\partial_x^{\alpha_1}\partial_v^{\beta_1}
\widetilde f|_2
|\partial_x^{\alpha_2}\partial_v^{\beta_2}
\{\mathbf I-\mathbf P\}f^{(i)}|_\nu
\Bigr\|_{L^2_x}^2 \lesssim
\mathcal D^H[f^{(i)}](t)
\widetilde{\mathcal E}^H(t).
\end{align*}

The transport commutator satisfies
\begin{align*}
 [\partial_x^\alpha\partial_v^\beta,
v\cdot\nabla_x]\{\mathbf I-\mathbf P\}\widetilde f 
 =
\sum_{\substack{1\leq j\leq3\\\beta_j>0}}
\beta_j\partial_x^{\alpha+e_j}\partial_v^{\beta-e_j}
\{\mathbf I-\mathbf P\}\widetilde f,
\end{align*}
where $e_j$ is the $j$-th coordinate multi-index.
For $|\beta|=1$, its pairing is bounded by
\[
\eta\|\partial_x^\alpha\partial_v^\beta
\{\mathbf I-\mathbf P\}\widetilde f\|_\nu^2
+C_\eta
\|\widetilde f\|_{L_v^2(\dot H^1\cap\dot H^{N-1})}^2.
\]
For $|\beta|\geq2$, it produces lower-order
velocity dissipation, as does the commutator
with $\mathcal L=\nu-\mathcal K$.
The external-force and remaining macroscopic terms
are treated as in Lemma \ref{L3.7}.

Starting with $c_{N-1}=1$, choose
$c_{N-2},\ldots,c_1$ successively sufficiently large
to absorb the lower-order velocity terms.
Multiplying the estimate at velocity order $k$
by $c_k$, summing over $1\leq k\leq N-1$,
and then taking $\delta_0$ sufficiently small,
we obtain \eqref{G4.12}.
\end{proof}


  \begin{lem} \label{L4.5}
For strong solutions of equation   \eqref{G4.1}, there exists a positive constant $\widetilde \lambda_{7}>0$ such that
\begin{align} \label{G4.13}
&\frac{{\rm d}}{{\rm d}t}\| \{\mathbf{I}-\mathbf{P}\}\widetilde f(t) \|_{L_{x,v}^2 }^2+ \widetilde\lambda_{7} \| \{\mathbf{I}-\mathbf{P}\}\widetilde f(t)\|_{\nu}^2  \lesssim \|\widetilde f\|_{L_v^2(\dot H^1)}^2,
\end{align}
for any  $t\geq 0$.
\end{lem} 

\begin{proof}
Using a similar argument as in the proof of Lemma \ref{L3.8},
together with the product estimate \eqref{newinterpolation},
we can obtain \eqref{G4.13}.
For brevity, we omit the details.
\end{proof}

Collecting all the estimates \eqref{G4.2}, \eqref{G4.7}, \eqref{G4.11}, \eqref{G4.12}, and \eqref{G4.13} in Lemmas \ref{L4.1}, \ref{L4.2}, \ref{L4.3}, \ref{L4.4} and \ref{L4.5}, we can deduce the following time-weighted estimate.

\begin{thm}\label{T4.6}
Under the assumptions of Theorem \ref{Th2}, we have
\begin{align} \label{G4.14}
&\sup_{t\geq 0} (1+t)^{\frac{1-\varepsilon-s_0}{2}}  \Big( 
 \|\widetilde f(t)\|_{L_v^2(\dot H^1\cap \dot H^{N-1})}+\|{\langle v\rangle}\widetilde f(t) \|_{L_v^2(\dot H^1\cap\dot H^{N-2})} +\|\{\mathbf{I}-\mathbf{P}\}\widetilde f(t)\|_{L_{x,v}^2}  \Big)\nonumber\\
& +\sup_{t\geq 0} (1+t)^{\frac{1-\varepsilon-s_0}{2}}   \sum_{\substack{ 1\leq |\beta|\leq N-1 \\|\alpha|+|\beta| \leq N-1}}\|\partial^\alpha_{x}\partial^\beta_{v}\{\mathbf{I}-\mathbf{P}\}\widetilde f(t) \|_{L_{x,v}^2 }    \nonumber\\
\lesssim&\,   \widetilde{\mathcal{D}}+\|\widetilde f_0\|_{L_v^2(\dot H^1\cap \dot H^{N-1})}+\|\langle v\rangle\widetilde f_0 \|_{L_v^2(\dot H^1\cap\dot H^{N-2})} +\|\{\mathbf{I}-\mathbf{P}\}\widetilde f_0\|_{L_{x,v}^2} \nonumber\\
&+\sum_{\substack{ 1\leq |\beta|\leq N-1 \\|\alpha|+|\beta| \leq N-1}}\|\partial^\alpha_{x}\partial^\beta_{v}\{\mathbf{I}-\mathbf{P}\}\widetilde f_0\|_{L_{x,v}^2 } ,
\end{align}
where  $\widetilde f_0(x,v)= f^{(1)}_0 (x,v)-f^{(2)}_0 (x,v)$
and $\widetilde{\mathcal{D}} $ is given by
\begin{align}\label{add.def.D}
\widetilde{\mathcal{D}}:=  \sup_{t\geq 0}\,(1+t)^{\frac{1}{2}-\frac{\varepsilon+s_0}{2}} \|\widetilde f (t)\|_{L_v^2(\dot B_{2,\infty}^{1-\varepsilon} )},  
\end{align}  
with $\varepsilon>0$ and $s_0\in\big(-\frac{3}{2},\frac{1}{2}\big]$ so as to ensure the choice of $s\in\big[-\frac{3}{2}+\varepsilon,1-\varepsilon\big]$ such that $s\geq s_0$ for the later use. 
\end{thm}

\begin{proof}
Collecting the estimates \eqref{G4.2}, \eqref{G4.7}
and \eqref{G4.11}--\eqref{G4.13}, we obtain
\begin{align}\label{G4.16}
 \frac{{\rm d}}{{\rm d}t}\widetilde{\mathcal E}^H(t)
+\widetilde\lambda_8\widetilde{\mathcal D}^H(t)
 \lesssim
(\delta_0+\tau_3)
\|\mathbf P\widetilde f(t)\|_{L_v^2(\dot H^1)}^2
+ \sum_{i=1}^2\mathcal D^H[f^{(i)}](t)
\widetilde{\mathcal E}^H(t),
\end{align}
for some constant $\widetilde\lambda_8>0$.
Here, $\widetilde{\mathcal E}^H(t)$ and
$\widetilde{\mathcal D}^H(t)$ are defined as
\begin{align}\label{G4.17}
\widetilde{\mathcal E}^H(t):=&\,
\|\widetilde f(t)\|_{L_v^2(\dot H^1\cap\dot H^{N-1})}^2
+\tau_4\sum_{k=1}^{N-2}\widetilde{\mathfrak E}_k(t)+\tau_3\|\nu\{\mathbf I-\mathbf P\}\widetilde f(t)\|_
{L_v^2(\dot H^1\cap\dot H^{N-2})}^2
\nonumber\\
&+\tau_3\bigg(
 c_0\|\{\mathbf I-\mathbf P\}\widetilde f(t)\|_{L_{x,v}^2}^2
  +
\sum_{\substack{
1\leq|\beta|\leq N-1\\
|\alpha|+|\beta|\leq N-1}}
c_{|\beta|}
\|\partial_x^\alpha\partial_v^\beta
\{\mathbf I-\mathbf P\}\widetilde f(t)\|_{L_{x,v}^2}^2
\bigg),
\end{align}
and
\begin{align*}
\widetilde{\mathcal D}^H(t):=&\,
\|\nu^{\frac32}\{\mathbf I-\mathbf P\}\widetilde f(t)\|_
{L_v^2(\dot H^1\cap\dot H^{N-2})}^2
+\|\{\mathbf I-\mathbf P\}\widetilde f(t)\|_\nu^2+
\|\nu^{\frac12}\{\mathbf I-\mathbf P\}\widetilde f(t)\|_
{L_v^2(\dot H^{N-1})}^2\\
&
+\|\mathbf P\widetilde f(t)\|_
{L_v^2(\dot H^2\cap\dot H^{N-1})}^2 +
\sum_{\substack{
1\leq|\beta|\leq N-1\\
|\alpha|+|\beta|\leq N-1}}
c_{|\beta|}
\|\partial_x^\alpha\partial_v^\beta
\{\mathbf I-\mathbf P\}\widetilde f(t)\|_\nu^2.
\end{align*}
The coefficients $c_k$ are chosen as in Lemma \ref{L4.4},
$c_0$ is sufficiently large, and
$\tau_4,\tau_3,\kappa,\delta_0$ are chosen successively
small, with $\tau_3\ll\tau_4$.

Owing to the estimate following \eqref{G4.6},
$\widetilde{\mathcal E}^H(t)$ is positive and equivalent
to the sum of the squared norms in \eqref{G4.17}.
Moreover,
\begin{align*}
\widetilde{\mathcal E}^H(t)
\lesssim
\widetilde{\mathcal D}^H(t)
+\|\mathbf P\widetilde f(t)\|_{L_v^2(\dot H^1)}^2.    
\end{align*} 
By interpolation and Young's inequality, for any $\eta>0$, it holds that
\begin{align*}
\|\mathbf P\widetilde f(t)\|_{L_v^2(\dot H^1)}^2
&\leq
C_\eta\|\widetilde f(t)\|_
{L_v^2(\dot B^{1-\varepsilon}_{2,\infty})}^2
+\eta\|\widetilde f(t)\|_{L_v^2(\dot H^{N-1})}^2\\
&\leq
C_\eta\|\widetilde f(t)\|_
{L_v^2(\dot B^{1-\varepsilon}_{2,\infty})}^2
+C\eta\widetilde{\mathcal D}^H(t).
\end{align*}
Combining the above estimates with \eqref{G4.16}
and choosing $\eta$ sufficiently small, we arrive at
\begin{align}\label{G4.18}
&\frac{{\rm d}}{{\rm d}t}\widetilde{\mathcal E}^H(t)
+\Big(
\widetilde\lambda_9
-C\sum_{i=1}^2\mathcal D^H[f^{(i)}](t)
\Big)\widetilde{\mathcal E}^H(t)
 \lesssim
 \|\widetilde f(t)\|_
{L_v^2(\dot B^{1-\varepsilon}_{2,\infty})}^2,
\end{align}
for some constant $\widetilde\lambda_9>0$.

On the other hand, \eqref{G3.53} and Theorem \ref{Th1}
imply that
\begin{align*}
\int_s^t\sum_{i=1}^2
\mathcal D^H[f^{(i)}](\tau) {\rm d}\tau
 \lesssim
 \sum_{i=1}^2\mathcal E^H[f^{(i)}](s)
+ \delta_0^2(1+t-s) \lesssim \delta_0^2(1+t-s),
\end{align*}
for all $0\leq s\leq t$, where
$\mathcal E^H[f^{(i)}]$ is the energy in
\eqref{G3.50} evaluated at $f^{(i)}$.
Taking $\delta_0$ sufficiently small, we further have
\begin{align*}
&-\widetilde\lambda_9(t-s)
+C\int_s^t\sum_{i=1}^2
\mathcal D^H[f^{(i)}](\tau) {\rm d}\tau \leq
C\delta_0^2-\frac{\widetilde\lambda_9}{2}(t-s).
\end{align*}
Thus, by using Gr\"onwall's inequality in \eqref{G4.18}
and the definition of $\widetilde{\mathcal D}$ in \eqref{add.def.D},
we derive
\begin{align*}
\widetilde{\mathcal E}^H(t)
\lesssim{}&
e^{-\frac{\widetilde\lambda_9}{2}t}
\widetilde{\mathcal E}^H(0)
+\int_0^t e^{-\frac{\widetilde\lambda_9}{2}(t-s)}
\|\widetilde f(s)\|_
{L_v^2(\dot B^{1-\varepsilon}_{2,\infty})}^2\,{\rm d}s\\
\lesssim{}&
e^{-\frac{\widetilde\lambda_9}{2}t}
\widetilde{\mathcal E}^H(0)
+\widetilde{\mathcal D}^2
\int_0^t e^{-\frac{\widetilde\lambda_9}{2}(t-s)}
(1+s)^{-(1-\varepsilon-s_0)} {\rm d}s\\
\lesssim{}&
(1+t)^{-(1-\varepsilon-s_0)}
\big(
\widetilde{\mathcal E}^H(0)
+\widetilde{\mathcal D}^2
\big),
\end{align*}
where we used $1-\varepsilon-s_0>0$.

Finally, by the macro-micro decomposition and
$\nu(v)\sim\langle v\rangle$, we arrive at
\begin{align*}
\|\langle v\rangle\widetilde f(t)\|_
{L_v^2(\dot H^1\cap\dot H^{N-2})}^2
\lesssim\widetilde{\mathcal E}^H(t).    
\end{align*}
Taking square roots and time suprema, we obtain
\eqref{G4.14} from \eqref{G4.17},
which completes the proof of Theorem \ref{T4.6}.
\end{proof}

\subsection{Semi-group method at low frequencies}
With the help of Theorem \ref{T4.6}, since we have obtained the decay of $\langle v\rangle \widetilde f$ and $\nabla_v \widetilde f$, we can further close our semi-group estimate.

\begin{thm}\label{T4.7}
Under the assumptions of Theorem \ref{Th2}, we have
\begin{align}\label{G4.19}
&\sup_{t\geq 0} (1+t)^{\frac{s-s_0}{2}}\|\widetilde f\|_{L_v^2(\dot B_{2,\infty}^s)} \nonumber\\
\lesssim&\, \delta_0 \big(\widetilde{\mathcal{D}}_{\varepsilon}+\widetilde{\mathcal{D}}\big)+\|\widetilde f_0\|_{L_v^2(\dot B_{2,\infty}^{s_0}\cap \dot B_{2,\infty}^{1-\varepsilon})}+\|\widetilde f_0\|_{L_v^2(\dot H^1\cap \dot H^{N-1})}+\|\langle v\rangle\widetilde f_0 \|_{L_v^2(\dot H^1\cap\dot H^{N-2})}  \nonumber\\
&+\|\{\mathbf{I}-\mathbf{P}\}\widetilde f_0\|_{L_{x,v}^2}+\sum_{\substack{ 1\leq |\beta|\leq N-1 \\|\alpha|+|\beta| \leq N-1}}\|\partial^\alpha_{x}\partial^\beta_{v}\{\mathbf{I}-\mathbf{P}\}\widetilde f_0\|_{L_{x,v}^2 } ,   
\end{align}
for any $s\in\big[-\frac{3}{2}+\varepsilon,1-\varepsilon\big]$ with $s\geq s_0$ and $s_0\in\big(-\frac{3}{2},\frac{1}{2}\big]$, 
where  $\widetilde f_0(x,v)= f_0^{(1)} (x,v)-f_0^{(2)} (x,v)$ and ${\widetilde{\mathcal{D}}_{\varepsilon}}$ is given by
\begin{align*}
\widetilde{\mathcal{D}}_{\varepsilon}:=\sup_{s_1\leq \bar s\leq 1-\varepsilon} \sup_{t\geq 0}\,(1+t)^{\frac{\bar s-s_0}{2}} \|\widetilde f (t)\|_{L_v^2(\dot B_{2,\infty}^{\bar s}\cap \dot H^{N-1})},\quad s_1=\max\{0,{s_0}\},   
\end{align*}
and $\delta_0$ is defined by \eqref{G1.6}.
\end{thm}  

\begin{proof}
By Duhamel's principle, the solution of \eqref{G4.1} can be written as
\begin{align}\label{G4.20}
\widetilde f(t)=e^{t\mathcal B} \widetilde f_0+\int_0^t e^{\tau\mathcal{B}}\Big[\Gamma(f^{(1)}+f^{(2)},\widetilde f)+\frac{1}{2}E\cdot v\widetilde{f}-E\cdot\nabla_{v}\widetilde{f}\Big](t-\tau){\rm d}\tau.
\end{align}
For any $s\in \big[-\frac{3}{2}+\varepsilon,1-\varepsilon \big]$, using the same argument as presented in \cite[Lemma 4.1]{DLN-2026}, we can conclude that
\begin{align}\label{G4.21}
&\bigg\|e^{t\mathcal{B}}\widetilde{f}_0+\int_0^t e^{\tau\mathcal{B}} \Gamma(f^{(1)}+f^{(2)},\widetilde f)(t-\tau){\rm d}\tau\bigg\|_{L_v^2(\dot B^s_{2,\infty})}\nonumber\\
&\quad \lesssim     (1+t)^{-\frac{s-s_0}{2} }\|\widetilde f_0\|_{L_v^2(\dot B_{2,\infty}^{s_0}\cap \dot B_{2,\infty}^{1-\varepsilon})}+ \delta_0 \widetilde{\mathcal{D}}_{\varepsilon}(1+t)^{-\frac{s-s_0}{2}}.
\end{align}
Therefore, we only need to estimate the remaining inhomogeneous part of \eqref{G4.20}:
\begin{align}\label{G4.22}
  \widetilde{H}(t):=\int_0^t e^{\tau\mathcal{B}}\Big[ \frac{1}{2}E\cdot v\widetilde{f}-E\cdot\nabla_{v}\widetilde{f}\Big](t-\tau){\rm d}\tau.
\end{align}
Write
\begin{align*}
\mathcal{I}_0
:=&\,
\widetilde{\mathcal{D}}
+\|\widetilde f_0\|_{L_v^2(\dot H^1\cap \dot H^{N-1})}
+\|\langle v\rangle\widetilde f_0 \|_{L_v^2(\dot H^1\cap\dot H^{N-2})}
+\|\{\mathbf{I}-\mathbf{P}\}\widetilde f_0\|_{L_{x,v}^2}\\
&+\sum_{\substack{ 1\leq |\beta|\leq N-1 \\|\alpha|+|\beta| \leq N-1}}
\|\partial^\alpha_{x}\partial^\beta_{v}\{\mathbf{I}-\mathbf{P}\}\widetilde f_0\|_{L_{x,v}^2}.
\end{align*}
By \eqref{G4.14} and the product estimate \eqref{newinterpolation},
\begin{align}\label{G4.14.force}
\|E\cdot(v\widetilde f,\nabla_v\widetilde f)(\tau)\|_{L_v^2(\dot B_{2,\infty}^{-\frac32})}
+\|E\cdot(v\widetilde f,\nabla_v\widetilde f)(\tau)\|_{L_v^2(\dot B_{2,\infty}^{s})}
\lesssim
\delta_0\,\mathcal{I}_0\,(1+\tau)^{-\frac{1-\varepsilon-s_0}{2}}.
\end{align}
Since $\varepsilon<\frac12$ and $s_0\le\frac12$,
\begin{align}\label{G4.14.majorant}
(1+\tau)^{-\frac{1-\varepsilon-s_0}{2}}
\le
(1+\tau)^{-\frac14+\frac{s_0}{2}},
\qquad
0\le \frac14-\frac{s_0}{2}<1.
\end{align}

Let's focus on the estimate of $\widetilde H(t)$ in \eqref{G4.22} for the case $s\in\big[-\frac{3}{2}+\varepsilon,\frac{1}{2}\big)$.
Thanks to the decomposition $\widetilde f=\widetilde f^L+\widetilde f^H$, the semi-group estimates \eqref{G2.14}--\eqref{G2.15} in Proposition \ref{prop2.8}, the interpolation \eqref{interpolation}, and the product estimate \eqref{newinterpolation}, we arrive at
\begin{align}\label{G4.23}
\|\widetilde H(t)\|_{L_v^2(\dot B^{s}_{2,\infty})}
\lesssim&\, \int_0^t \|e^{\tau\mathcal{B}}_{L} E\cdot (v\widetilde f,\nabla_v\widetilde f)(t-\tau)\|_{L_v^2(\dot B_{2,\infty}^{s})}{\rm d}\tau\nonumber\\
&+\int_0^t \|e^{\tau\mathcal{B}}_{H} E\cdot (v\widetilde f,\nabla_v\widetilde f)(t-\tau)\|_{L_v^2(\dot B_{2,\infty}^{s})}{\rm d}\tau  \nonumber\\
\lesssim&\, \int_0^t (1+\tau)^{-\frac{3}{4}-\frac{s}{2}} \|   E\cdot (v\widetilde f,\nabla_v\widetilde f)(t-\tau)\|_{L_v^2(\dot B_{2,\infty}^{-\frac{3}{2}})} {\rm d}\tau \nonumber\\
&+\int_0^t e^{-\kappa_2\tau} \|   E\cdot (v\widetilde f,\nabla_v\widetilde f)(t-\tau)\|_{L_v^2(\dot B_{2,\infty}^{1-\varepsilon})} {\rm d}\tau \nonumber\\
\lesssim&\, \sup_{t\geq 0} \|E(t)\|_{\dot B^{-\frac{3}{2}}_{2,\infty}\cap  \dot H^N} \int_0^t(1+\tau)^{-\frac{3}{4}-\frac{s}{2}} \|(v\widetilde f,\nabla_v\widetilde f)(t-\tau)\|_{L_v^2(\dot H^1)}{\rm d}\tau.
\end{align}
Putting \eqref{G4.14.force} and \eqref{G4.14.majorant} into \eqref{G4.23}, and noting that both exponents
$\frac34+\frac s2$ and $\frac14-\frac{s_0}{2}$ lie in $(0,1)$, we find that
\begin{align}\label{G4.24}
 \|\widetilde H(t)\|_{L_v^2(\dot B^{s}_{2,\infty})}
\lesssim&\,     \delta_0 \,\mathcal{I}_0\int_0^t(1+\tau)^{-\frac{3}{4}-\frac{s}{2}}  (1+t-\tau)^{-\frac{1}{4}+\frac{s_0}{2}}{\rm d}\tau \nonumber\\
\lesssim&\,  \delta_0 \,\mathcal{I}_0\,(1+t)^{-\frac{s-s_0}{2}}.
\end{align}

For the remaining cases $\frac12\le s\le 1-\varepsilon$, we split the low-frequency time integral into $\big[0,\frac{t}{2}\big]$ and $\big[\frac{t}{2},t\big]$. The source cannot be evaluated at the single time $t$: on the recent half it comes out as a supremum over $\tau\in[\frac t2,t]$, which does decay by \eqref{G4.14.force}. By \eqref{duality} and \eqref{G2.16}, for any test function $\varPhi \in L_v^2(\mathcal S)$,
\begin{align*}
&\bigg|
\bigg\langle
\int_{\frac{t}{2}}^t e_L^{(t-\tau)\mathcal B}
\Big[
\frac12 E\cdot v\widetilde f
-E\cdot\nabla_v\widetilde f
\Big](\tau) {\rm d}\tau,
\varPhi
\bigg\rangle_{x,v}
\bigg|\\
\lesssim&\,
\int_{\frac{t}{2}}^t
\|E\cdot(v\widetilde f,\nabla_v\widetilde f)(\tau)\|_{L_v^2(\dot B_{2,\infty}^{s-2})}
\|e_L^{(t-\tau)\mathcal B^*}\varPhi\|_{L_v^2(\dot B_{2,1}^{2-s})}
{\rm d}\tau\\
\lesssim&\,
\sup_{\frac{t}{2}\leq\tau\leq t}
\|E\cdot(v\widetilde f,\nabla_v\widetilde f)(\tau)\|_{L_v^2(\dot B_{2,\infty}^{s-2})}
\int_0^\infty
\|e_L^{r\mathcal B^*}\varPhi\|_{L_v^2(\dot B_{2,1}^{2-s})}{\rm d}r\\
\lesssim&\,
\sup_{\frac{t}{2}\leq\tau\leq t}
\|E\cdot(v\widetilde f,\nabla_v\widetilde f)(\tau)\|_{L_v^2(\dot B_{2,\infty}^{s-2})}
\|\varPhi\|_{L_v^2(\dot B_{2,1}^{-s})}.
\end{align*}
Hence
\begin{align}\label{G4.25.recent}
\bigg\|
\int_{\frac{t}{2}}^t e_L^{(t-\tau)\mathcal B}
\Big[
\frac12 E\cdot v\widetilde f
-E\cdot\nabla_v\widetilde f
\Big](\tau)\,{\rm d}\tau
\bigg\|_{L_v^2(\dot B_{2,\infty}^s)}
\lesssim
\sup_{\frac{t}{2}\leq\tau\leq t}
\|E\cdot(v\widetilde f,\nabla_v\widetilde f)(\tau)\|_{L_v^2(\dot B_{2,\infty}^{s-2})}.
\end{align}

For the case $s=\frac12$, the low-frequency decay from $\dot B_{2,\infty}^{-3/2}$ to $\dot B_{2,\infty}^{1/2}$ is of order $(1+t-\tau)^{-1}$. On $[0,t/2]$ one has $t-\tau\ge t/2$, so this factor is $(1+t)^{-1}$. Combining this with \eqref{G4.25.recent}, \eqref{G2.15} and \eqref{G4.14.force}, we obtain
\begin{align}\label{G4.25}
\|\widetilde H(t)\|_{L_v^2(\dot B_{2,\infty}^{\frac12})}
\lesssim&\,
(1+t)^{-1}
\int_0^{\frac{t}{2}}
\|E\cdot(v\widetilde f,\nabla_v\widetilde f)(\tau)\|_{L_v^2(\dot B_{2,\infty}^{-\frac32})}
{\rm d}\tau
\nonumber\\
&+
\sup_{\frac{t}{2}\leq\tau\leq t}
\|E\cdot(v\widetilde f,\nabla_v\widetilde f)(\tau)\|_{L_v^2(\dot B_{2,\infty}^{-\frac32})}
\nonumber\\
&+
\int_0^t e^{-\kappa_2(t-\tau)}
\|E\cdot(v\widetilde f,\nabla_v\widetilde f)(\tau)\|_{L_v^2(\dot B_{2,\infty}^{\frac12})}{\rm d}\tau
\nonumber\\
\lesssim&\,
\delta_0\,\mathcal{I}_0
\Big[
(1+t)^{-1}
\int_0^{\frac{t}{2}}
(1+\tau)^{-\frac14+\frac{s_0}{2}}
{\rm d}\tau
+
\Big(1+\frac{t}{2}\Big)^{-\frac{1-\varepsilon-s_0}{2}}
+
\int_0^t e^{-\kappa_2(t-\tau)}
(1+\tau)^{-\frac{1-\varepsilon-s_0}{2}}{\rm d}\tau
\Big]
\nonumber\\
\lesssim&\,
\delta_0\,\mathcal{I}_0
\Big[
(1+t)^{-1}
\frac{
\big(1+\frac{t}{2}\big)^{\frac34+\frac{s_0}{2}}-1
}{
\frac34+\frac{s_0}{2}
}
+
\Big(1+\frac{t}{2}\Big)^{-\frac{1-\varepsilon-s_0}{2}}
+
e^{-\frac{\kappa_2 t}{2}}
\Big]
\nonumber\\
\lesssim&\,
\delta_0\,\mathcal{I}_0\,(1+t)^{-\frac{\frac12-s_0}{2}},
\end{align}
for $s=\frac12\ge s_0$. The remote piece is of order $(1+t)^{-1+\frac34+\frac{s_0}{2}}=(1+t)^{-\frac12+\frac{s_0}{2}}$, which is exactly the target rate. The recent piece is faster because $\frac{1-\varepsilon-s_0}{2}>\frac14-\frac{s_0}{2}$.

Finally, let $s\in\big(\frac12,1-\varepsilon\big]$. The low-frequency decay from $\dot B_{2,\infty}^{-3/2}$ to $\dot B_{2,\infty}^{s}$ is of order $(1+t-\tau)^{-\frac34-\frac s2}$. The same splitting gives
\begin{align}\label{G4.26}
\|\widetilde H(t)\|_{L_v^2(\dot B_{2,\infty}^{s})}
\lesssim&\,
(1+t)^{-\frac34-\frac s2}
\int_0^{\frac{t}{2}}
\|E\cdot(v\widetilde f,\nabla_v\widetilde f)(\tau)\|_{L_v^2(\dot B_{2,\infty}^{-\frac32})}
{\rm d}\tau
\nonumber\\
&+
\sup_{\frac{t}{2}\leq\tau\leq t}
\|E\cdot(v\widetilde f,\nabla_v\widetilde f)(\tau)\|_{L_v^2(\dot B_{2,\infty}^{s-2})}
\nonumber\\
&+
\int_0^t e^{-\kappa_2(t-\tau)}
\|E\cdot(v\widetilde f,\nabla_v\widetilde f)(\tau)\|_{L_v^2(\dot B_{2,\infty}^{s})}
{\rm d}\tau
\nonumber\\
\lesssim&\,
\delta_0\,\mathcal{I}_0
\Big[
(1+t)^{-\frac34-\frac s2}
\int_0^{\frac{t}{2}}
(1+\tau)^{-\frac14+\frac{s_0}{2}}
{\rm d}\tau
+
\Big(1+\frac{t}{2}\Big)^{-\frac{1-\varepsilon-s_0}{2}}
+
e^{-\frac{\kappa_2 t}{2}}
\Big]
\nonumber\\
\lesssim&\,
\delta_0\,\mathcal{I}_0
\Big[
(1+t)^{-\frac34-\frac s2}
\frac{
\big(1+\frac{t}{2}\big)^{\frac34+\frac{s_0}{2}}-1
}{
\frac34+\frac{s_0}{2}
}
+
\Big(1+\frac{t}{2}\Big)^{-\frac{1-\varepsilon-s_0}{2}}
\Big]
\nonumber\\
\lesssim&\,
\delta_0\,\mathcal{I}_0\,(1+t)^{-\frac{s-s_0}{2}}.
\end{align}
Here, since $\varepsilon<\frac12$, the weight $(1+\tau)^{-\frac{1-\varepsilon-s_0}{2}}$ is majorized by $(1+\tau)^{-\frac14+\frac{s_0}{2}}$. The displayed fraction is the corresponding primitive, and
\[
(1+t)^{-\frac34-\frac s2}\cdot\big(1+t\big)^{\frac34+\frac{s_0}{2}}
=(1+t)^{-\frac{s-s_0}{2}}.
\]
The contribution of $\big[\frac t2,t\big]$ is of order $(1+t)^{-\frac{1-\varepsilon-s_0}{2}}$, which is at least as fast because $s\le 1-\varepsilon$.

By collecting the estimates \eqref{G4.21} with \eqref{G4.24}, \eqref{G4.25} and \eqref{G4.26}, we complete the proof of \eqref{G4.19}.
\end{proof}

\subsection{Proof of asymptotic  stability of Cauchy problem }
In this subsection, with the help of Theorems \ref{T4.6}--\ref{T4.7}, we prove the stability of solutions for the problem \eqref{G1.5}.

\begin{proof}[Proof of Theorem \ref{Th2}]
Combining   \eqref{G4.14} in Theorem \ref{T4.6} with   \eqref{G4.19} in Theorem \ref{T4.7} yields
\begin{align*}
&\sup_{t\geq 0} (1+t)^{\frac{1-\varepsilon-s_0}{2}}  \Big( 
 \|\widetilde f(t)\|_{L_v^2(\dot H^1\cap \dot H^{N-1})}+\|\langle v\rangle\widetilde f(t) \|_{L_v^2(\dot H^1\cap\dot H^{N-2})} +\|\{\mathbf{I}-\mathbf{P}\}\widetilde f(t)\|_{L_{x,v}^2}  \Big)\nonumber\\
& +\sup_{t\geq 0} (1+t)^{\frac{1-\varepsilon-s_0}{2}}   \sum_{\substack{ 1\leq |\beta|\leq N-1 \\|\alpha|+|\beta| \leq N-1}}\|\partial^\alpha_{x}\partial^\beta_{v}\{\mathbf{I}-\mathbf{P}\}\widetilde f(t) \|_{L_{x,v}^2 }  +\sup_{t\geq 0} (1+t)^{\frac{s-s_0}{2}}\|\widetilde f\|_{L_v^2(\dot B_{2,\infty}^s)}  \nonumber\\
\lesssim&\,   \|\widetilde f_0\|_{L_v^2(\dot B_{2,\infty}^{s_0}\cap \dot B_{2,\infty}^{1-\varepsilon})}+\|\widetilde f_0\|_{L_v^2(\dot H^1\cap \dot H^{N-1})}+\|\langle v\rangle\widetilde f_0 \|_{L_v^2(\dot H^1\cap\dot H^{N-2})} +\|\{\mathbf{I}-\mathbf{P}\}\widetilde f_0\|_{L_{x,v}^2} \nonumber\\
&+\sum_{\substack{ 1\leq |\beta|\leq N-1 \\|\alpha|+|\beta| \leq N-1}}\|\partial^\alpha_{x}\partial^\beta_{v}\{\mathbf{I}-\mathbf{P}\}\widetilde f_0\|_{L_{x,v}^2 } \nonumber\\
\lesssim&\,  \|\widetilde f_0\|_{L_v^2(\dot B_{2,\infty}^{s_0}\cap \dot  H^{N-1})}+\|\langle v\rangle \widetilde f_0 \|_{L_v^2(\dot H^1\cap\dot H^{N-2})} +\|\{\mathbf{I}-\mathbf{P}\}\widetilde f_0\|_{L_{x,v}^2} \nonumber\\
&+\sum_{\substack{ 1\leq |\beta|\leq N-1 \\|\alpha|+|\beta| \leq N-1}}\|\partial^\alpha_{x}\partial^\beta_{v}\{\mathbf{I}-\mathbf{P}\}\widetilde f_0\|_{L_{x,v}^2 },
\end{align*}
for any $s\in\big[ -\frac{3}{2}+\varepsilon,1-\varepsilon\big]$,
where we used the smallness of $\delta_0$. This proves \eqref{G1.9}. Thus, we complete the proof of Theorem \ref{Th2}.
\end{proof}

\section{Existence and stability of  time-periodic solutions}
In this section, we shall prove the existence and stability of the time-periodic solution to the problem \eqref{G1.10}. By using Serrin’s method (cf. \cite{Serrin-1959-ARMA,Mp-1991-Nonlinearity}), we present our argument for the construction of time-periodic solutions of the Boltzmann equation.

\begin{proof}[Proof of Theorem \ref{Th1.3}]
Assume that the external force $E(t,x)$ is time-periodic with period $T>0$ and satisfies
\begin{equation}\label{pf.cond.E}
 \|E (t)\|_{\mathcal{C} (\mathbb R; \dot B_{2,\infty}^{-\frac{3}{2}}\cap \dot H^N)  }\leq \delta,
\end{equation}
with $N\geq 4$, where $0<\delta<\delta_0$ is a constant to be chosen later. Let $f_{*}(t,x,v)$ be the global solution of the Cauchy problem \eqref{G1.5} with the initial data $f_{*}(0,x,v)\equiv 0$ in view of Theorem \ref{Th1}. It then follows from Theorem \ref{Th1}  that
\begin{equation*}
f_\ast\in \mathcal{C}\big([0,\infty);L_v^2(\dot B_{2,\infty}^{\frac{1}{2}}\cap \dot H^N)\big),\quad  \langle v\rangle f_\ast\in \mathcal{C}\big([0,\infty);L_v^2( \dot H^1\cap \dot H^{N-1})\big),\quad 
\{\mathbf{I}-\mathbf{P}\}f_\ast\in \mathcal{C}\big([0,\infty);H^N_{x,v}\big),
\end{equation*}
and
\begin{equation} \label{G5.1}
\sup_{t\geq 0}\|f_\ast(t)\|_{\CE^{\frac{1}{2},N}}\leq  C_0\|E (t)\|_{\mathcal{C} (\mathbb R; \dot B_{2,\infty}^{-\frac{3}{2}}\cap \dot H^N)  }.
\end{equation}
Assuming $(1+C_0)\delta\leq \delta_0$, it follows from \eqref{pf.cond.E} and \eqref{G5.1} that
\begin{equation}\label{cond.fE}
\sup_{t\geq 0}\|f_\ast(t)\|_{\CE^{\frac{1}{2},N}}+\|E (t)\|_{\mathcal{C} (\mathbb R; \dot B_{2,\infty}^{-\frac{3}{2}}\cap \dot H^N)  }\leq  \delta_0.
\end{equation}

Let $m\geq k\geq 1$ be integers. From \eqref{cond.fE} and by applying Theorem \ref{Th1}, $f_*(t +(m-k)T,x,v)$ is the global solution of the Cauchy problem \eqref{G1.5} with the initial data $f_\ast((m-k)T,x,v)$. Let $0<\varepsilon<\frac{1}{2}$, then by Theorem \ref{Th2}, it follows that 
\begin{align*}
&\,\|f_*(t +(m-k)T)-f_{*}(t)\|_{L_v^2(\dot B^{1-\varepsilon}_{2,\infty}{\cap \dot H^{N-1}})}\nonumber\\
\lesssim&\,   (1+t)^{-\frac{1}{4}+\frac{\varepsilon}{2}}    \Big(\|f_*((m-k)T)-f_{*}(0)\|_{L_v^2(\dot B_{2,\infty}^{\frac{1}{2}}\cap\dot H^{N-1})}+ \big\|\{\mathbf{I}-\mathbf{P}\}(f_*((m-k)T)-f_{*}(0))\big\|_{L_{x,v}^2}  \nonumber\\
&+\big\|\langle v\rangle (f_*((m-k)T)-f_{*}(0))\big\|_{L_v^2(\dot H^1\cap \dot H^{N-2})}\nonumber\\
&+\sum_{\substack{ 1\leq |\beta|\leq N-1 \\|\alpha|+|\beta| \leq N-1}}\big\|\partial^\alpha_{x}\partial^\beta_{v}\{\mathbf{I}-\mathbf{P}\}(  f_*((m-k)T)-f_{*}(0))\big\|_{L_{x,v}^2 }\Big) \nonumber\\
\lesssim&\,  \delta_0(1+t)^{-\frac{1}{4}+\frac{\varepsilon}{2}} ,
\end{align*}
 for $t\geq 0$.
Taking $t=kT$ in the above inequality, one has
\begin{align}\label{G5.2}
&\,\|f_*(mT)-f_{*}(kT)\|_{L_v^2(\dot B^{1-\varepsilon}_{2,\infty}\cap \dot H^{N-1})}\lesssim \delta_0(1+kT)^{-\frac{1}{4}+\frac{\varepsilon}{2}},    
\end{align}
for all integers $m\geq k\geq 1$.
As $(1+kT)^{-\frac{1}{4}+\frac{\varepsilon}{2}}\rightarrow 0$ as $k\rightarrow \infty$ in \eqref{G5.2},  it shows that  the sequence  $\{f_*(nT)\}_{n\geq 1}$  is a Cauchy sequence in $L_v^2(\dot B_{2,\infty}^{1-\varepsilon}\cap \dot H^{N-1})$, so there is $f^{\infty}_*(x,v)\in L_v^2(\dot B_{2,\infty}^{1-\varepsilon}\cap\dot H^{N-1})$ such that 
\begin{align}\label{G5.4}
    {\lim_{n\rightarrow\infty}\|f_*(nT)-f_*^{\infty}\|_{L_v^2(\dot B_{2,\infty}^{1-\varepsilon}\cap\dot H^{N-1})}=0.}
\end{align}

By \eqref{G5.1}, the sequence $\{f_*(nT)\}_{n\geq 1}$ is bounded in $\CE^{\frac{1}{2},N}$.
Hence, up to a subsequence, it converges weakly in each Hilbert component of $\CE^{\frac{1}{2},N}$
(namely $L_v^2(\dot H^N)$, $L_v^2(\dot H^1\cap\dot H^{N-1})$ for $\langle v\rangle f$,
and the microscopic $H^N_{x,v}$ pieces)
and weak-star in $L_v^2(\dot B_{2,\infty}^{\frac{1}{2}})$.
The strong limit $f_*^{\infty}$ in $L_v^2(\dot B_{2,\infty}^{1-\varepsilon}\cap\dot H^{N-1})$
furnished by \eqref{G5.4} identifies all of these limits,
so the whole sequence, not only a subsequence, converges to the same $f_*^{\infty}$.
Weak lower semicontinuity on the Hilbert pieces, together with the Fatou property
of $\dot B_{2,\infty}^{\frac{1}{2}}\cap\dot B_{2,\infty}^{1-\varepsilon}$
in Proposition~\ref{prop2.5} (here the second index $1-\varepsilon<\frac{3}{2}$
lies in the admissible range), yields
\begin{align} \label{G5.3}
\|f_*^{\infty}\|_{\CE^{\frac{1}{2},N}}\leq   \liminf_{n\to\infty}  \|f_*(nT)\|_{\CE^{\frac{1}{2},N}}\leq C_0\delta.
\end{align}

Therefore, from \eqref{G5.3} and \eqref{pf.cond.E}, the condition \eqref{G1.6} is satisfied with $f_0=f^{\infty}_*$. Applying Theorem \ref{Th1} again, let $f_{T}(t)$ be the global solution of the Cauchy problem \eqref{G1.5} with the initial data $f_{T}(0,x,v)=f^{\infty}_{*}(x,v)$. By Theorem \ref{Th2}, it holds
\begin{equation}
\label{pf.diff}
\|f_{T}(t)-f_{*}(t+(n-1)T)\|_{L_v^2(\dot B_{2,\infty}^{1-\varepsilon}\cap\dot H^{N-1})}\leq C_1\|f^{\infty}_*-{f_\ast}((n-1)T)\|_{L_v^2(\dot B_{2,\infty}^{1-\varepsilon}\cap\dot H^{N-1})},
\end{equation}
for any $t\geq 0$ and $n\geq 1$. By \eqref{pf.diff}, we let $t=T$ and take the infimum in $n$ so as to obtain
\begin{align*} 
&\|f_{T}({T})- f_{T}(0)\|_{L_v^2(\dot B_{2,\infty}^{1-\varepsilon}\cap\dot H^{N-1})}\\
&=\|f_{T}({T})- f^{\infty}_*\|_{L_v^2(\dot B_{2,\infty}^{1-\varepsilon}\cap\dot H^{N-1})}\\
&\leq 
 \liminf_{n\rightarrow\infty} \|f_{T}(T)-f_{*}(nT)\|_{L_v^2(\dot B_{2,\infty}^{1-\varepsilon}\cap\dot H^{N-1})}\nonumber\\ 
&\leq C_1 \liminf_{n\rightarrow\infty} \|f^{\infty}_*-{f_\ast}((n-1)T)\|_{L_v^2(\dot B_{2,\infty}^{1-\varepsilon}\cap\dot H^{N-1})}\nonumber\\
&= C_1 \lim_{n\rightarrow\infty} \|f^{\infty}_*-{f_\ast}((n-1)T)\|_{L_v^2(\dot B_{2,\infty}^{1-\varepsilon}\cap\dot H^{N-1})}\nonumber\\
&= 0,
\end{align*}
where we have used \eqref{G5.4} in the last line. 
Therefore, it holds 
$$
f_{T}({T},x,v)= f_{T}(0,x,v), \ \forall\,(x,v)\in \R^3\times \R^3,
$$
so $f_T(t,x,v)$ is the unique time-periodic solution with the same period $T>0$ for the problem \eqref{G1.10} and further \eqref{fT.bdd} follows.

Finally, we prove \eqref{G1.15}.
Set
\begin{align*}
 s_0=-3\Big(\frac1p-\frac12\Big),
\qquad
\bar s=\max\Big\{s_0,-\frac32+\varepsilon\Big\}.   
\end{align*} 
Since $s>s_0$ and
$-\frac32+\varepsilon<s<1-\varepsilon$,
we have $\bar s<s<1-\varepsilon$.
Choose $\zeta'\in(0,1)$ such that
\begin{align*}
\bar s\,\zeta'
+(1-\varepsilon)(1-\zeta')=s.    
\end{align*}
By \eqref{interpolation} and Theorem \ref{Th2},
we obtain that
\begin{align*}
\|(f-f_T)(t)\|_{L_v^2(\dot H^s)}
\lesssim&\,
\|(f-f_T)(t)\|_{L_v^2(\dot B^s_{2,1})}
\\
\lesssim&\,
\|(f-f_T)(t)\|_
{L_v^2(\dot B^{\bar s}_{2,\infty})}^{\zeta'}
\|(f-f_T)(t)\|_
{L_v^2(\dot B^{1-\varepsilon}_{2,\infty})}^{1-\zeta'}
\\
\lesssim&\,
(1+t)^{-\frac{s-s_0}{2}}
\|(f-f_T)(0)\|_{\mathcal E^{s_0,N-1}}.
\end{align*}
Using
$L^p\hookrightarrow\dot B^{s_0}_{2,\infty}$
and the definition \eqref{norm.E},
we arrive at \eqref{G1.15}.
Thus, the proof of Theorem \ref{Th1.3} is completed.
\end{proof}

\noindent{\bf Acknowledgments:} 
The research of Renjun Duan was partially supported by the General Research Fund (Project No.~14303523) from RGC of Hong Kong and also by the grant from the National Natural Science Foundation of China (Project No.~12425109). The first author would like to thank Professor Yoshiyuki Kagei for many stimulating communications on the relevant time-periodic problem over the past years.

\medskip
\noindent\textbf{Conflict of interest.} The authors do not have any possible conflicts of interest.

\medskip
\noindent\textbf{Data availability statement.}
 Data sharing is not applicable to this article as no data sets were generated or analyzed during the current study.

\bibliographystyle{plain}

\end{document}